\documentclass[11pt]{amsart}

\usepackage{amscd, amssymb}
\usepackage[all]{xypic}
\usepackage{multirow}
\usepackage{hhline}
\usepackage{verbatim}
\usepackage{mathdots}
\usepackage{comment}
\usepackage{tikz-cd}
\usepackage[titletoc]{appendix}
\usepackage{stmaryrd}
\usepackage{graphicx}
\usepackage{mathtools}
\usepackage{leftidx}
\usepackage{mathabx }
\usepackage{amsfonts}
\usepackage{mathrsfs}

\usepackage[shortlabels]{enumitem}

\usepackage{dynkin-diagrams}
\usepackage{xcolor}

\usepackage[colorlinks=true, linkcolor=blue,
 citecolor=blue, urlcolor=blue]{hyperref}

\theoremstyle{plain}
\newtheorem{Thm}[equation]{Theorem}
\newtheorem*{Thm*}{Theorem}
\newtheorem{Cor}[equation]{Corollary}
\newtheorem{Prop}[equation]{Proposition}
\newtheorem{Lem}[equation]{Lemma}
\newtheorem{Construction}[equation]{Construction}

\theoremstyle{definition}
\newtheorem*{Rmk}{Remark}
\newtheorem{Def}[equation]{Definition}

\numberwithin{equation}{section}

\newcommand{\C}{\mathbb{C}}
\newcommand{\F}{\mathbb{F}}

\newcommand{\Q}{\mathbb{Q}}
\newcommand{\R}{\mathbb{R}}
\newcommand{\Z}{\mathbb{Z}}
\newcommand{\GL}{\operatorname{GL}}
\newcommand{\SL}{\operatorname{SL}}

\newcommand{\pr}{\operatorname{pr}}

\newcommand{\xt}{\tilde{x}}
\newcommand{\yt}{\tilde{y}}

\newcommand{\wt}{\tilde{w}}
\newcommand{\htt}{\tilde{h}}

\newcommand{\xa}{x_{\alpha}}
\newcommand{\wa}{w_{\alpha}}
\newcommand{\ha}{h_{\alpha}}
\newcommand{\xb}{x_{\beta}}
\newcommand{\wb}{w_{\beta}}
\newcommand{\hb}{h_{\beta}}

\newcommand{\xta}{\tilde{x}_{\alpha}}
\newcommand{\wta}{\tilde{w}_{\alpha}}
\newcommand{\hta}{\tilde{h}_{\alpha}}
\newcommand{\xtb}{\tilde{x}_{\beta}}
\newcommand{\wtb}{\tilde{w}_{\beta}}
\newcommand{\htb}{\tilde{h}_{\beta}}

\newcommand{\alphac}{{\check{\alpha}}}
\newcommand{\betac}{{\check{\beta}}}

\newcommand{\Bt}{\widebar{B}}

\newcommand{\Pt}{\widetilde{P}}

\newcommand{\Gt}{\widebar{G}}
\newcommand{\Ht}{\widebar{H}}
\newcommand{\It}{\widebar{I}}

\newcommand{\Kt}{\widebar{K}}
\newcommand{\Tt}{\widebar{T}}
\newcommand{\Xt}{X_{Q,n}}

\newcommand{\Yt}{Y_{Q,n}}
\newcommand{\Wt}{\widetilde{W}}
\newcommand{\Omegat}{\widetilde{\Omega}}

\newcommand{\Gbbt}{\widebar{{\Gbb}}}
\newcommand{\Tbbt}{\widebar{{\Tbb}}}

\newcommand{\Gammat}{\widebar{\Gamma}}
\newcommand{\Gammaft}{\widebar{\mathsf{\Gamma}}}

\newcommand{\Gbb}{{\mathbb{G}}}
\newcommand{\Tbb}{{\mathbb{T}}}
\newcommand{\Bbbb}{{\mathbb{B}}}

\newcommand{\gt}{\tilde{g}}

\newcommand{\kt}{\tilde{k}}

\newcommand{\ttt}{\tilde{t}}

\newcommand{\pit}{\tilde{\pi}}

\newcommand{\Ind}{\operatorname{Ind}}

\newcommand{\ind}{\operatorname{ind}}
\newcommand{\Irr}{\operatorname{Irr}}

\newcommand{\supp}{\operatorname{supp}}

\newcommand{\End}{\operatorname{End}}
\newcommand{\Hom}{\operatorname{Hom}}
\newcommand{\Aut}{\operatorname{Aut}}

\newcommand{\Int}{\operatorname{Int}}

\newcommand{\Tr}{\operatorname{Tr}}

\newcommand{\aff}{\mathrm{aff}}
\newcommand{\af}{\mathrm{af}}
\newcommand{\ea}{\mathrm{ea}}

\newcommand{\s}{\mathbf{s}}

\newcommand{\Ocal}{\mathcal{O}}

\newcommand{\Wcal}{\mathcal{W}}
\newcommand{\Hcal}{\mathcal{H}}

\newcommand{\Bf}{\mathsf{B}}

\newcommand{\Hf}{\mathsf{H}}
\newcommand{\If}{\mathsf{I}}
\newcommand{\Kf}{\mathsf{K}}

\newcommand{\Nf}{\mathsf{N}}
\newcommand{\Tf}{\mathsf{T}}
\newcommand{\Uf}{\mathsf{U}}

\newcommand{\Gammaf}{\mathsf{\Gamma}}

\newcommand{\Hft}{\widebar{\mathsf{H}}}
\newcommand{\Ift}{\widebar{\mathsf{I}}}
\newcommand{\Kft}{\widebar{\mathsf{K}}}

\newcommand{\Bft}{\widebar{\mathsf{B}}}
\newcommand{\Pft}{\widebar{\mathsf{P}}}
\newcommand{\Tft}{\widebar{\mathsf{T}}}

\newcommand{\la}{\langle}
\newcommand{\ra}{\rangle}

\newcommand{\ie}{{\it i.e.}\ }

\newcommand{\st}{\,:\,}
\newcommand{\qand}{{\quad\text{and}\quad}}

\newcommand{\one}{\mathbf{1}}

\title[$K$-types for wild double covers]{Minimal depth $K$-types for wild double covers
and Shimura correspondences}

\author{Edmund Karasiewicz}

\author{Shuichiro Takeda}

\begin{document}

\begin{abstract}

	Using distinguished pseudo-spherical representations of covering tori constructed in a companion work and Dat's parahoric induction, we construct a distinguished family of representations for an Iwahori subgroup $\It$ in $\Gt$, an $n$-fold Brylinski-Deligne (BD) cover of a split connected reductive $p$-adic group $G$, under explicit hypotheses on the BD datum. No tameness assumption $\gcd(p,n)=1$ is imposed.

    The construction suggests a conjectural theory of Iwahori-types for BD covers beyond the tame case. We verify this for $2$-fold covers of almost simple simply-laced simply-connected Chevalley groups of rank at least $2$ over $2$-adic fields. In particular, for these non-tame $2$-fold covers we prove:
    
    1) The pair $(\It,\sigma)$ is a Bushnell-Kutzko type for the Bernstein block analogous to the block of unramified principal series in a linear reductive group. This is a non-tame analogue of the Borel-Casselman theorem.

    2) The Hecke algebra $\Hcal(\Gt,\It,\sigma^{\vee})$ has an Iwahori-Matsumoto presentation up to normalization. In types $A_{r},D_{2r+1},E_{6},E_{7}$ we obtain the complete presentation and show that $\Hcal(\Gt,\It,\sigma^{\vee})$ is isomorphic to the Iwahori-Hecke algebra of an explicit linear group. This determines an equivalence of Bernstein blocks producing a local Shimura correspondence.

    3) The representation $\sigma$ extends to a hyperspecial maximal compact subgroup $\Kt$. Among the genuine $\widebar{K}$-representations, $\sigma$ has minimal depth with respect to the principal congruence filtration. It also yields a finite Shimura correspondence, generalizing a result of Savin.
\end{abstract}

\maketitle

\setcounter{tocdepth}{1}
\tableofcontents


\section{Introduction}


\subsection{Motivations and objectives} Let $\Gbb$ be a split connected reductive group over a number field $F$. Let $\mathbb{A}$ be the adele ring of $F$. A fundamental result in the theory of automorphic representations is that any automorphic representation $\pi$ on $\Gbb$ can be factored into a product of $\Gbb(F_{v})$-representations $\pi_{v}$, where $v$ is a place of $F$ and $F_{v}$ is the completion of $F$ at $v$. Moreover for almost all places $v$ the representation $\pi_{v}$ is spherical, meaning that it contains a nonzero vector fixed by a special maximal compact subgroup.

Thus 1) the study of automorphic representations requires understanding the representation theory of $\Gbb(F_{v})$ for all places $v$ and 2) the spherical representations occupy a distinguished role among the representations of $\Gbb(F_{v})$.

There are interesting automorphic representations that live on certain central extensions of reductive groups. For example, the Jacobi theta function defines an automorphic representation that lives on a double cover of $\mathrm{SL}_{2}$ that is not a matrix group. To accommodate these examples, an effort is underway to develop the representation theory of certain central extensions of reductive groups over local and global fields. For this purpose a good family of covers is provided by Brylinski-Deligne (BD) \cite{BD01}. 

Let $\Gbbt$ be an $n$-fold BD cover of $\Gbb$ over the global field $F$. Automorphic representations on $\Gbbt$ can also be factored into representations of the local groups $\Gbbt(F_{v})$ almost all of which are spherical. So again 1) it is important to understanding the representation theory of $\Gbbt(F_{v})$ for all places $v$ and 2) the spherical representations play a distinguished role.

However, our current understanding of the representation theory of $\Gbbt(F_{v})$ is highly dependent on whether the residual characteristic of the place $v$ divides $n$ or not, in which case we call the cover tame and wild, respectively. In the tame case, the representation theory of $\Gbbt(F_{v})$ is more developed, and in particular we have a notion of spherical representations; in the wild case, much less is known and in particular there is no general notion of spherical representations outside of a few special cases \cite{LS10a, S12, Wood,Takeda_Wood, Karasiewicz}. 

Even though there are only finitely many places with residual characteristic dividing $n$, it will still be valuable to have a theory of spherical representations at all places. For example, the Kohnen plus space for classical modular forms of half-integral weight \cite{Koh80} was described in \cite{LS10a} using a $\mathsf{K}$-type in the local representation theory of the wild double cover of $\mathrm{SL}_{2}(\Q_{2})$. Loke-Savin call this a pseudo-spherical type and the construction we describe later generalizes it. 

The present work has two objectives. 

First,  we will formulate a notion of spherical representations for a large family of BD covers of split connected reductive groups over $p$-adic fields beyond the tame case. We will call these representations (distinguished) pseudo-spherical representations, following \cite{ABPTV,W09,LS10a}.

Second, we will develop the theory of these pseudo-spherical representations for a family of double covers over $2$-adic fields.

\subsection{Spherical representations for tame covers} From now on we focus on local representation theory. Let $F$ be a finite extension of $\Q_{p}$. Let $\Gbb$ be a split connected reductive group over $F$. We write $G=\Gbb(F)$ for its group of $F$ points and let $K\subseteq G$ be a hyperspecial maximal compact subgroup. An irreducible smooth $G$-representation $\pi$ is said to be $K$-spherical if $\pi$ restricted to $K$ contains $\one_{K}$, where for any group $H$ we write $\one_{H}$ for the trivial representation of $H$. We note that in addition to their role in the theory of automorphic representations, the spherical representations of $G$ provide an accessible yet nontrivial family of representations; they all live as constituents of the unramified principal series and can be classified using the Satake isomorphism \cite{Satake63}. 

Now we turn to covering groups. Let $\Gbbt$ be a BD cover of $\Gbb$. Let $n\in\Z_{\geq 1}$ and let $\mu_{n}\subseteq F$ be the $n$-th roots of unity in the $p$-adic field $F$. If $\mu_{n}$ is a cyclic group of order $n$, then from $\Gbbt$ we can get a group $\Gt$ that fits into a topological central extension
\begin{equation}\label{Eq:CExt_intro}
    1\longrightarrow \mu_{n}\longrightarrow \Gt \xrightarrow{\;\operatorname{pr}\;} G\longrightarrow 1.
\end{equation}
For organizing the representation theory of $\Gt$ it is convenient to fix $\epsilon:\mu_{n}\rightarrow \C^{\times}$ a faithful character and focus on the $\epsilon$-genuine $\Gt$-representations, meaning $\mu_{n}$ acts by $\epsilon$.

When the cover $\Gt$ is tame the theory of spherical representations of $\Gt$ is well-understood under some assumptions on $\Gbbt$ (e.g., assuming \ref{Hyp_eta_int} from the next subsection, or more generally working in the setting of \cite{W16}). The main consequence of the tameness assumption and \ref{Hyp_eta_int} is that the extension \eqref{Eq:CExt_intro} splits over $K\subseteq G$ \cite[Theorem 4.2]{GG18}. So, there exists a group homomorphism $\mathbf{s}_{K}:K\rightarrow \Gt$ such that $\mathrm{pr}\circ\mathbf{s}_{K}=id_{K}$.

With this splitting, we say that an $\epsilon$-genuine irreducible $\Gt$-representation $\pi$ is $\mathbf{s}_{K}$-spherical if $\pi$ restricted to $\Kt:=\mathrm{pr}^{-1}(K)$ contains the representation $\epsilon\otimes \one_{K}$, where $\mu_{n}\times K\cong \Kt$ via the isomorphism
\begin{equation}\label{Eq:Intro_Split_Map_K}
    (\zeta,k)\mapsto \zeta\mathbf{s}_{K}(k).
\end{equation}
Again the $\mathbf{s}_{K}$-spherical representations can all be found as constituents of unramified principal series (these principal series will be defined precisely later) and can be classified using a Satake isomorphism \cite{Mc12}.

\subsection{Pseudo-spherical representations}
For wild covers (i.e., $\gcd(p,n)\neq 1$), the extension \eqref{Eq:CExt_intro} need not split over $K$ and thus it is not clear how to formulate a notion of spherical $\Gt$-representations. Instead, we will formulate a notion of {\it pseudo-spherical representations}.

We begin with some notation. Recall that $\Gbb$ is a split connected reductive group over $F$. Let $\Tbb\subseteq \Gbb$ be a maximal split torus with character lattice $X$ and cocharacter lattice $Y$. Let $\Phi\subseteq X$ be the set of roots of the pair $(\Gbb,\Tbb)$. Let $Y^{sc}\subseteq Y$ be the sublattice generated by the coroots $\Phi^{\vee}$. For each $\alpha\in \Phi$, let $\mathbb{U}_{\alpha}$ be the root subgroup of $\alpha$.

Let $\Bbbb=\Bbbb^{+}\subseteq \Gbb$ be a Borel subgroup containing $\Tbb$ with unipotent radical $\mathbb{U}$ and opposite unipotent radical $\mathbb{U}^{-}$.

We fix a Chevalley pinning with respect to $(\Gbb,\Bbbb,\Tbb)$ and thus for each $\alpha\in\Phi$ we have an isomorphism of algebraic groups
\begin{equation*}
    x_{\alpha}:\Gbb_{a}\rightarrow \mathbb{U}_{\alpha}.
\end{equation*}

Let $\Gbbt$ be a BD cover of $\Gbb$ incarnated by $(D,\eta)$, where $D\in X\otimes X$ is a fair bisector (see \cite[Definitions 2.2 and 2.11]{W14}) and $\eta:Y^{sc}\rightarrow F^{\times}$ is a group homomorphism. The cover $\Gbbt$ fits into a central extension in the category of sheaves of groups on the big Zariski site on $\mathrm{Spec}(F)$
\begin{equation}\label{Eq:CEXT_Intro_Gen}
    \mathbb{K}_{2}\hookrightarrow \Gbbt\twoheadrightarrow\Gbb,
\end{equation}
where $\mathbb{K}_{2}$ is Quillen's $K_{2}$. For any algebraic subgroup $\mathbb{H}\subseteq \Gbb$ we write $\widebar{\mathbb{H}}$ for the $\mathbb{K}_{2}$-extension of $\mathbb{H}$ obtained by pulling back \eqref{Eq:CEXT_Intro_Gen} by the inclusion $\mathbb{H}\subseteq \Gbb$.

The bisector $D$ connects with the structure of $\Gbbt$ through a section
\begin{equation}\label{Eqn:T_Sec_Intro}
    \mathbf{s}:\Tbb\rightarrow \Tbbt,
\end{equation}
such that  $\mathbf{s}(id_{\Tbb})=id_{\Tbbt}$, and for all $y\in Y$, $u\in \mathbb{G}_{m}$, and $t\in\Tbb$
\begin{equation}\label{Eq:TSectionMult}
    \mathbf{s}(y(u))\cdot \mathbf{s}(t) = \{u,t^{D(y,-)}\}\;\mathbf{s}(y(u)\cdot t).
\end{equation}
(See \cite[Section 0.N.5]{BD01} for $\{-,-\}$ notation.)

Recall that $n\in\Z_{\geq 1}$ and that $\mu_{n}\subseteq F$, the subgroup of $n$-th roots of unity in $F$, has order $n$. By taking $F$-points and pushing out by the $n$-th order Hilbert symbol we get a topological central extension of groups
\begin{equation}\label{Eq:CEXT_Intro_Gen_F}
    1\longrightarrow \mu_{n}\longrightarrow \Gt \xrightarrow{\;\operatorname{pr}\;} G\longrightarrow 1,
\end{equation}
where $G=\Gbb(F)$ is the group of $F$-points of $\Gbb$. For any algebraic subgroup $\mathbb{H}\subseteq \Gbb$ we write $H$ for its group of $F$-points, and for any subset $S\subseteq G$ we write $\widebar{S}:=\mathrm{pr}^{-1}(S)$. 

From \cite[Proposition 3.2 and subsequent comment]{BD01} we know that the central extension \eqref{Eq:CEXT_Intro_Gen} splits uniquely over the unipotent subgroups $\mathbb{U}^{\pm}$. This induces a splitting of the extension \eqref{Eq:CEXT_Intro_Gen_F} over the unipotent subgroups $U^{\pm}$. Using this canonical splitting, we can identify $U^{\pm}$ with subgroups of $\Gt$.  

Using the Chevalley pinning with respect to $(\Gbb,\Bbbb,\Tbb)$ we can define $K\subseteq G$ the hyperspecial maximal compact subgroup generated by the subgroups $x_{\alpha}(\Ocal)$ and $T_{0}$, where $\Ocal\subseteq F$ is the ring of integers and $T_{0}$ is the unique maximal compact subgroup of $T$. Let $T_{0}\subseteq I\subseteq K$ be the Iwahori subgroup aligned with $\Bbbb$.

Our next task will be to construct a family of genuine $\It$-representations using the distinguished pseudo-spherical $\Tt_{0}$-representations from \cite{KT26}. For this we will need to review some material from \cite{KT26}.

There is a subgroup $T_{R}\subseteq T_{0}$ built from $R$ the integral radical of the $n$-th order Hilbert symbol of $F$. If we identify $T\cong Y\otimes F^{\times}$, then $T_{0}\cong Y\otimes \Ocal^{\times}$ and $T_{R}\cong Y\otimes R$. The section $\mathbf{s}$ from \eqref{Eqn:T_Sec_Intro} defines a splitting $\mathbf{s}_{R}:T_{R}\rightarrow \Tt_{0}$ of the central extension \eqref{Eq:CEXT_Intro_Gen_F} over $T_{R}$ \cite[Proposition 3.1]{KT26}. An $\mathbf{s}_{R}$-pseudo-spherical $\Tt_{0}$-representation is a genuine irreducible $\Tt_{0}$-representation such that $\mathbf{s}_{R}(T_{R})$ acts trivially \cite[Definition 3.2]{KT26}. We will just call these representations pseudo-spherical since we only use the splitting $\mathbf{s}_{R}$.

The subset of distinguished pseudo-spherical $\Tt_{0}$-representations has a more involved definition, involving the pair $(\Gbb,\Tbb)$. For details see \cite[Definition 3.3]{KT26}.

Let $\eta_{n}$ be the composition of $\eta$ followed by the natural quotient map $F^{\times}\rightarrow F^{\times}/F^{\times n}$. From now on we assume that $\eta$ has the following two properties.

\begin{enumerate}[label=\textbf{(H\arabic*)}, ref=(H\arabic*)]
    \item\label{Hyp_eta_ext} $\eta_{n}$ can be extended to a homomorphism of $Y$. 
    \item\label{Hyp_eta_int} $\eta$ is valued in $\Ocal^{\times}F^{\times n}$.
\end{enumerate}
We note that the two hypotheses are satisfied whenever $\eta$ is trivial. This occurs, for example, for any Steinberg cover of a simply-connected Chevalley group.

Now we can construct the $\It$-representations. By \ref{Hyp_eta_ext} and \ref{Hyp_eta_int}, $Z(\Tt)$ admits an $\epsilon$-genuine distinguished pseudo-spherical character $\chi_{dist}$ \cite[Proposition 5.10]{KT26}. Let $\tau=\tau_{\chi_{dist}}$ be the unique irreducible genuine $\Tt$-representation with central character $\chi_{dist}$. Let $\tau_{0}$ be the unique distinguished pseudo-spherical $\Tt_{0}$-constituent of $\tau$ \cite[Corollary 5.12]{KT26}. Let $U_{0}:=U\cap I$ and $U^{-}_{1}:=U^{-}\cap I$. Multiplication defines a homeomorphism $U^{-}_{1}\times \Tt_{0}\times U_{0}\rightarrow \It$, called an Iwahori factorization of $\It$. Now given the $\Tt_{0}$-representation $\tau_{0}$ and the Iwahori factorization of $\It$ associated to the triple $(U^{-}_{1},\Tt_{0},U_{0})$ we can construct an $\epsilon$-genuine irreducible $\It$-representation $\sigma=\sigma_{\chi_{dist}}:=i_{U^{-}_{1},U_{0}}\tau_{0}$ where $i_{U^{-}_{1},U_{0}}$ is Dat's parahoric induction functor \cite{D09}, reviewed in Subsection \ref{SS:ParaInd}. Thus we have the following construction.

\begin{Construction}\label{Construction:PS_I-rep}
    Assuming \ref{Hyp_eta_ext} and \ref{Hyp_eta_int}, for each $\chi_{dist}$ a distinguished pseudo-spherical character of $Z(\Tt)$, we have constructed a genuine irreducible $\It$-representation $\sigma=\sigma_{\chi_{dist}}$.
\end{Construction}
We will call $\sigma$ a distinguished pseudo-spherical representation of $\It$. If $\Gt=G$, then $\sigma$ is just the trivial representation of $I$.

 We believe that the distinguished pseudo-spherical representation of $\It$ will be fundamental to the representation theory of $\Gt$. Roughly, the pair $(\It,\sigma)$ should play a role similar to the trivial Iwahori type for a $p$-adic reductive group. More specifically, we expect that the representation $\sigma$ satisfies the following properties.

 \begin{enumerate}[label=\textbf{P\arabic*)}, ref=P\arabic*)]
    \item\label{Exp:BC_Thm} We expect that the pair $(\It,\sigma)$ is a type for the inertial equivalence class $[\Tt,\tau_{\chi_{dist}}]_{\Gt}$, in the sense of Bushnell-Kutzko \cite{BK98}. This inertial equivalence class is analogous to the block of unramified principal series for a linear connected reductive group. Thus this is analogous to the Borel-Casselman Theorem \cite{B76}.

    \item\label{Exp:Heck_Iso} We expect that the Hecke algebra attached to $(\It,\sigma)$ is isomorphic to the Iwahori-Hecke algebra of a linear split connected reductive group $G^{\prime}$. Specifically, the group $G^{\prime}$ should be the unique such group so that the Langlands dual group of $\Gt$ and $G^{\prime}$ are isomorphic. Such an isomorphism defines an equivalence between Bernstein blocks setting up a local Shimura correspondence.
 
    \item\label{Exp:Sig_Ext} We expect that the $\It$-representation $\sigma$ extends to a representation of $\Kt$. Furthermore, we expect this extension has minimal depth with respect to the principal congruence filtration among all genuine $\Kt$-representations.
 \end{enumerate}

\ref{Exp:BC_Thm} and \ref{Exp:Heck_Iso} were proved for a large family of tame covers by Savin \cite{S04}. This was extended to tame BD covers in \cite{GG18}. For wild covers, only a few special cases for double covers have been worked out \cite{LS10a, Wood, Takeda_Wood, Karasiewicz}. 

\ref{Exp:Sig_Ext} for tame covers is straightforward. In this case, $R=\Ocal^{\times}$, so $T_{R}=T_{0}$, and a pseudo-spherical representation of $\Tt_{0}$ is just the character $\epsilon\otimes \one_{\mathbf{s}_{\Ocal^{\times}}(T_{0})}$ of $\Tt_{0}\cong \mu_{n}\times \mathbf{s}_{\Ocal^{\times}}(T_{0})$, where this isomorphism is induced by the splitting $\mathbf{s}_{\Ocal^{\times}}=\mathbf{s}_{R}$. Under the hypothesis \ref{Hyp_eta_int}, the central extension \eqref{Eq:CEXT_Intro_Gen_F} splits over $K$ and there is a splitting $\mathbf{s}_{K}$ of $K$ so that $\mathbf{s}_{K}|_{\Tt_{0}}=\mathbf{s}_{\Ocal^{\times}}$ \cite[Proposition 4.1 and Theorem 4.2]{GG18}. From here it is not difficult to complete \ref{Exp:Sig_Ext} by showing that $\sigma$ is the restriction of $\epsilon\otimes \one_{\mathbf{s}_{K}(K)}$ to $\It$. For wild covers, \ref{Exp:Sig_Ext} was proved for certain double covers over unramified finite extensions of $\Q_{2}$ in \cite{S12}.

Our main theorems in this paper prove properties \ref{Exp:BC_Thm}, \ref{Exp:Heck_Iso}, \ref{Exp:Sig_Ext} for a family of double covers over $2$-adic fields. We discuss these results in the next subsection.

We end this subsection with the definition of pseudo-spherical representations of $\Gt$ and some comments on a Satake isomorphism for a pseudo-spherical Hecke algebra.

  \begin{Def}
    Let $\pi$ be an irreducible smooth $\Gt$-representation. We say that $\pi$ is $(\It,\sigma)$-pseudo-spherical if the restriction of $\pi$ to $\It$ contains $\sigma$.
 
    Now assuming that the $\It$-representation $\sigma$ extends to $\Kt$, we say that $\pi$ is $(\Kt,\sigma)$-pseudo-spherical if the restriction of $\pi$ to $\Kt$ contains $\sigma$.
 \end{Def}
\quad 
 Let $\chi:T\rightarrow \C^{\times}$ be an unramified character of $T$. Recall that the unramified principal series $\mathrm{Ind}_{B}^{G}(\chi)$ contain the trivial $K$-representation with multiplicity one. Furthermore, each spherical representation of $G$ is the unique irreducible constituent of $\mathrm{Ind}_{B}^{G}(\chi)$ containing the trivial $K$-representation for some unramified character $\chi$.
 
 For $\Gt$ the analogue of this multiplicity one property is satisfied by the $\Kt$-representation $\sigma$. The unramified principal series on $\Gt$ are constructed as follows. Again let $\chi:T\rightarrow \C^{\times}$ be an unramified character of $T$ and let $I(\chi):=\mathrm{Ind}_{\Bt}^{\Gt}(\chi\otimes \tau_{\chi_{dist}})$ be the normalized principal series representation. The next proposition is proved in Appendix \ref{S:Appendix_PS_Mult_One}.

\begin{Prop}\label{Prop:PS_Multi_one_Intro}
    Assume that the $\It$-representation $\sigma$ extends to $\Kt$. Let $\chi:T\rightarrow \C^{\times}$ be an unramified character. The multiplicity space $\mathrm{Hom}_{\Kt}(\sigma,I(\chi))$ has dimension $1$. Thus $I(\chi)$ contains a unique $(\Kt,\sigma)$-pseudo-spherical constituent.
\end{Prop}

In the case of a linear group $G$, the classification of the spherical representations is intimately connected with the study of the spherical Hecke algebra $\Hcal(G,K)$ of smooth compactly supported $K$-biinvariant functions on $G$, and in particular follows from the Satake isomorphism which shows that $\Hcal(G,K)\cong\C[Y]^{W}$ as $\C$-algebras. 

For $\Gt$, we expect some sort of Satake isomorphism to control the pseudo-spherical representations. Assuming \ref{Exp:Sig_Ext}, we can define a pseudo-spherical Hecke algebra $\Hcal(\Gt,\Kt,\sigma^{\vee})$. Then \ref{Exp:Heck_Iso} essentially implies there is a linear group $G^{\prime}$ with a hyperspecial maximal compact subgroup $K^{\prime}$ such that $\Hcal(\Gt,\Kt,\sigma^{\vee})\cong \Hcal(G^{\prime},K^{\prime})$. Using the Satake isomorphism for $G^{\prime}$ we get $\Hcal(\Gt,\Kt,\sigma^{\vee})\cong \C[Y^{\prime}]^{W^{\prime}}$, where $Y^{\prime}$ is the cocharacter lattice of a maximal split torus $T^{\prime} \subseteq G^{\prime}$ and $W^{\prime}$ is the Weyl group of $(G^{\prime},T^{\prime})$.

Thus assuming \ref{Exp:Heck_Iso} and \ref{Exp:Sig_Ext} we get a Satake isomorphism $\Hcal(\Gt,\Kt,\sigma^{\vee})\cong \C[Y^{\prime}]^{W^{\prime}}$ for the pseudo-spherical Hecke algebra $\Hcal(\Gt,\Kt,\sigma^{\vee})$. We work this out completely for some double covers in Subsection \ref{SS:Satake}.

\subsection{Main results}

Our main theorems prove properties \ref{Exp:BC_Thm}, \ref{Exp:Heck_Iso}, \ref{Exp:Sig_Ext} for certain double covers over $2$-adic fields. In this subsection we will elaborate on these results and their proofs. Specifically, we take $\Gbb$ to be an almost simple simply-laced simply-connected Chevalley group of rank at least $2$ over an \textit{arbitrary} $2$-adic field $F$. We retain the other notation from the previous subsections. We note that for these covers all pseudo-spherical representations of $\Tt_{0}$ are distinguished.

By the work of Steinberg, Moore, and Matsumoto, $G$ has an (essentially) unique nontrivial topological $2$-fold central extension $\Gt$, fitting into an exact sequence
\begin{equation}\label{CExtIntro}
     1\longrightarrow\mu_2\longrightarrow\Gt\stackrel{\pr}{\longrightarrow} G\longrightarrow 1.
\end{equation}
We note that $\Gt$ can be constructed from a BD cover $\Gbbt$ incarnated by $(D,\eta)$, where $\eta$ is trivial. Thus in this setting \ref{Hyp_eta_ext} and \ref{Hyp_eta_int} are satisfied.

Now we can state our main theorem on the analogue of the Borel-Casselman Theorem for $\Gt$, which proves \ref{Exp:BC_Thm} for these covers.

\begin{Thm}[Theorem \ref{IType}]
Let $\Gt$ be the double cover of an almost simple simply-laced simply-connected Chevalley group of rank at least $2$ over a $2$-adic field $F$. Let $\tau_{0}$ be a pseudo-spherical $\Tt_{0}$-representation and let $\tau$ be a distinguished genuine $\Tt$-representation containing $\tau_{0}$. Let $\sigma=\sigma_{\tau_{0}}$ be the $\It$-representation parahorically induced from $\tau_{0}$.

The pair $(\It, \sigma)$ is a type in the sense of Bushnell-Kutzko \cite[(4.1) Definition]{BK98} for the inertial class $[\Tt, \tau]_{\Gt}$. Thus, the category of right $\Hcal(\Gt, \It; \sigma^{\vee})$-modules is equivalent to the Bernstein block of $[\Tt, \tau]_{\Gt}$, and no other Bernstein block contains representations with $(\It,\sigma)$-isotypic vectors.
\end{Thm}

\begin{Rmk}
   The assumption that the rank is at least $2$ only excludes the double cover of $\mathrm{SL}_{2}(F)$, where the result is already known from \cite{Takeda_Wood}.
\end{Rmk}

With this type theorem, it is desirable to know the structure of the Hecke algebra $\Hcal(\Gt, \It; \sigma^{\vee})$. As in the linear case, we expect that $\Hcal(\Gt, \It; \sigma^{\vee})$ admits an Iwahori-Matsumoto (IM) presentation \cite{IM}, which describes the algebra as a deformation of the group algebra of an extended affine Weyl group. We  prove the IM presentation up to normalization in all (irreducible) Cartan types, obtaining the complete presentation in types $A_r,D_{2r+1},E_{6},E_{7}$ (Theorem \ref{IMPres}). In type $D_{2r}$ we obtain an IM presentation only up to a $2$-cocycle on the length $0$ elements. In type $E_{8}$ the remaining issue is the normalization of the generator corresponding to the simple affine reflection.

One of the consequences of the IM presentation is that there is an explicit linear Chevalley group $G^{\prime}$ and Iwahori subgroup $I^{\prime}\subseteq G^{\prime}$ such that $\Hcal(G', I')\cong\Hcal(\Gt, \It; \sigma^{\vee})$. This Hecke algebra isomorphism gives a Shimura correspondence of representations.

\begin{Thm}[Corollary \ref{C:local_Shimura_Correspondence}]
    For $G$ of Cartan type $A_{r},D_{2r+1},E_{6},E_{7}$, there is an equivalence of categories between the $(I',\one)$-block of $G'$ and the $(\It, \sigma)$-block of $\Gt$.
\end{Thm}
The IM presentation and this corollary prove \ref{Exp:Heck_Iso} for these $\Gt$. This result will also hold for the other Cartan types, once the IM presentation is completed.

On our way to proving our type theorem and the IM presentation we obtain results on the structure and representation theory of $\Kt$, which we now describe.

For any smooth irreducible genuine $\Kt$-representation $\pi$, the subgroup $\pr(\ker\pi)$ must split into $\Kt$ and the image of this splitting (which is just $\ker\pi$) must be normal in $\Kt$. Our main theorem on the structure of $\Kt$ identifies the maximal principal congruence subgroup of $K$ that splits into $\Kt$ with normal image.

\begin{Thm}[Propositions \ref{Gamma12eSplitting}, \ref{P:SplitNormIm}]\label{Thm:K_Structure_intro}
    Let $e$ be the ramification index of $F/\mathbb{Q}_{2}$. Let $\Gamma(k)\subseteq K$ be the $k$-th principal congruence subgroup of $K$. Then
    \begin{enumerate}
        \item There exists a splitting $\s_{\Gamma(2e)}:\Gamma(2e)\rightarrow \Gt$ such that its image $\s_{\Gamma(2e)}(\Gamma(2e))$ is normal in $\Kt$;
        \item if $\Gamma(2e-1)$ splits into $\Kt$, then the image of this splitting is not normal in $\Kt$, nor $\It$.
    \end{enumerate}
\end{Thm}

Recall that the notion of depth introduced by Moy-Prasad \cite{MP94,MP96} has been generalized to tame covers by Howard-Weissman \cite{HW09}, using the fact that any pro-$p$ subgroup must split uniquely \cite[Proposition 2.3]{HW09}. The previous theorem demonstrates that the minimal depth with respect to the principal congruence filtration of a genuine irreducible representation of $\Kt$ is $2e$. 

Now we turn to the representation theory of $\Kt$. In the tame case, the isomorphism \eqref{Eq:Intro_Split_Map_K} gives a correspondence between representations of $\Kt$ and $K$. For wild covers, $K$ need not split, so there is no obvious correspondence. However, we can generalize Savin's finite Shimura correspondence \cite{S12}, which he proved for unramified $F/\mathbb{Q}_{2}$. This gives us a correspondence between the depth $0$ representations of $K$ and certain representations of $\Kt$ (Theorem \ref{FiniteShimCor}). We begin with a preliminary result.

\begin{Thm}[Theorem \ref{KRep}]\label{P:sigma_extends_to_K_Intro}
    The irreducible genuine $\It$-representation $\sigma$ constructed by parahoric induction from a pseudo-spherical $\Tt_{0}$-representation $\tau_{0}$ extends uniquely to a $\Kt$-representation. The extension of $\sigma$ is a minimal depth genuine representation of $\Kt$.

    The $\It$-representation $\sigma$ also remains irreducible when restricted to the first principal congruence subgroup $\Gammat(1)$.
\end{Thm}
This theorem proves \ref{Exp:Sig_Ext} for this family of wild double covers. We also write $\sigma$ for the extension of $\sigma$ to $\Kt$ and its restriction to $\Gammat(1)$. We note that this extension along with the IM presentation implies that the subalgebra $\Hcal(\Gt,\Kt,\sigma^{\vee})\subseteq \Hcal(\Gt,\It,\sigma^{\vee})$ satisfies a Satake isomorphism (Corollary \ref{C:Satake}).

Let $\kappa$ be the residue field of $F$, and $G_\kappa$ the $\kappa$-points of $G$. For a totally disconnected topological group $H$ let $\mathcal{M}(H)$ be the category of smooth representations of $H$. Let $\mathcal{M}_{(\Gammat(1),\sigma)}(\Kt)$ be the full subcategory of $\mathcal{M}(\Kt)$ consisting of smooth representations of $\Kt$ such that their restriction to $\Gammat(1)$ is $\sigma$-isotypic. Note that any representation of $G_{\kappa}$ inflates to a non-genuine representation of $\Kt$ via the isomorphism $G_{\kappa}\cong K/\Gamma(1)$ induced by the reduction mod $p$ map. Now we can state the finite Shimura correspondence.

\begin{Thm}[Theorem \ref{FiniteShimCor}]
The functor $\mathcal{M}(G_{\kappa})\rightarrow \mathcal{M}_{(\Gammat(1),\sigma)}(\Kt)$ defined by $\pi\mapsto \pi\otimes\sigma$ defines an equivalence of categories. 
\end{Thm}

This finite Shimura correspondence provides a construction of a large family of minimal depth representations of $\Kt$ and may be valuable for the construction of minimal depth supercuspidal representations of $\Gt$. It can also be used to reduce the study of certain Hecke algebras on $\Kt$ to related Hecke algebras on $K$ (Theorem \ref{T:finite_Hecke_algebra_isom}).

\quad

Next we highlight several of the ideas and tools that we develop to overcome the technical issues raised by the ramification of $F/\mathbb{Q}_{2}$. The first issue is that the structure of the $2$-fold cover $\Gt$ over $F$ is closely connected with the quadratic Hilbert symbol of $F$. The quadratic Hilbert symbol is well-understood in odd residual characteristic, but in even residual characteristic it is more complicated, due to ramification, and its properties seem to be less well-known. The most important result we need is Proposition \ref{P:property_Hilber_symbol_not_in_Appendix} \eqref{ORad}. We prove this and other results about the Hilbert symbol in Appendix \ref{S:Appendix_Hilbert_symbol}. We expect these results are known, but we have not been able to locate them in the literature. In any case, they provide the crucial input for Theorem \ref{Thm:K_Structure_intro}, which sets the foundation for our investigation of $\Gt$-representations.

The second issue is related to the construction of a type and understanding its associated Hecke algebra. Here we want to highlight two points. 

The first point is about Bushnell-Kutzko's theory of $G$-covers. Given the type $\tau_{0}$ for $\Tt$, the theory suggests a candidate type $(\It_{2e},\sigma_{2e})$, where $\It_{2e}\subseteq \It$ is an open compact subgroup and $\sigma_{2e}$ an irreducible $\It_{2e}$-representation. But $\It_{2e}$ is quite small compared to $\It$, so computing the support of the associated Hecke algebra is difficult. While this resembles a problem addressed by Roche \cite{R98}, his techniques exclude residual characteristic $2$ in most Cartan types. Instead, we set out to construct an Iwahori type $\sigma$ using Construction \ref{Construction:PS_I-rep}. Although $\sigma$ does not fit into the theory of $G$-covers, we extended the theory to our setting. (See Lemmas \ref{LemSurj}, \ref{LemInj}.) The use of an Iwahori type was inspired by Takeda-Wood \cite{Takeda_Wood}, where the type is constructed using the Weil representation. 

The second point is that our arguments crucially use two different models for the Hecke algebra. The first is the $\It$-model $\Hcal(\Gt, \It; \sigma^{\vee})$; the other involves a smaller open compact subgroup $\It_{2}\subseteq \It$ and an irreducible representation $\sigma_{2}$, also constructed from $\tau_{0}$ via parahoric induction. In fact, $\sigma\cong \Ind_{\It_{2}}^{\It}\sigma_{2}$ (Proposition \ref{IndSig2}) and so $\Hcal(\Gt, \It; \sigma^{\vee})\cong \Hcal(\Gt, \It_{2}; \sigma_{2}^{\vee})$, from general theory. Roughly, the $\It$-model makes bounding the support of the Hecke algebra more transparent, allowing us to avoid the tedious calculations in Karasiewicz \cite{Karasiewicz}; the $\It_{2}$-model is more compatible with the underlying affine Weyl group structure of $\Gt$. We will briefly elaborate on this affine Weyl group comment.

Let $T\subseteq G$ be a maximal split torus with cocharacter lattice $Y$. Let $\Phi$ be the root system of $(G,T)$. Let $\mathfrak{A}\subseteq \mathbb{R}\otimes Y$ be an alcove, with respect to the usual set of affine roots, containing $0$ in its closure. The stabilizer of $\mathfrak{A}$ in $G$ is an Iwahori subgroup, say $I$. Now we turn to $\Gt$. The root datum for the Langlands dual group of $\Gt$ suggests that the relevant root system for $\Gt$ is $\frac{1}{2}\Phi$. Thus in place of the alcove $\mathfrak{A}$ we should consider $2\mathfrak{A}$, and the stabilizer of $2\mathfrak{A}$ in $G$ is the group we have been calling $I_{2}$.

\quad

To close this introduction we will summarize the content of each section. In \S \ref{S:Prelim} we set up notation and recall some well-known results. We note that \S \ref{SS:HilbSymb} includes the fundamental properties of the Hilbert symbol that we use repeatedly. 

In \S \ref{S:2Cover} we recall the presentation of $\Gt$ and prove some basic results on the modified cocharacter lattice $\Yt\subseteq Y$ and the action of the Weyl group on it.

\S \ref{SecTORep} - \ref{S:IwahoriType} are the core of our work. In \S \ref{SecTORep} we recall the theory of pseudo-spherical representations of $\Tt_{0}$ developed in \cite{KT26}, and prove a few new results. 

In \S \ref{S:PSKType} we investigate the structure and representation theory of a hyperspecial maximal compact subgroup $\Kt\subseteq\Gt$ and other open compact subgroups. We review the general theory of parahoric induction and investigate our candidate Iwahori type $\sigma$. The main results in this section are a proof that the $\It$-representation $\sigma$ extends to $\Kt$ (Theorem \ref{KRep}) and the finite Shimura correspondence (Theorem \ref{FiniteShimCor}).

In \S \ref{S:Hecke} we investigate the structure of the Hecke algebra $\Hcal=\Hcal(\Gt,\It,\sigma^{\vee})$. 

In \S \ref{S:IwahoriType} we prove our main Theorem \ref{TypesThm}, which is an analogue of the Borel-Casselman Theorem for $\Gt$. As mentioned above, the pair $(\It,\sigma)$ does not fit into Bushnell-Kutzko's theory of $G$-covers. So, we also show how the notion of $G$-cover extends to our setting.

We end with Appendix \ref{S:Appendix_Hilbert_symbol}, collecting important properties of the Hilbert symbol, and Appendix \ref{S:Appendix_PS_Mult_One}, containing the proof of Proposition \ref{Prop:PS_Multi_one_Intro}.

\quad

\begin{center}{\bf Notation and assumptions}\end{center}

Throughout the paper, we let $F$ be a finite extension of $\Q_2$, with ring of integers $\Ocal$ and a chosen uniformizer $\varpi$, and let $\kappa=\Ocal\slash\varpi\Ocal$ be the residue field. Let $e, f\in\Z_{\geq 1}$ be the ramification index and the inertia degree of $F\slash\Q_2$, respectively, so that
\[
[F:\Q_2]=ef,\quad 2\Ocal=\varpi^e\Ocal\qand \kappa=\Ocal\slash\varpi\Ocal\cong \F_{q},
\]
where $q=2^f$. Let $d = [F:\Q_2]$.

Let $G$ be any group and $H\subseteq G$ a subgroup. Let $x\in G$. We set
\[
^xH=xHx^{-1}.
\]
For a representation $\pi$ of $G$, we define
\begin{equation*}
^x\pi
\end{equation*}
to be the representation of ${^xH}$ defined by
\[
^x\pi(h)=\pi(x^{-1}hx),\quad (h\in {^xH}).
\]
Note that $x^{-1}hx\in H$ because $h\in{^xH}=xHx^{-1}$, and hence this is well defined. More generally, for another group $H'$ and a group isomorphism $\phi:H\to H'$, we define 
\[
^\phi\pi
\]
to be the representation of $H'$ defined by
\begin{equation}\label{E:definition_phi_pi}
^\phi\pi(h')=\pi(\phi^{-1}(h')),\quad(h'\in H').
\end{equation}
We let
\[
\Int(x):H\longrightarrow xHx^{-1},\quad h\mapsto xhx^{-1}.
\]
Then $^x\pi={^{\Int(x)}}\pi$. We denote by $\pi^H$ the space of $H$-invariant vectors in $\pi$.

\quad

\begin{center}{\bf Acknowledgments}\end{center}
We would like to thank Gus Lehrer for pointing us to the work of Uri Onn, from which we learned about Dat \cite{D09} and Crisp-Meir-Onn \cite{CMO19}. We also thank Wee Teck Gan for numerous discussions throughout the preparation of this work. The second author was partially supported by the Sumitomo Foundation Fiscal 2024 Grant for Basic Science Research Projects J230603025, and  JSPS KAKENHI grant number 24K06648. Also part of the research was done while both authors were attending the conference ``The 25th Autumn Workshop on Number Theory" at Hokkaido University in fall 2024, and while the second author was visiting the National University of Singapore in Septembers of 2024 and 2025. We would like to thank their hospitality.

\begin{center}{\bf AI disclosure}\end{center}
ChatGPT-6 Astra was used for proof-reading and pointed out minor gaps in Proposition \ref{ExtICompat} and Proposition \ref{QuadSimpAff}, which were fixed by the authors.


\section{Preliminaries}\label{S:Prelim}


In this section, we first establish basic properties of the quadratic Hilbert symbol over our $2$-adic field $F$. After that, we will review the basic theory of Chevalley groups.

\subsection{$2$-adic Hilbert symbols}\label{SS:HilbSymb}
Let
\[
(-,-):F^\times\times F^\times\longrightarrow \{\pm 1\}
\]
be the quadratic Hilbert symbol of the $2$-adic field $F$. (For the general properties of the Hilbert symbol we refer the reader to \cite[Chapter V \S 3]{Neukirch}; for the more specialized results used in this paper see Appendix \ref{S:Appendix_Hilbert_symbol}.) Throughout, we often use
\[
(x, -x)=(x, 1-x)=1\qand (x, x)=(x, -1)
\]
for all $x\in F^\times$ such that the symbols are defined.

To investigate the structure of a double cover of a Chevalley group over a $2$-adic field we need to analyze the Hilbert symbol restricted to $\Ocal^\times\times\Ocal^\times$. Define
\begin{equation}\label{E:integral_radical_def}
R:=\{a\in\Ocal^\times\st \text{$(a, -)$ is trivial on $\Ocal^\times$}\},
\end{equation}
which we call the {\it integral radical} of the Hilbert symbol. 

In the next proposition we collect some properties of the quadratic $2$-adic Hilbert symbol.
\begin{Prop}\label{P:property_Hilber_symbol_not_in_Appendix}$ $
	\begin{enumerate}
		\item $(1+\varpi^{2e}\Ocal)\Ocal^{\times 2}= R$ and $1+\varpi^{2e+1}\Ocal\subseteq \Ocal^{\times 2}$. \label{ORad}
		\item $(1+\varpi^{2e}\Ocal)\slash((1+\varpi^{2e}\Ocal)\cap\Ocal^{\times 2})\cong R\slash\Ocal^{\times 2}\cong\Z\slash 2\Z$.
		\item $[\Ocal^{\times}:R]=2^{d}$.
	\end{enumerate}
\end{Prop}
\begin{proof}
See Appendix \ref{S:Appendix_Hilbert_symbol}.
\end{proof}

The following corollary will be crucially used in this paper.
\begin{Cor}\label{C:Hilbert_symbol_unique_element_integral_radical}
Let $u\in (1+\varpi^{2e}\Ocal)\smallsetminus\Ocal^{\times 2}$, which exists by the above proposition. Then for all odd $k\in\Z$ and $a\in\Ocal^\times$, we have
\[
(u, a\varpi^k)=-1.
\]
\end{Cor}
\begin{proof}
Since $u\in R$, we know $(a, u)=1$. So $(u, a\varpi^k)=(u, \varpi)$ because $k$ is odd. Since every element in $F^\times$ is written as $a\varpi^n$ for some $a\in \Ocal^\times$ and $n\in\Z$, if $(u, \varpi)=1$, then we would have $(u, x)=1$ for all $x\in F^\times$, which would imply $u\in \Ocal^{\times 2}$.
\end{proof}

The Hilbert symbol restricted to units descends to a non-degenerate bilinear form
\[
\Ocal^\times\slash R\times\Ocal^\times\slash R\longrightarrow\{\pm 1\}.
\]
Since $\Ocal^{\times 2}\subseteq R$, we can consider $\Ocal^\times\slash R$ as an $\F_2$ vector space with a non-degenerate symmetric bilinear form. (We note that this form may not arise from a quadratic form on $\Ocal^\times\slash R$.)

\subsection{Chevalley groups}\label{SS:ChevGroup}

In this subsection we review the basic structure theory of Chevalley groups and establish our notation. In particular we recall a presentation of such groups.

Let $G$ be the group of $F$-rational points of an almost simple simply-connected Chevalley group. Let $T\subseteq G$ be a maximal split torus and let $\Phi=\Phi(G,T)$ be the irreducible root system associated with the pair $(G,T)$. Let $W=N_G(T)\slash T$ be the Weyl group. Let $X=X^*(T)$ and $Y=X_*(T)$ be the group of rational characters and cocharacters, respectively, so that $(X, \Phi, Y, \Phi^\vee)$ is the root datum for $(G, T)$, where $\Phi^\vee$ is the set of coroots as usual. Note that we have the bijection
\[
\Phi\xrightarrow{\;\sim\;}\Phi^\vee,\quad \alpha\mapsto \alphac.
\]
Let
\[
\la-,-\ra:X\times Y\longrightarrow\Z
\]
be the canonical pairing. It is $W$-invariant and $\la\alpha, \alphac\ra=2$ for all $\alpha\in\Phi$. It should be mentioned that if $\Phi$ is simply-laced we have
\[
\la \alpha, \betac\ra=\la \beta, \alphac\ra
\]
for all $\alpha, \beta\in\Phi$. Since $G$ is simply-connected, we have $Y=\Z\Phi^{\vee}$. So, the identification of $\Phi$ with $\Phi^\vee$ gives rise to a $W$-invariant inner product on $X$, which we denote by the same symbol
\[
\la-,-\ra:X\times X\longrightarrow\Z.
\]
Although the symbol $\la-,-\ra$ has two meanings, the correct one will be clear from context.


Let $B=B^{+}\subseteq G$ be a Borel subgroup containing $T$ with unipotent radical $U=U^{+}$. Let $B^{-}\supseteq U^{-}$ be the opposite Borel and its unipotent radical. The Borel subgroup $B$ determines a set of simple roots
\[
\Delta=\{\alpha_1,\dots,\alpha_r\}\subseteq \Phi,
\]
which determines positive roots $\Phi^+\subseteq\Phi$. We order the simple roots by Bourbaki's convention. Let
\[
\alpha_0=\text{lowest root with respect to $\Delta$}.
\]


Now we recall a presentation for $G$. The group $G$ is generated by the symbols of the form $x_{\alpha}(t)$, namely
\[
G=\la \xa(t)\st \alpha\in\Phi, t\in F\ra.
\] 
For $t\in F^\times$, define
\begin{align*}
\wa(t)&:=\xa(t)x_{-\alpha}(-t^{-1})\xa(t);\\
\ha(t)&:=w_{\alpha}(t)\wa(-1).
\end{align*}
Then the following is a complete set of relations for the generators $\xa(t)$, when $rk(G)\geq 2$:
\begin{align}
\tag{R1}\xa(t)\xa(u)&=\xa(t+u);\label{R1}\\
\tag{R2}[\xa(t), \xb(u)]&=\prod_{i\alpha+j\beta\in\Phi}x_{i\alpha+j\beta}(c_{ij}t^iu^j),\quad(\alpha+\beta\neq 0);\label{R2}\\
\tag{R3}\ha(t)\ha(u)&=\ha(tu),\label{R3}
\end{align}
where the $c_{ij}\in \{\pm1\}$ are independent of $t$ and $u$ but dependent on $\alpha$ and $\beta$ and the ordering of the roots. Note that, if $i\alpha+j\beta\notin\Phi$, then we interpret $x_{i\alpha+j\beta}(c_{ij}t^iu^j)$ as 1. If $\Phi$ is simply-laced, \eqref{R2} simplifies to 
\begin{equation}
\tag{R2$^{\prime}$}[\xa(t), \xb(u)]=\begin{cases}1&\text{if $\alpha+\beta\notin\Phi$};\\
x_{\alpha+\beta}(ctu)&\text{if $\alpha+\beta\in\Phi$},
\end{cases}
\quad (\alpha+\beta\neq 0),
\end{equation}
where $c=c(\alpha, \beta)\in\{\pm 1\}$. Note also that $\wa(t)^{-1}=\wa(-t)$.

The following relations follow from \eqref{R1}, \eqref{R2}, \eqref{R3}:
\begin{align*}
\tag{R4}\ha(t)\hb(u)&=\hb(u)\ha(t);\label{R4}\\
\tag{R5}\wa(t)\hb(u)\wa(-t)&=h_{w_\alpha(\beta)}(u);\label{R5}\\
\tag{R6}\wa(t)\wb(u)\wa(-t)&=w_{w_\alpha(\beta)}(ct^{-\la\alpha, \beta\ra}u)\label{R6},
\end{align*}
where $c=c(\alpha, \beta)\in\{\pm 1\}$.

\subsection{Some notation for subgroups}
Let us establish some notation for subgroups of $G$.

The maximal torus $T$ is generated by the $\ha(t)$'s, namely
\[
T=\la\ha(t)\st t\in F^\times, \alpha\in\Phi\ra.
\]
For $\alpha\in \Phi$ and a subgroup $A\subseteq F^{\times}$, let 
\[
T_{\alpha, A}:=\{h_{\alpha}(a)\st a\in A\},
\]
and
\[
T_A:=\la h_{\alpha}(a)\st a\in A, \alpha\in\Phi\ra,
\]
both of which are subgroups of T. In particular, we often use
\begin{equation}\label{E:def_of_T_R}
T_R:=\la h_{\alpha}(a)\st a\in R, \alpha\in\Phi\ra,
\end{equation}
where $R$ is the integral radical defined in \eqref{E:integral_radical_def}. We also set
\begin{equation}\label{E:def_of_T_0}
T_0=T_{\Ocal^{\times}}=\la h_{\alpha}(a)\st a\in \Ocal^{\times}, \alpha\in\Phi\ra.
\end{equation}
We write
\[
T_{\alpha, k}:=T_{\alpha, 1+\varpi^k\Ocal}=\{h_{\alpha}(a)\st a\in 1+\varpi^k\Ocal\}
\]
and
\begin{equation}\label{E:def_of_T_k}
T_{k}:=T_{1+\varpi^k\Ocal}=\la h_{\alpha}(a)\st a\in 1+\varpi^k\Ocal, \alpha\in\Phi\ra.
\end{equation}
Note that every element $t\in T_{A}$ is uniquely written as $t=\prod_{1\leq i\leq r}h_{\alpha_{i}}(a_{i})$ for some $a_i\in A$ by using only the simple roots.

Next, for any subgroup $A\subseteq F$ and any $\alpha\in\Phi$, we set 
\[
U_{\alpha}(A)=\la x_{\alpha}(a)\st a\in A\ra
\] 
or more generally for any subset $S\subseteq \Phi$, we set
\[
U_{S}(A)=\la x_{\alpha}(a)\st a\in A, \alpha\in S\ra.
\] 
In particular, we set
\[
U^{+}(A)=U_{\Phi^{+}}(A)\qand U^{-}(A)=U_{\Phi^{-}}(A).
\]
Also we set
\begin{equation}\label{E:notation_U_alpha_j}
U_{\alpha, j}:=U_{\alpha}(\varpi^j\Ocal),\quad U_j^{\pm}:=U_{\Phi^{\pm}}(\varpi^j\Ocal)\qand U^{\pm}=U^{\pm}(F),
\end{equation}
where $j\in\Z$. We often omit the superscript $^+$, so in particular
\begin{equation}\label{E:notation_U_j}
U_j=U_j^+=U_{\Phi^{+}}(\varpi^j\Ocal).
\end{equation}

Set
\[
K:=\la x_{\alpha}(a)\st a\in\Ocal, \alpha\in\Phi\ra,
\]
namely the group generated by the subgroups $U_{\alpha,0}$ for all $\alpha\in\Phi$. This group $K$ is a hyperspecial maximal compact subgroup of $G$. Note that $T_{0}\subseteq K$. Also set
\[
I:=\la U_0, T_0, U_1^-\ra,
\]
namely the group generated by $U_0, T_0, U_1^-$, which is an Iwahori subgroup of $G$.

\subsection{Affine Weyl group of $G$}\label{SS:affine_Weyl_linear_group}
In this subsection, we review the basic theory of the affine Weyl group of the split reductive group $G$ over $F$. (In this subsection, $G$ is not necessarily our Chevalley group but any connected almost simple split reductive group over $F$.)

Let
\[
W_{\ea}:=N_G(T)\slash T_0
\]
be the extended affine Weyl group. The Iwahori-Bruhat decomposition tells that there is a natural bijection
\begin{equation}\label{E:Iwahori-Bruhat_decomp_reductive}
W_{\ea}\xrightarrow{\;\sim\;} I\backslash G\slash I,\quad w\mapsto IwI.
\end{equation}
Note that $W_{\ea}$ acts faithfully on $Y\otimes \R$ as affine transformations via the Tits homomorphism \cite[page 32]{T79}. 

For each $\alpha\in\Phi$ and $k\in\Z$, let
\[
H_{\alpha, k}:=\{v\in Y\otimes \R\st \la\alpha, v\ra=k\}
\]
be the affine hyperplane, where $\la-,-\ra$ is the ($\R$-valued pairing) on $(X\otimes\R)\times (Y\otimes \R)$ naturally induced from the canonical pairing on $X\times Y$. We then define
\[
w_{\alpha, k}: Y\otimes \R\longrightarrow Y\otimes \R
\]
by
\[
w_{\alpha, k}(v)=v-(\la\alpha, v\ra-k)\alphac,
\]
which is the affine reflection fixing $H_{\alpha, k}$. We let the affine Weyl group
\[
W_{\af}:=\la w_{\alpha, k}\st \alpha\in\Phi, k\in\Z\ra
\]
be the group generated by the affine reflections $w_{\alpha, k}$. We will write $w_{\alpha}=w_{\alpha,0}$. Let
\[
W=\la w_{\alpha}\ra=\{w\in W_{\af}\st w(0)=0\}.
\]
We call this subgroup the finite Weyl group. It is isomorphic to the Weyl group of $(G,T)$.

With the choice of a set of simple roots $\Delta\subseteq \Phi$, the affine Weyl group $W_{\af}$ can be realized as a Coxeter group with generators $S_{\af,\Delta}=\{w_{\alpha}|\alpha\in \Delta\}\cup\{w_{\alpha_{0},-1}\}$, where we recall that $\alpha_0$ is the lowest root relative to $\Delta$. Similarly, $W$ is a Coxeter group with generating set $S_{\Delta}=\{w_{\alpha}|\alpha\in \Delta\}$. Note that the generators in $S_{\af,\Delta}$ are the reflections through the walls of the alcove 
\[
\mathfrak{A}:=\{v\in Y\otimes \mathbb{R}\st\alpha(v)>0 \text{ for all }\alpha\in \Delta\text{ and }\alpha_{0}(v)>-1\}.
\]
One can define the length function
\[
\ell=\ell_{\Delta}:W_{\af}\longrightarrow\Z_{\geq 0}
\]
by setting $\ell(w)$ to be the minimum number of generators needed to express $w$. This naturally extends the length function of the finite Weyl group $W$. 

The map $W\ltimes \mathbb{Z}[\Phi^{\vee}]\to W_{\af}$, defined by $(w, y)\mapsto (v\mapsto w(v)+y)$, is a group isomorphism
\[
W_{\af}\cong W\ltimes \Z[\Phi^{\vee}].
\]
Similarly, the group $W_{\ea}$ is generated by $W_{\af}$ and translations by elements of $Y$, and one has
\[
W_{\ea}\cong W\ltimes Y.
\]

Let $\Omega\subseteq W_{\ea}$ be the subgroup that preserves $\mathfrak{A}$. The action of $\Omega$ on $\mathfrak{A}$ permutes the supporting hyperplanes of $\mathfrak{A}$ and thus induces an automorphism on the affine Dynkin diagram. Thus $\Omega$ permutes the generators of $W_{\af}$ from which it follows that $\Omega$ normalizes $W_{\af}$. It is also known that $W_{\ea}$ is generated by $\Omega$ and $W_{\af}$, and $\Omega\cap W_{\af}=\{1\}$, and thus 
\[
W_{\ea}\cong \Omega\ltimes W_{\af}.
\]
With this decomposition we can extend the length function $\ell$ on $W_{\af}$ to a length function on $W_{\ea}$ by defining $\ell(\omega w)=\ell(w)$, where $w\in W_{\af}$ and $\omega\in\Omega$. Finally, the group $\Omega$ can be shown to be isomorphic to the finite abelian group $Y/\mathbb{Z}[\Phi^{\vee}]$.


\section{Double covers}\label{S:2Cover}

It is known that the Chevalley group $G$ has a unique nontrivial topological $\mu_2$-extension $\Gt$, which fits into an exact sequence
\begin{equation}\label{CExt}
	1\longrightarrow\mu_2\longrightarrow\Gt\stackrel{\pr}{\longrightarrow} G\longrightarrow 1.
\end{equation}
In this section, we will establish various important properties of this $\mu_2$-extension $\Gt$. Although we assume that everything in this section is essentially known to experts, we will give a self-contained exposition and fix our notation. Also from this section on, we assume that $G$ is the almost simple simply-connected Chevalley group associated with a simply-laced root system $\Phi$, unless otherwise stated. Also we exclude the groups of type  $A_1$ and $D_2$. Hence $G$ is a Chevalley group of type $A_r (r\geq 2)$, $D_r (r\geq 3)$, $E_6, E_7$ or $E_8$.


\subsection{Steinberg relations}


In terms of generators and relations, $\Gt$ is given by
\[
\Gt=\la \epsilon\xta(t)\st t\in F, \alpha\in\Phi, \epsilon\in\mu_2\ra
\]
subject to the relations
\begin{align}
\tag{MR1}\xta(t)\xta(u)&=\xta(t+u);\label{MR1}\\
\tag{MR2}[\xta(t), \xtb(u)]&=\prod_{i\alpha+j\beta\in\Phi}\tilde{x}_{i\alpha+j\beta}(c_{ij}t^iu^j),\quad(\alpha+\beta\neq 0);\label{MR2}\\
\tag{MR3}\hta(t)\hta(u)&=(t, u)\hta(tu),\label{MR3}
\end{align}
where $c_{ij}$ is the same as linear case. Again, if $\Phi$ is simply-laced, we have
\begin{equation}\label{E:MR2'}
\tag{MR2$'$}[\xta(t), \xtb(u)]=\begin{cases}1&\text{if $\alpha+\beta\notin\Phi$};\\
\xt_{\alpha+\beta}(ctu)&\text{if $\alpha+\beta\in\Phi$},
\end{cases}
\quad (\alpha+\beta\neq 0),
\end{equation}
where $c\in\{\pm 1\}$. As before, we set
\begin{align*}
\wta(t)&:=\xta(t)\xt_{-\alpha}(-t^{-1})\xta(t);\\
\hta(t)&:=\wt_{\alpha}(t)\wta(-1).
\end{align*}
Let us note that, under the projection $\Gt\xrightarrow{\,\pr\,} G$, the images of the elements $\xta(t), \wta(t)$ and $\hta(t)$ are, respectively, $\xa(t), \wa(t)$ and $\ha(t)$.

Many of the important relations holding for these symbols are summarized in \cite[p.24]{Matsumoto}. In particular, we need to mention the following:
\begin{align}
\tag{MR4}\label{MR4}\hta(t)\htb(u)&=(u, t^{\la\beta,\alpha\ra})\htb(u)\hta(t);\\
\tag{MR5}\label{MR5}\wta(t)\htb(u)\wta(-t)&=\htt_{w_\alpha(\beta)}(c t^{-\la\alpha, \beta\ra}u)\htt_{w_\alpha(\beta)}(c t^{-\la\alpha, \beta\ra})^{-1}\\
\notag&=(u, c t^{-\la\beta,\alpha\ra})\htt_{w_\alpha(\beta)}(u);\\
\tag{MR6}\label{MR6}\wta(t)\wtb(u)\wta(-t)&=\wt_{w_\alpha(\beta)}(ct^{-\la\alpha, \beta\ra}u),
\end{align}
where $c=c(\alpha, \beta)\in\{\pm 1\}$ is as in \eqref{R6}, which is denoted by $\eta$ in \cite[(c), p.24]{Matsumoto}. As a special case of \eqref{MR5}, we have
\begin{gather}\tag{MR5'}\label{MR5'}
\begin{aligned}
\wta(1)\htb(u)\wta(-1)&=\htb(u)\hta(u),\quad (\alpha+\beta\in\Phi);\\
\wta(1)\hta(u)\wta(-1)&=\hta(u^{-1}).
\end{aligned}
\end{gather}

%

We often use the following.

\begin{Prop}\label{PMRootSwap}
Let $\alpha\in \Phi$. Let $t,u\in F^{\times}$ be such that $1+tu\neq 0$. Then
\begin{equation*}
\xt_{\alpha}\,(t)\xt_{-\alpha}(u)=(1+tu,-t)\,\xt_{-\alpha}(\frac{u}{1+tu})\,\htt_{\alpha}(1+tu)\,\xt_{\alpha}(\frac{t}{1+tu}),
\end{equation*}
where we note
\[
(1+tu, -t)=(1+tu, u).
\]
\end{Prop}
\begin{proof}
This is an elementary exercise using the Steinberg relations. Also note that $(1+tu, -t)=(1+tu, 1-(1+tu))(1+tu, -t)=(1+tu, u)$.
\end{proof}

The following corollary is immediate. 

\begin{Cor}\label{PMRootSwapCor}
Let $\alpha\in \Phi$. Let $t\in\mathcal{O}$ and $u\in \varpi^{2e}\mathcal{O}$. Then
\begin{equation*}
		\xt_{\alpha}(t)\,\xt_{-\alpha}(u)=\xt_{-\alpha}(\frac{u}{1+tu})\,\htt_{\alpha}(1+tu)\,\xt_{\alpha}(\frac{t}{1+tu}).
\end{equation*}
\end{Cor}
\begin{proof}
If $t\in\mathcal{O}^{\times}$, then $(1+tu,-t)=1$ because $1+tu\in U_{2e}\subseteq R$, where we recall $R$ is the integral radical.  If $t\in\varpi \mathcal{O}$, then $(1+tu,-t)=1$ because $1+tu\in U_{2e+1}\subseteq\Ocal^{\times 2}$.
\end{proof}


For any subgroup $H\subseteq G$ let 
\[
\widebar{H}:=\pr^{-1}(H).
\]
We say that $H$ splits sequence \eqref{CExt} or splits to $\Gt$ if there exists a group homomorphism $\s:H\rightarrow \widebar{G}$ such that $\pr\circ \s=id_{H}$. In this case, $H\times \mu_{2}\cong \widebar{H}$ via the map $(h,\pm1)\mapsto \pm\s(h)$. 

From the Steinberg relations \eqref{MR1} and \eqref{MR2} there is a group homomorphism 
\begin{equation}\label{E:section_for_unipotent_radical}
\s^{\pm}:U^{\pm}(F)\longrightarrow \Gt \text{\quad such that\quad} \s^{\pm}(x_{\alpha}(t))=\xt_{\alpha}(t)
\end{equation}
for all $\alpha\in\Phi^{\pm}$ and $t\in F$. This defines a splitting for $U^{\pm}$.

\subsection{The lattice $\Yt$}


In this subsection, we will introduce a certain sublattice $\Yt$ of $Y$ and explicitly describe the quotient $\Yt\slash 2Y$. We will then show how the lattice $\Yt$ is used to describe the center $Z(\Tt)$ of the covering torus $\Tt$. The description of the quotient $\Yt\slash 2Y$ (Proposition \ref{tildYmod2Y}) is essentially \cite[Proposition 4.2]{S04}. For our construction of pseudo-spherical representations, however, it will be convenient for us to reprove this result.

Recall that $Y=X_*(T)$ is the group of cocharacters (\ie cocharacter lattice). Since our group $G$ is simply-connected, we have $Y=\Z[\Phi^\vee]$. Hence each $y\in Y$ is uniquely written as 
\[
y=\sum_{i}c_i\alphac_i.
\] 
For each $t\in F^\times$, we set
\[
y(t)=h_{\alpha_1}(t^{c_1})\cdots h_{\alpha_n}(t^{c_n})\in T.
\]
Then the maximal torus $T$ is generated by $y(t)$'s. Indeed, we have the isomorphism
\[
Y\otimes_{\Z} F^\times\longrightarrow T,\quad y\otimes t\mapsto y(t).
\]

Let us consider the analogue for the cover $\Tt$. For each $y=\sum_{i}c_i\alphac_i$ and $t\in F^\times$, set
\[
\yt(t)=\htt_{\alpha_1}(t^{c_1})\cdots \htt_{\alpha_n}(t^{c_n})\in \Tt,
\]
where the order of multiplication is the order of $\Delta^\vee$. However, since $\Tt$ is not commutative due to (MR3) and (MR4), this product is not independent of the order of multiplication. We can, however, choose a sublattice of $Y$ so that the product is independent of the ordering as follows.

Let
\begin{equation}\label{E:definition_of_Yt}
\Yt=\{y\in Y\st \la \alpha, y\ra\in 2\Z\;\text{for all $\alpha \in \Phi$}\}.
\end{equation}
Note that
\[
2Y\subseteq \Yt\subseteq Y.
\]
Let us mention that we use the notation $Y_{Q, n}$ because our $Y_{Q, n}$ is the modified cocharacter lattice used to define the dual groups of BD-covers, which is also denoted by $Y_{Q, n}$. Indeed, $\Gt$ can be incarnated as a BD-cover by choosing $n=2, \eta=1$ and a fair bisector $D$ as in \cite{S04}. For more details, see \cite[Section 16.1]{GG18}. Note that $B_Q=\la-,-\ra$.

We will prove that the product of $\yt(t)$ is independent of the ordering of $y\in\Yt$. For this purpose, we explicitly derive the conditions that the coefficients $c_i$'s of $y=\sum_ic_i\alphac_i\in \Yt$ have to satisfy modulo $2\Z$, namely the ``parity conditions" to be satisfied by the $c_i$'s, as follows.

\begin{Lem}\label{L:parity_of_adjacent_root}
Let $y=\sum_ic_i\alphac_i\in \Yt$ and fix a simple root $\alpha_k$. Then we have
\[
\sum_{\text{$\alpha_i$ is adjacent to $\alpha_k$}}c_i\in 2\Z.
\]
Note that the number of $c_i$'s in the sum is at most $3$. 

Further, if $\alpha_k$ is adjacent to a root which has no other adjacent root than $\alpha_k$ itself (\ie $\alpha_k$ is adjacent to an ``endpoint" in the Dynkin diagram), then $c_k\in 2\Z$.
\end{Lem}
\begin{proof}
Since 
\[
\la \alpha_k,  y\ra=2c_k-\sum_{\text{$\alpha_i$ is adjacent to $\alpha_k$}}c_i\in 2\Z,
\]
the first part of the lemma follows.

Assume $\alpha_k$ is adjacent to an endpoint $\alpha_\ell$. Then
\[
\la \alpha_\ell, y\ra=2c_\ell-c_k\in 2\Z,
\]
which gives $c_k\in 2\Z$. The lemma follows.
\end{proof}

Now, given $y=\sum_ic_i\alphac_i\in \Yt$, we describe the parity of each $c_i$ by using the Dynkin diagram as follows. If $c_i$ is even, we denote the corresponding node for $\alpha_i$ in the Dynkin diagram by a black dot, and if it is odd, we denote it by a white dot. For example, assume our $\Phi$ is the $E_7$ root system. If $y=\sum_{i=1}^6c_i\alphac_i$ is such that $c_1, c_3, c_4$ and $c_6$ are even and the others are odd, then one can readily see that $y\in\Yt$. We denote this $y$ modulo $2Y$ by 
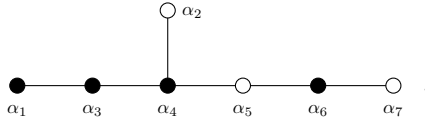
\[
\pgfkeys{/Dynkin diagram, edge-length=1cm, root-radius=.1cm}
\dynkin[labels={\alpha_1,\alpha_2,\alpha_3,\alpha_4,\alpha_5,\alpha_6,\alpha_7}]  E{*o**o*o}\rlap{\, .}
\]
Furthermore, one can see that this is the only non-zero element in $\Yt\slash 2Y$ for $E_7$ by using Lemma \ref{L:parity_of_adjacent_root} above, and hence $\Yt\slash 2Y=\Z\slash 2\Z$.

With this convention, we have
\begin{Prop}\label{tildYmod2Y}
Depending on the type of the Dynkin diagram, we have
\[
\Yt\slash 2Y=\begin{cases}\Z\slash 2\Z\times \Z\slash 2\Z&\text{Type $D_{2r}$};\\
\Z\slash 2\Z&\text{Type $A_{2r+1}, D_{2r+1}$ and $E_7$};\\
1&\text{otherwise}.
\end{cases}
\]
The results are summarized in Table \ref{T:table}, where the nontrivial elements in $\Yt\slash 2Y$ are described by Dynkin diagrams with white dots indicating odd and black dots even.
\end{Prop}
\begin{proof}
This can be verified case by case using Lemma \ref{L:parity_of_adjacent_root}. Since the verification is elementary, we omit the details.
\end{proof}

{

\begin{table}[h]
\caption{Lattice $\Yt$}\label{T:table}
\centering
\renewcommand{\arraystretch}{1.5}
\begin{tabular}{cccl}
\hline
Type & $\Yt\slash 2Y$ & $\left|Y\slash \Yt\right|$ &non-trivial elements in $\Yt\slash 2Y$ \\\hline\hline
$A_{2r+1}$ & $\Z\slash 2\Z$ & $2^{2r}$ & \dynkin A{o*o*.*o}\\
$A_{2r}$ & 1 & $2^{2r}$ &\\\hline
$D_{2r+1}$  & $\Z\slash 2\Z$ & $2^{2r}$ & \dynkin D{****.**oo}\\
$D_{2r}$ & $\Z\slash 2\Z\times \Z\slash 2\Z$  & $2^{2r-2}$ & \dynkin D{o*o*.o**o}\;\dynkin D{o*o*.o*o*}\;\dynkin D{****.**oo}\\\hline
 $E_6$ & 1 & $2^6$ & \\\hline
 $E_7$ & $\Z\slash 2\Z$ & $2^6$ &\dynkin E{*o**o*o} \\\hline
 $E_8$ & 1 & $2^8$ &\\\hline
\end{tabular}
\end{table}

}

As a small corollary, let us mention
\begin{Cor}\label{C:parity_of_adjacent_root3}
Let $y=\sum_ic_i\alphac_i\in \Yt$. If $\alpha_i$ and $\alpha_j$ are adjacent to each other, then $c_ic_j\in 2\Z$.
\end{Cor}
\begin{proof}
This can be read off from the Dynkin diagrams in Table \ref{T:table}.

%
\end{proof}

This corollary implies
\begin{Prop}\label{ZTbiject}
Let $y=\sum_ic_i\alphac_i\in\Yt$ and $t_i\in F^\times$. The product
\[
\prod_{i}\htt_{\alpha_i}(t_i^{c_i})
\]
is independent of the ordering.
\end{Prop}
\begin{proof}
We have only to show that if $\alpha_i$ and $\alpha_j$ are adjacent then
\[
\htt_{\alpha_i}(t_i^{c_i})\htt_{\alpha_j}(t_j^{c_j})=\htt_{\alpha_j}(t_j^{c_j})\htt_{\alpha_i}(t_i^{c_i}).
\]
But
\begin{align*}
\htt_{\alpha_i}(t_i^{c_i})\htt_{\alpha_j}(t_j^{c_j})&=(t_i^{c_i}, t_j^{-c_j})\,\htt_{\alpha_j}(t_j^{c_j})\htt_{\alpha_i}(t_i^{c_i})\quad\text{by (MR4)}\\
&=(t_i, t_j)^{c_ic_j}\htt_{\alpha_j}(t_j^{c_j})\htt_{\alpha_i}(t_i^{c_i}),
\end{align*}
and $(t_i, t_j)^{c_ic_j}=1$ by the above corollary.
\end{proof}

In particular, for each $y=\sum_ic_i\alphac_i\in\Yt$ and $t\in F^\times$, we can define
\[
\yt(t):=\prod_{i}\htt_{\alpha_i}(t^{c_i}),
\]
which is independent of the ordering by Proposition \ref{ZTbiject} with $t=t_i$ for all $i$.

\subsection{The Weyl group action on $\Yt$}


The Weyl group $W$ acts on $\Yt$ because the inner product $\la-,-\ra$ is $W$-invariant. Also $2Y$ is $W$-invariant. Hence the action of $W$ on $\Yt$ descends to $\Yt\slash 2Y$. But this action of $W$ on $\Yt\slash 2Y$ is trivial; namely we have
\begin{Lem}\label{L:Yt_is_Weyl_invariant_mod_2Y}
If $y\in \Yt$, then $w y=y$ in $\Yt\slash 2Y$ for all $w\in W$.
\end{Lem}
\begin{proof}
Let $\alpha\in\Delta$ be fixed. We have to show $w_\alpha y=y\pmod{2Y}$. Note that
\begin{equation*}
    w_{\alpha}y = y-\langle\alpha,y\rangle\alphac.
\end{equation*}
Now the result follows because $y\in \Yt$.

\end{proof}


\subsection{The group $\Wcal$}

For each $\alpha\in\Phi$, let
\[
\wta:=\wta(1).
\]
We then define
\begin{equation}\label{E:definition_of_Wcal}
\Wcal:=\la \wta\st\alpha\in\Phi\ra,
\end{equation}
namely the group generated by the $\wta$'s.

Note that
\[
\wta(1)^{-1}=\wta(-1).
\]
Since $\wta(-1)^{2}=\hta(-1)$ and $\hta(-1)^{-1}=(-1,-1)\hta(-1)$, we have
\begin{equation}\label{E:square_of_w}
\wta(1)^2=(-1,-1)\cdot\hta(-1).
\end{equation}
Hence each $\wta$ has order 4 or 8, depending on $(-1,-1)=1$ or $-1$ because $\hta(-1)^2=(-1,-1)$. The following should be mentioned : if $\alpha+\beta$ is a root, then
\begin{gather}\label{E:Savin_relation}
\begin{aligned}
\wta(1)\wtb(1)\wta(1)&=\wtb(1)\wta(1)\wtb(1)\quad\text{(braid relation)};\\
\wta(1)\htb(t)\wta(-1)&=\htb(t)\hta(t),
\end{aligned}
\end{gather}
where the second one is actually \eqref{MR5'} listed earlier. (See \cite[Th\'eor\`eme 6]{Matsumoto} and \cite[Lemma 3.2]{S04}, which does not assume the cover is tame.) Set
\[
\Tt(\Z):=\la\hta(-1)\st\alpha\in\Delta\ra.
\]
Then we have the exact sequence
\begin{equation}\label{E:exact_sequence_of_Weyl_group}
1\longrightarrow \Tt(\Z)\longrightarrow\Wcal\longrightarrow W\longrightarrow 1,
\end{equation}
where $W$ is the Weyl group of $G$ and the surjection sends each $\wta(1)$ to the corresponding Weyl group element.

\subsection{Affine Weyl group of $\Gt$}\label{SS:affine_Weyl_covering_group}

In this section, we modify the theory of affine Weyl groups of the linear group $G$ described in \S \ref{SS:affine_Weyl_linear_group} to that of the covering group $\Gt$.

First, note that the root datum $(X, \Phi, Y, \Phi^\vee)$ is isomorphic to the root datum
\[
(\frac{1}{2}X, \frac{1}{2}\Phi, 2Y, 2\Phi^\vee).
\]
We also have a modified root datum of $\Gt$ \cite{Mc12,W14}
\[
(\Xt, \frac{1}{2}\Phi, \Yt, 2\Phi^\vee),
\]
where $\Yt$ is as before and $\Xt=\Hom(\Yt,\mathbb{Z})$. This modified datum is that of the linear group
\[
G':=G\slash Z_2,
\]
where $Z_2$ is the 2-torsion of the center of $G$. Note that
\[
Z_2\cong \Yt\slash \Z[2\Phi^{\vee}].
\]

Let $\Wt_{\af}$ be the affine Weyl group associated with the root system $\frac{1}{2}\Phi$, which acts on $Y\otimes \mathbb{R}$. Fix a subset of simple roots $\Delta\subseteq \Phi$. Then $\frac{1}{2}\Delta $ is a subset of simple roots for $\frac{1}{2}\Phi$. Let
\begin{equation}\label{E:w_aff_definition}
    w_{\af}:=w_{\alpha_0, -2}=w_{\frac{1}{2}\alpha_{0},-1}.
\end{equation}
One can then see that
\[
\Wt_{\af}=\la w_{\af}, w_{\alpha_1}, \dots, w_{\alpha_n}\ra.
\]

Also let $\Wt_{\ea}$ be the extended affine Weyl group associated with $G'$. We then have
\[
\Wt_{\ea}\cong W\ltimes \Yt\qand \Wt_{\af}\cong W\ltimes 2Y,
\]
so that 
\[
1\longrightarrow \Wt_{\af}\longrightarrow \Wt_{\ea}\longrightarrow \Yt\slash 2Y\longrightarrow 1.
\]
(Note that in \cite{Karasiewicz}, our $\Wt_{\ea}$ is denoted by $\Wt_{\aff}$, and our $\Wt_{\af}$ by $\widehat{W}_{\aff}$.) 

Let 
\[
\tilde{\ell}=\tilde{\ell}_{\frac{1}{2}\Delta}:\Wt_{\ea}\cong W\ltimes \Yt\longrightarrow\Z_{\geq 0},
\]
be the length function, which can be also computed by the formula of \cite[Proposition 1.23]{IM}; namely
\[
\tilde{\ell}(w, y)=\sum_{\alpha\in\frac{1}{2}\Phi^+\cap w\frac{1}{2}\Phi^+}|\la\alpha, y\ra|+\sum_{\alpha\in\frac{1}{2}\Phi^{+}\cap w\frac{1}{2}\Phi^-}|\la\alpha, y\ra+1|,
\]
where $(w, y)\in W\ltimes\Yt$. Let $\Omegat$ be the set of length 0 elements. Then 
\[
\Omegat\cong \Yt\slash 2Y;
\]
indeed, we have
\[
\tilde{\ell}:\Omegat\ltimes \Wt_{\af}\longrightarrow\Z_{\geq 0},\quad \tilde{\ell}(w, \varepsilon)=\tilde{\ell}(w),
\]
for $(w, \varepsilon)\in \Omegat\ltimes \Wt_{\af}$. Note that $\Omegat$ is the subgroup that preserves the alcove $2\mathfrak{A}$, where $\mathfrak{A}$ is defined relative to $\Delta\subseteq\Phi$, as in \S \ref{SS:affine_Weyl_linear_group}.

Since $G$ is simply-connected and 
\[
\It\backslash\Gt\slash \It = I\backslash G\slash I,
\]
the Iwahori-Bruhat decomposition \eqref{E:Iwahori-Bruhat_decomp_reductive} implies
\[
W_{\ea}=W_{\af}\cong \It\backslash\Gt\slash \It.
\]
As we will see, however, the Iwahori Hecke algebra we will be considering is not supported on all of $W_{\af}$, but only on $\Wt_{\ea}$. This is precisely the rationale for introducing the groups $\Wt_{\af}$ and $\Wt_{\ea}$. 


\section{Pseudo-spherical representations of $\Tt_{0}$}\label{SecTORep}


Our goal in this section is to review the theory of pseudo-spherical representations for tori developed in \cite{KT26} and prove a few new results.


\subsection{Structure of $\Tt_{0}$}


Let us recall from \eqref{E:def_of_T_0} and \eqref{E:def_of_T_R} that we have defined
\[
T_0=T_{\Ocal}=\{h_{\alpha}(a)\st a\in \Ocal, \alpha\in\Phi\}
\qand
T_R=\{h_{\alpha}(a)\st a\in R, \alpha\in\Phi\},
\]
where $R$ is the integral radical as in \eqref{E:integral_radical_def}. We let $\Tt_0$ be the preimage of $T_0$ in $\Tt$, which is the maximal compact subgroup of $\Tt$.  

Let 
\[
\s_{T_{R}}:T_{R}\longrightarrow \Tt,\quad \prod_{i}h_{\alpha_{i}}(u_{i})\mapsto \prod_{i}\htt_{\alpha_{i}}(u_{i}),
\]
where the product on the right-hand side is independent of the ordering because $u_i\in R$ for all $i$. By comparing the relations \eqref{R3} and \eqref{R4} with \eqref{MR3} and \eqref{MR4}, one can readily see that this map is a homomorphism (\ie the sequence \eqref{CExt} splits over $T_{R}$), and moreover its image $\s_{T_R}(T_R)$ is central in $\Tt_0$, so that we can view $T_R$ as a normal subgroup of $\Tt_0$. Note that the above discussion is just a specialization of \cite[Proposition 3.1]{KT26}.

\subsection{Definition}


We recall the definition of a pseudo-spherical representation.
\begin{Def}[\cite{KT26}]\label{Def:PSTRep}
	A {\it pseudo-spherical representation} of $\Tt_{0}$ is an irreducible genuine representation of $\Tt_{0}$ which is trivial on $T_{R}$. 
\end{Def}

\begin{Rmk} 
    For the groups $G\supseteq T$ considered in this paper the notion of pseudo-spherical representation and distinguished pseudo-spherical representation coincide.
\end{Rmk}

Note that the pseudo-spherical representations of $\Tt_{0}$ remain irreducible when restricted to the first congruence subgroup.

\begin{Lem}\label{RedtoU1}
	The natural inclusion $\Tt_{1}\subseteq \Tt_{0}$ descends to a group isomorphism 
	\begin{equation*}
		\Tt_{1}/T_{R}\cap \Tt_{1}\xrightarrow{\;\sim\;}\Tt_{0}/T_{R},
	\end{equation*}
    where we recall that $\Tt_{1}$ is defined in \eqref{E:def_of_T_k}. 	Consequently, if $\tau$ is a pseudo-spherical representation of $\Tt_{0}$, then $\tau$ remains irreducible when restricted to $\Tt_{1}$.
\end{Lem}

\begin{proof}
	Recall that $\Ocal^{\times}=\Ocal^{\times 2}(1+\varpi\Ocal)$ by Lemma \ref{L:square_root_exits_in_redidue_field}. Thus the natural inclusion $1+\varpi\Ocal\subseteq\Ocal^{\times}$ descends to an isomorphism $(1+\varpi\Ocal)/R\cap (1+\varpi\Ocal)\xrightarrow{\sim}\Ocal^{\times}/R$. The result follows.
\end{proof}

\subsection{Pseudo-spherical representations: Summary of results}\label{SSPSpherT}




In this subsection we will summarize the relevant results about pseudo-spherical representations that are proved in \cite[Theorem 3.4]{KT26}.

\begin{Thm}\label{T:PseudoT0}
There exists a pseudo-spherical representation of $\Tt_{0}$. Moreover, every pseudo-spherical representation $\tau$ is $\Wcal$-invariant, namely ${^{\wt}\tau}\cong\tau$ for all $\wt\in\Wcal$.
\end{Thm}

\begin{Prop}$ $
\begin{itemize}
    \item Let $\tau$ be a pseudo-spherical representation. Then
\[
\dim\tau=\left|Y\slash \Yt\right|^{\frac{d}{2}},
\]
where we recall $d=[F:\Q_2]$.
    \item The number of the pseudo-spherical representations of $\Tt_{0}$ is $\left|\Yt\slash 2Y\right|^{d}$.
\end{itemize}

\end{Prop}

The dimension and the number of the pseudo-spherical representations of $\Tt_{0}$ is summarized in Table \ref{T:table2}, where $A_{1}$ and $D_2$ are excluded.

\begin{table}[h]
\caption{dimension and number of pseudo-spherical representations}\label{T:table2}
\centering
\renewcommand{\arraystretch}{1.5}
\begin{tabular}{ccccc}
\hline
Type & $\left| \Yt\slash 2Y\right|$ & $\left|Y\slash \Yt\right|$ &dimension &number \\\hline\hline
$A_{2r+1}$ & 2 & $2^{2r}$& $2^{dr}$ & $2^{d}$\\
$A_{2r}$ &  1 & $2^{2r}$ & $2^{dr}$ & 1\\\hline
$D_{2r+1}$  & 2 & $2^{2r}$ & $2^{dr}$ & $2^{d}$\\
$D_{2r}$ & 4  & $2^{2r-2}$ & $2^{d(r-1)}$ & $4^{d}$\\\hline
 $E_6$ & 1 & $2^6$ & $2^{3d}$ & 1\\\hline
 $E_7$ & 2 & $2^6$ & $2^{3d}$ & $2^{d}$ \\\hline
 $E_8$ & 1 & $2^8$ & $2^{4d}$ & 1\\\hline
\end{tabular}
\end{table}








\subsection{Branching}

For later purposes, we need the following.

\begin{Prop}\label{PseudoSphericalBranch}
	Suppose that $\Phi$ is a root system of type $A_r$ or of type $D_r$. Let $\Phi^{\prime}\subseteq\Phi$ be a subroot system of rank $r-1$ and of the same type as $\Phi$. Let $\Tt_{0}^{\prime}\subseteq \Tt_{0}$ be the normal subgroup generated by the elements $\htt_{\alpha}(u)$, where $u\in \Ocal$ and $\alpha\in\Phi^{\prime}$. 
    \begin{enumerate}[(1)]
    \item For a pseudo-spherical representation $\tau$ of $\Tt_{0}$, the restriction $\tau|_{\Tt_0^{\prime}}$ to $\Tt_{0}^{\prime}$ is a multiplicity-free direct sum of pseudo-spherical representations of $\Tt_{0}^{\prime}$.
    \item If $\tau'$ is a pseudo-spherical representation of $\Tt_0'$, then there exists a pseudo-spherical representation $\tau$ of $\Tt_0$ such that $\tau'\subseteq\tau|_{\Tt_0}$.
    \end{enumerate}
\end{Prop}

\begin{proof}
	The restriction $\tau|_{\Tt_0^{\prime}}$ is a direct sum of pseudo-spherical representations of $\Tt_{0}^{\prime}$ {\it by definition} because $T_R'\subseteq T_R$. Now we will break the proof of (1) up into cases.
	
	First suppose that $\Phi$ is of type $A_{2r+1}$ or $D_{2r}$. From the dimension column of Table \ref{T:table2}, we see that this restriction to $\Tt_{0}^{\prime}$ must remain irreducible. Thus the result follows in this case.
	
	Second, suppose that $\Phi$ is of type $A_{2r}$ or $D_{2r+1}$. Let $\Delta\subseteq\Phi$ be a set of simple roots. Then $\Delta^{\prime}:=\Delta\cap\Phi^{\prime}$ is a set of simple roots in $\Phi^{\prime}$. Let $\alpha\in \Delta\smallsetminus\Delta^{\prime}$. Let $\Tt_{\alpha,0}\subseteq \Tt_{0}$ be the subgroup generated by the elements of the form $\htt_{\alpha}(u)$, where $u\in \Ocal^{\times}$. Note that $\Tt_{\alpha,0}$ is a normal subgroup of $\Tt_{0}$. By definition the restriction of $\tau$ to $\Tt_{\alpha,0}$ factors through the finite abelian group $\Tt_{\alpha,0}/(\Tt_{\alpha,0}\cap T_R)$, which has size $2[\Ocal^{\times}:R]=2^{d+1}$.
	
	Let $\tau^{\prime}$ be a pseudo-spherical representation of $\Tt_{0}^{\prime}$. By a direct computation with the central character it follows that $\tau^{\prime}\cong \,^{\htt_{\alpha}(u)}\tau^{\prime}$ if and only if $u\in R$. From this it follows that the restriction $\tau|_{\Tt_{0}^{\prime}}$ must contain at least $[\Ocal:R]=2^{d}$ distinct pseudo-spherical $\Tt_{0}^{\prime}$-modules. By comparing the dimensions in Table \ref{T:table2} it follows that the restriction of $\tau$ to $\Tt_{0}^{\prime}$ is a multiplicity-free direct sum of pseudo-spherical representations of $\Tt_{0}^{\prime}$.
    Item (1) is proven. 
    
    To prove (2), let $\tau'$ be a pseudo-spherical representation of $\Tt'_0$. Since $T_R'=\Tt_0'\cap T_R$ acts trivially and $T_R$ commutes with $\Tt_0'$ pointwise, we can extend $\tau'$ to a representation of $T_R\Tt_0'$ by letting $T_R$ act trivially. Consider $\Ind_{T_R\Tt_0'}^{\Tt_0}\tau'$. Since $T_R$ is in the center of $\Tt_0$, one can see that $T_R$ is in the kernel of this induced representation. Hence any irreducible constituent of this induced representation is pseudo-spherical and contains $\tau'$.
\end{proof}


\subsection{Square roots of $\htt_{\alpha}(-1)$}\label{SSSqRoot}


In this subsection we record a few identities in the group algebra $\mathbb{C}[\Tt_{0}/T_{R}]$ and their relation to each pseudo-spherical representation $\tau$ of $\Tt_{0}$. In particular, we will compute a square root of $\tau(\htt_{\alpha}(-1))$. Later these will be important to our investigation of representations of $\It$. Also in this subsection, we set
\begin{equation}\label{E:U_k_def_subset_of_Ocal}
U_k:=1+\varpi^k\Ocal
\end{equation}
for $k\in\Z_{\geq 0}$ with the convention that $U_0=\Ocal^{\times}$.

Let $\alpha\in\Phi$, $j\in \mathbb{Z}$ such that $0\leq j\leq 2e$ and $\chi\in \Hom(U_{j}/(U_{j}\cap R),\mathbb{C}^{\times})$. Define
\begin{equation*}
	X_{\alpha,\chi,j}:=\frac{1}{\left|U_j/U_{2e}\right|}\sum_{v\in U_{j}/U_{2e}}\chi(v)\htt_{\alpha}(v).
\end{equation*}
If $\chi$ is such that $\chi(-)=(-, u)$ for some $u\in F^{\times}$, then we write 
\begin{equation}\label{E:X_alpha_u_j_def}
X_{\alpha,u,j}:=X_{\alpha,\chi,j}.
\end{equation}

\begin{Lem}\label{Xpmalpha}
With the above notation, we have
	\begin{equation*}
		X_{-\alpha,u,j}=X_{\alpha,-u,j}.
	\end{equation*}
\end{Lem}

\begin{proof}
	This follows because $\htt_{-\alpha}(v)=(v,-1)\htt_{\alpha}(v^{-1})$.
\end{proof}

The next lemma shows that the elements $X_{\alpha,u,j}$ define nontrivial endomorphisms of $\tau$.

\begin{Lem} Let $\tau$ be a pseudo-spherical representation of $\Tt_{0}$. Then
	\begin{equation*}
		\Tr(\tau (X_{\alpha,u,j})) = |(U_{j}\cap R)/U_{j}|\dim\tau\neq 0.
	\end{equation*}
\end{Lem}

\begin{proof}
	For $t\in \Tt_{0}$, we have $\Tr(\tau(t))\neq 0$ if and only if $t\in Z(\Tt_{0})$. (This follows because if $t\in\Tt_0\smallsetminus Z(\Tt_0)$, then there exists $s\in\Tt_0$ such that $sts^{-1}=-t$, so $\Tr(\tau(sts^{-1}))=-\Tr(\tau(t))$.) The result now follows because $\htt_{\alpha}(v)\in Z(\Tt_{0})$ if and only if $v\in R$. (Here we use that $\Phi$ is simply-laced and not of type $A_{1}$ or $D_2$.)
\end{proof}

\begin{Lem} If $j=0$ or $1$, then
	\begin{equation*}
		X_{\alpha,\chi,j}^{2}=\chi(-1)|U_{j}/U_{2e}|^{-1}\htt_{\alpha}(-1)\sum_{v\in U_{j}\cap R/U_{2e}}\htt_{\alpha}(v).
	\end{equation*}
\end{Lem}

\begin{proof} This follows by a direct calculation:
	\begin{align*}
		\left|U_j/U_{2e}\right|^{2}X_{\alpha,\chi,j}^{2}=&\sum_{v,w\in U_{j}/U_{2e}}\chi(vw)\htt_{\alpha}(v)\htt_{\alpha}(w)\\
		=&\sum_{v,w}\chi(vw)(v,w)\htt_{\alpha}(vw)\\
		=&\sum_{v,w}\chi(v)(vw^{-1},w)\htt_{\alpha}(v)\\
		=&\sum_{v}\chi(v)\htt_{\alpha}(v)\sum_{w}(-v,w),
	\end{align*}
where we used $(vw^{-1}, w)=(vw,w)=(v, w)(w, w)=(v, w)(-1, w)=(-v, w)$.

	Because $j=0,1$ we have
	\begin{align*}
		\sum_{v}\chi(v)\htt_{\alpha}(v)\sum_{w}(-v,w)=&\chi(-1)|U_{j}/U_{2e}|\sum_{-v\in U_{j}\cap R/U_{2e}}\htt_{\alpha}(v)\\
		=&\chi(-1)|U_{j}/U_{2e}|\sum_{v\in U_{j}\cap R/U_{2e}}\htt_{\alpha}(-v)\\
		=&\chi(-1)|U_{j}/U_{2e}|\htt_{\alpha}(-1)\sum_{v\in U_{j}\cap R/U_{2e}}\htt_{\alpha}(v).
	\end{align*}
    The lemma follows.
\end{proof}

\begin{Cor}\label{XInvert}
	Let $\alpha\in\Phi$, $\chi\in \Hom(U_{j}/U_{j}\cap R,\mathbb{C}^{\times})$, and $j=0,1$. If $\tau$ is a pseudo-spherical representation, then
	\begin{equation*}
		\tau(X_{\alpha,\chi,j}^{2}) = \chi(-1)\frac{|U_{j}\cap R|}{|U_{j}|}\tau(\htt_{\alpha}(-1)).
	\end{equation*} 
In particular, $\tau(X_{\alpha,\chi,j})$ is invertible.
\end{Cor}

\begin{Rmk}
For $\alpha\in \Delta$, the element $\tau(X_{\alpha,1,1})$ is proportional to the extension of the pseudo-spherical $\Tt$-module $\tau$ to the element $w_{\alpha}(-1)\in\Wcal$.
\end{Rmk}


\subsection{Representation theory of $\Tt$}\label{TtReps}


Let $\tau$ be a pseudo-spherical representation of $\Tt_0$. In \cite[Corollary 5.6]{KT26} we prove that $(\Tt_0, \tau)$ is a type of $\Tt$ in the sense of Bushnell-Kutzko \cite[(4.1) Definition]{BK98}. In this subsection we recall this result.

First, we extend $\tau$ to an irreducible representation of $Z(\Tt)\Tt_0$ by adding a central character. Namely, we need a character on $Z(\Tt)$ which agrees with $\tau$ on $Z(\Tt)\cap\Tt_0$. This can be done by using \cite[Lemma 5.4]{KT26}.

Let $\Yt(\varpi)\subseteq T$ be the subgroup generated by elements of the form $y(\varpi)$, where $y\in \Yt$.
\begin{Lem}
We have the short exact sequence
\[
1\longrightarrow\mu_2 \xrightarrow{\;\Delta\;}\mu_2\Yt(\varpi)\times \Tt_0\longrightarrow Z(\Tt)\Tt_0\longrightarrow 1,
\]
where $\mu_2$ embedds anti-diagonally and the surjection is given by multiplication.
\end{Lem}

Let 
\[
\eta:\mu_2\Yt(\varpi)\longrightarrow \C^\times
\]
be a genuine $W$-invariant character trivial on elements of the form $h_{\alpha}(\varpi^{2})$, where $\alpha\in \Phi$. (For example, one can restrict the character of $Z(\Tt)$ constructed in \cite[page 118]{S04} or \cite[\S 7]{GG18}.) Any two such characters differ by a quadratic character of $\Yt/2Y$. Consider the representation $\eta\otimes\tau$ of $\mu_2\Yt(\varpi)\times\Tt_0$. Since both $\eta$ and $\tau$ are genuine, the diagonal $\mu_2$ acts trivially, which allows
\[
\tau_\eta:=\eta\otimes\tau
\]
to descend to a representation of $Z(\Tt)\Tt_0$, which we also denote by the same symbol $\tau_\eta$.

Now, for each unramified character $\chi:T\to\C^\times$, we define
\[
i_{\eta,\tau}(\chi):=\chi\otimes\Ind_{Z(\Tt)\Tt_0}^{\Tt}\tau_{\eta}.
\]
Note the following. First, for any $\alpha\in\Phi$ and $t\in F^{\times}$ the elements $\htt_{\alpha}(t^{2})\in Z(\Tt)$ act trivially on $i_{\eta,\tau}(\mathbf{1})$, where $\mathbf{1}$ is the trivial character of $T$. Thus $i_{\eta,\tau}(\mathbf{1})$ is a distinguished genuine $\Tt$-representation (\cite[Section 6]{GG18}). Furthermore, by choosing $\eta$ and $\tau$ we can construct $|\Yt/2Y|^{1+ef}$ distinguished genuine representations of $\Tt$. There are $|\Yt/2Y|^{2+ef}$ in total (\cite[Theorem 6.6]{GG18}). Second, the two representations $i_{\eta,\tau}(\chi)$ and $i_{\eta^{\prime},\tau^{\prime}}(\chi^{\prime})$ can be related by an unramified twist if and only if $\tau\cong\tau^{\prime}$ and $\eta^{-1}\otimes \eta^{\prime}$ defines a quadratic character of $\Yt/2Y$. Thus the Bernstein block of $i_{\eta,\tau}(\chi)$ depends only on $\tau$.

Finally we recall a slight variation on \cite[Proposition 5.5]{KT26}.

\begin{Prop}\label{TTypes}
Let $\pi$ be a smooth irreducible $\Tt$-module. The $\tau$-isotypic subspace of $\pi|_{\Tt_{0}}$ is nonzero if and only if $\pi\cong i_{\eta,\tau}(\chi)$ for some unramified character $\chi:T\rightarrow \mathbb{C}^{\times}$.
\end{Prop}

\section{Pseudo-spherical representations of $\Kt$}\label{S:PSKType}


Now that we have constructed a pseudo-spherical representation $\tau$ of $\Tt_0$, the next task is to construct an irreducible representation $\sigma$ of the preimage $\It$ of the Iwahori subgroup $I\subseteq G$ that is to be used to define our Iwahori Hecke algebra. We achieve this by applying Dat's {\it parahoric induction} \cite{D09}. We call this representation $\sigma$ of $\It$ a pseudo-spherical representation of $\It$. This representation is essentially ``Weyl invariant", which allows us to extend it to a representation of $\Kt$, the preimage of the hyperspecial maximal compact subgroup $K\subseteq G$. We also call this extension a pseudo-spherical representation of $\Kt$.

Also at the end, we establish a finite group Shimura correspondence, generalizing the work of Savin \cite{S12} which proves the case when $F$ is unramified over $\Q_2$. More precisely, tensoring by the pseudo-spherical representation defines an equivalence of categories between the category of representations of $G(\Ocal/\varpi\Ocal)$ and a certain subcategory of $\Kft$-modules, where $\Kft$ is a finite quotient of $\Kt$.


\subsection{Structure of $\Kt$}\label{SSecSplittings}

In this subsection we prove that certain subgroups of $K$ split to $\Gt$. In particular, we identify the maximal principal congruence subgroup of $K$ which splits with a normal image in $\Kt$.

To begin we set up some notation for congruence subgroups of $G$. For $k\in\mathbb{Z}_{\geq1}$, define
\begin{align*}
\Gamma(k)&:=\la T_{k},\, \xa(t)\st t\in\varpi^k\Ocal,\, \alpha\in\Phi\ra;\\
\Gamma_1(k)&:=\la\Gamma(k),\, \xa(t)\st t\in\Ocal,\, \alpha\in\Phi^+\ra;\\
\Gamma_0(k)&:=\la \Gamma_1(k),\, T_{0}\ra,
\end{align*}
where we recall
\[
T_k=T_{1+\varpi^k\Ocal}=\la\htt_{\alpha}(t)\st t\in 1+\varpi^k\Ocal, \alpha\in\Phi\ra.
\]
Note that 
\[
\Gamma(k)\subseteq\Gamma_1(k)\subseteq\Gamma_0(k).
\] 
Alternatively, if $\mathbb{G}$ is an $\Ocal$-model for $G$ such that $K=\mathbb{G}(\Ocal)$, then one can consider the surjective mod $\varpi^k$ map (\cite[Proposition 1.6]{A69})
\[
\mathbb{G}(\Ocal)\longrightarrow \mathbb{G}(\Ocal\slash\varpi^k\Ocal),
\]
and the groups $\Gamma(k), \Gamma_1(k)$ and $\Gamma_0(k)$ are the preimages of the trivial group, the unipotent radical $U(\Ocal\slash\varpi^k\Ocal)$ and the Borel subgroup $B(\Ocal\slash\varpi^k\Ocal)$, respectively.

We are mostly concerned with the groups
\[
\Gamma(2e),\quad \Gamma(1), \quad \Gamma_1(2e),\quad   \Gamma_0(1),\quad\Gamma_0(2),
\]
where we often write
\begin{equation}\label{E:I_and_I_2_def}
    I:=\Gamma_0(1)\qand I_2:=\Gamma_0(2).
\end{equation}
The former is, of course, the Iwahori subgroup. Note the inclusions
\[
\Gamma_1(2e)\subseteq I_2\subseteq I\subseteq K
\qand \Gamma(2e)\subseteq\Gamma(1)\subseteq I\subseteq K.
\]

Let $\Gamma$ be any of the groups $\Gamma_{0}(k),\Gamma_{1}(k),\Gamma(k)$. The group $\Gamma$ admits an Iwahori factorization, meaning the multiplication map 
\begin{equation*}
	\mu_{\Gamma}: (U^{-}_{0}\cap \Gamma)\times (T_{0}\cap \Gamma)\times (U_{0}\cap \Gamma)\rightarrow \Gamma
\end{equation*}
is a homeomorphism. Using this factorization we can define a set-theoretic section of the exact sequence \eqref{CExt} for these congruence subgroups.

First, recall the sections $\s^{\pm}$ as in \eqref{E:section_for_unipotent_radical}. Second, recall that every element of $T$ can be uniquely written in the form $\prod_{\alpha}h_{\alpha}(u_{\alpha})$, where $u_{\alpha}\in F^{\times}$. If we fix a linear order on $\Delta$ we can define a (set-theoretic) section $\s_{T}:T\rightarrow \Tt$ by $\prod_{\alpha}h_{\alpha}(u_{\alpha})\mapsto \prod_{\alpha}\htt_{\alpha}(u_{\alpha})$.

Now we can define a section of $\Gamma$ by
\[
\s_{\Gamma}:=(\s^{-}\times \s_{T}\times \s^{+})\circ \mu_{\Gamma}^{-1}:\Gamma \rightarrow \Gt.
\]
The map $\s_{\Gamma}$ is not a group homomorphism in general, but for $\Gamma=\Gamma_{1}(2e)$ it is a homomorphism, namely
\begin{Prop}\label{Gamma12eSplitting}
	Let $\Gamma=\Gamma_{1}(2e)$. The section 
    \[
    \s_{\Gamma}:=(\s^{-}\times \s_{T}\times \s^{+})\circ \mu_{\Gamma}^{-1}:\Gamma \longrightarrow \Gt
    \]
    is a splitting of the exact sequence \eqref{CExt}.
\end{Prop}
\begin{proof}
    We already know that each of $\s^{+}, \s_{T}|_{T_{2e}}, \s^{-}$ is a homomorphism. Hence, it suffices to show that the image is closed under multiplication. Note that $\Gamma$ is generated by the subgroups $T_{2e}$, $U_{\alpha_{0},2e}$ and $U_{\alpha,0}$ with $\alpha\in \Delta$, where we recall the notation $U_{\alpha_{0},2e}$ and $U_{\alpha,0}$ from \eqref{E:notation_U_alpha_j}. It is clear that the image is closed under multiplication on the left by $T_{2e}$ and $U_{\alpha_{0},2e}$, and under multiplication on the right by $T_{2e}$ and $U_{0}$. So it remains to show that for any $\xt_{\alpha}(a)$, where $\alpha\in\Delta$ and $a\in\Ocal$, and any $y\in U^{-}_{2e}$, we have
	\begin{equation*}
		\xt_{\alpha}(a)\s_{\Gamma}(y)\in \s_{\Gamma}(\Gamma).
	\end{equation*}
	
	Let $V=U_{\Phi^{-}\smallsetminus\{-\alpha\}}(4\Ocal)\subseteq U^{-}_{2e}$. Since $\alpha$ is a simple root, the multiplication map defines a set bijection $V\times U_{-\alpha,2e}\cong U^{-}_{2e}$, and $\xt_{\alpha}(a)$ normalizes $V$. 
	
	Now let $v\in V$ and $b\in \Ocal$ be such that $y=v\xt_{-\alpha}(4b)$. Because $\xt_{\alpha}(a)$ normalizes $\s^{-}(V)$ by \eqref{MR2}, we have
	\begin{align*}
		\xt_{\alpha}(a)\s_{\Gamma}(y) =& [\xt_{\alpha}(a)\s^{-}(v)\xt_{\alpha}(-a)] \xt_{\alpha}(a)\xt_{-\alpha}(4b)\\
		=& [\xt_{\alpha}(a)\s^{-}(v)\xt_{\alpha}(-a)] \xt_{-\alpha}(\frac{4b}{1+4ab})\htt_{\alpha}(1+4ab)\xt_{\alpha}(\frac{a}{1+4ab})\\
		=& \s_{\Gamma}\left(x_{\alpha}(a)vx_{\alpha}(-a) x_{-\alpha}(\frac{4b}{1+4ab})\right)\htt_{\alpha}(1+4ab)\xt_{\alpha}(\frac{a}{1+4ab}),
	\end{align*}
    where we used Lemma \ref{PMRootSwapCor} for the second equality. The result follows.
\end{proof}

\begin{Rmk}
Proposition \ref{Gamma12eSplitting} generalizes the splitting results in type $A_2$ of \cite[Theorem 3.3]{Karasiewicz}, which are local counterparts to the splitting studied in \cite{K22}.
\end{Rmk}

Now, let us denote the splitting $\s_{\Gamma_1(2e)}$ simply by $\s$, so we have a homomorphism 
\[
\s:\Gamma_1(2e)\longrightarrow\Kt.
\]

The following proposition follows directly from the Steinberg relations, Corollary \ref{PMRootSwapCor}, and our results on the Hilbert symbol.

\begin{Prop}\label{P:SplitNormIm}$ $
    \begin{enumerate}
        \item The subgroup $\s(\Gamma_1(2e))\subseteq \Gammat_{0}(2e)$ is normal.
        \item\label{Gam2eSplitNorm} The subgroup $\s(\Gamma(2e))\subseteq \Kt$ is normal.
        \item If $\Gamma(2e-1)$ splits, then its image cannot be normal in $\Kt$.
    \end{enumerate}
\end{Prop}

\begin{Rmk}\label{R:Gamma(2e)_is_max} Let $e^{\prime}$ be the even element in the set $\{e,e+1\}$. A more detailed analysis of the Hilbert symbol should show that $\s_{\Gamma(e^{\prime})}$ is a splitting of $\Gamma(e^{\prime})$. However, one can readily see, by comparing the relations \eqref{R3} and \eqref{R4} with \eqref{MR3} and \eqref{MR4}, that  the image of any splitting of $\Gamma(2e-1)$ cannot be normal in $\Kt$, and thus $\Gamma(2e)$ is the maximal principal congruence subgroup of $K$ which splits with normal image in $\Kt$.
\end{Rmk}

Now, since the group $\s(\Gamma(2e))$ is normal in $\Kt$, we can let
\begin{equation}\label{E:def_finite_K_tilde}
\Kft:=\Kt/\s(\Gamma(2e)).
\end{equation}
Then this finite group is a nontrivial $\mu_{2}$-central extension of $K/\Gamma(2e)$; namely we have
\begin{equation}\label{Mod4CExt}
	1\rightarrow \mu_{2}\rightarrow \Kft\rightarrow K/\Gamma(2e) \rightarrow 1.
\end{equation}
We know the extension is nontrivial because $\Tt_{0}/\s(T_{2e})$ is nonabelian.

\begin{Rmk} Suppose that $\mathbb{G}$ is an $\Ocal$-model for $G$ such that $K=\mathbb{G}(\Ocal)$. Then the group $K/\Gamma(2e)$ is isomorphic to $\mathbb{G}(\Ocal/4\Ocal)$. Central extensions of such groups, which include $\Kft$, have been studied by Stein \cite{S73}.
\end{Rmk}


\subsection{Parahoric induction}\label{SS:ParaInd}


In this subsection we review Dat's parahoric induction \cite{D09} and a generalization due to \cite{CMO19}, and prove a few new results. For this subsection we let $G$ denote an arbitrary profinite group instead of our Chevalley group, and let $\mathcal{M}(G)$ be the category of smooth representations of $G$.

Let us recall a couple of definitions from \cite{D09,CMO19}.

\begin{Def}
	The group $G$ is said to have a {\it virtual Iwahori decomposition} $(U,T,V)$ if $U,T,V\subseteq G$ are closed subgroups such that 
	\begin{itemize}
		\item $T$ normalizes $U$ and $V$,
		\item the multiplication map $U\times T\times V\rightarrow G$ is an open embedding,
		\item $G$ contains arbitrarily small open normal subgroups $K$ such that the multiplication map $(U\cap K)\times (T\cap K)\times (V\cap K)\rightarrow K$ is a homeomorphism.
	\end{itemize}
	A virtual Iwahori decomposition is call an {\it Iwahori decomposition} if the multiplication map $U\times T\times V\rightarrow G$ is a homeomorphism.
\end{Def}

Let $\Hcal(G)$ be the space of locally constant $\mathbb{C}$-valued functions on $G$. The convolution product gives $\Hcal(G)$ the structure of an idempotented $\mathbb{C}$-algebra, and each smooth representation is viewed as an $\Hcal(G)$-module, so that the category $\mathcal{M}(G)$ is identified with the category of smooth non-degenerate $\Hcal(G)$-modules. We write $e_G\in\Hcal(G)$ for the idempotent of the trivial representation of $G$. We then define parahoric induction and restriction from \cite{D09,CMO19} as follows.

\begin{Def}
	Let $(U,T,V)$ be a virtual Iwahori decomposition of a profinite group $G$. Let $\tau$ be a smooth $T$-module and let $\sigma$ be a smooth $G$-module. Let $e_{U}\in\Hcal(U)$ and $e_{V}\in\Hcal(V)$ be the idempotents of the trivial representation of $U$ and $V$ respectively. 
	
	Define functors $i_{U,V}:\mathcal{M}(T)\rightarrow \mathcal{M}(G)$ and $r_{U,V}:\mathcal{M}(G)\rightarrow \mathcal{M}(T)$ by
	\begin{align*}
		i_{U,V}:=&\Hcal(G)e_{U}e_{V}\otimes_{\Hcal(T)}(-);\\	
		r_{U,V}:=&e_{U}e_{V}\Hcal(G)\otimes_{\Hcal(G)}(-).
	\end{align*}
    We may write $i_{U,V}^{G}=i_{U,V}$ and $r_{U,V}^{G}=r_{U,V}$ if we want to emphasize $G$.
\end{Def}

Let us list a couple of important properties of these functors proved in \cite{CMO19,D09}.
\begin{Prop}\label{P:two_functions_are_adjoint}
Let $(U, T, V)$ be a virtual Iwahori decomposition of a profinite group $G$. 
\begin{enumerate}[(1)]
\item There are natural isomorphisms 
\[
i_{U, V}\cong i_{V, U}\qand r_{U, V}\cong r_{V, U}.
\]
\item $i_{U, V}$ is both left and right adjoint to $r_{U, V}$; namely we have
\[
\Hom_{\Hcal(G)}(\sigma, i_{U, V}\tau)\cong \Hom_{\Hcal(T)}(r_{U, V}\sigma, \tau)
\qand
\Hom_{\Hcal(G)}(i_{U, V}\tau, \sigma)\cong\Hom_{\Hcal(T)}(\tau, r_{U,V}\sigma)
\]
for all $\sigma\in\mathcal{M}(G)$ and $\tau\in\mathcal{M}(T)$.
\end{enumerate}
\end{Prop}
\begin{proof}
    This is \cite[Theorem 2.18 (1), (2) and (3)]{CMO19}.
\end{proof}

If $(U, T, V)$ is an Iwahori decomposition (instead of merely virtual), the above two functors have particularly nice properties as follows.
\begin{Prop}\label{P:two_functors_for_Iwahori_decomp}
Suppose that $(U, T, V)$ is an Iwahori decomposition of $G$. Then we have the following:
\begin{enumerate}[(1)]
\item $i_{U, V}$ sends $\Irr(T)$ to $\Irr(G)$, and $r_{U, V}$ sends $\Irr(G)$ to $\Irr(T)\cup\{0\}$.
\item $r_{U, V}\circ i_{U, V}\cong id_{\mathcal{M}(T)}$.\label{P:two_functors_for_Iwahori_decomp2}
\item If $\sigma\in\Irr(G)$ is such that $r_{U, V}\sigma\neq 0$, then $i_{U, V}r_{U, V}\sigma\cong\sigma$.
\item If $\sigma\in\Irr(G)$, then
\[
\dim_{\C}\Hom_T(\sigma^U, \sigma^V)\leq 1.
\]
Further, if this dimension is zero then $r_{U, V}\sigma=0$, and if it is 1, then the $\Hom$ space is generated by the isomorphism $e_Ve_U$ and $r_{U, V}\sigma\cong\sigma^U\cong\sigma^V$.
\end{enumerate}
\end{Prop}
\begin{proof}
This is \cite[Theorem 2.23]{CMO19}.
\end{proof}

The following definition is motivated by the well-known result that we will recall in Lemma \ref{WCompat}.
\begin{Def}
	Suppose that $G$ has two virtual Iwahori decompositions $(U,T,V)$ and $(U^{\prime},T,V^{\prime})$. We call the two virtual decompositions {\it compatible} if the following four identities hold. 
	\begin{align*}
		U=&(U\cap U^{\prime})(U\cap V^{\prime}),\\
		V=&(V\cap \,U^{\prime})(V\cap V^{\prime}),\\
		U^{\prime}=&(U^{\prime}\cap U)(U^{\prime}\cap V),\\
		V^{\prime}=&(V^{\prime}\cap U)(V^{\prime}\cap V).
	\end{align*}
\end{Def}

The following is an important property of compatible decompositions.
\begin{Prop}\label{P:compatible_Iwahori_decomp}
Assume $(U, T, V)$ and $(U', T, V')$ are compatible virtual Iwahori decompositions of $G$. Then we have
\[
i_{U, V}\cong i_{U', V'}\qand r_{U, V}\cong r_{U', V'}.
\]
\end{Prop}
\begin{proof}
    This is \cite[Theorem 2.18 (4)]{CMO19}.
\end{proof}

\begin{Lem}\label{IsoTwist}
	Let $H_1$ and $H_2$ be profinite groups and assume each $H_i$ has a virtual Iwahori decomposition $(U_{i},T_{i},V_{i})$. Let $\phi:H_{1}\rightarrow H_{2}$ be a group isomorphism such that $\phi(U_{1})=U_{2}$, $\phi(T_{1})=T_{2}$, and $\phi(V_{1})=V_{2}$. Let $\tau\in\Irr(T_{1})$. Then as $H_{2}$-modules
	\begin{equation*}
		\,^{\phi}i_{U_1,V_1}\tau\cong i_{\phi(U_1),\phi(V_1)}\,^{\phi}\tau. 
	\end{equation*}
\end{Lem}
\begin{proof}
	Consider the linear map $\,^{\phi}i_{U_1,V_1}\tau\rightarrow i_{\phi(U_1),\phi(V_1)}\,^{\phi}\tau$ defined by $f(g)e_{U_{1}}e_{V_{1}}\otimes v\mapsto f(\phi^{-1}(g))e_{U_{2}}e_{V_{2}}\otimes v$, where $g\in G$ and $v\in \tau$. It is a well defined isomorphism of $H_{2}$-modules. 
\end{proof}

\begin{Prop}\label{TwistAut}
	Let $G$ be a profinite group with an Iwahori decomposition $(U,T,V)$. Let $\phi\in\Aut(G)$ be such that $\phi(T)=T$. Let $\tau\in\Irr(T)$.
	Suppose that 
	\begin{itemize}
		\item $\,^{\phi}\tau\cong \tau$, 
		\item the two Iwahori factorizations $(U,T,V)$ and $(\phi(U),T,\phi(V))$ are compatible.
	\end{itemize} 
	Then
	\begin{equation*}
		\,^{\phi}i_{U,V}\tau\cong i_{U,V}\tau.
	\end{equation*}
\end{Prop}

\begin{proof}
	Lemma \ref{IsoTwist} implies $\,^{\phi}i_{U,V}\tau\cong i_{\phi(U),\phi(V)}\,^{\phi}\tau$. Since $\,^{\phi}\tau\cong \tau$ we have $i_{\phi(U),\phi(V)}\,^{\phi}\tau\cong i_{\phi(U),\phi(V)}\tau$. But since $(U, T, V)$ and  $(\phi(U),T,\phi(V))$ are compatible by our assumption, the proposition follows from Proposition \ref{P:compatible_Iwahori_decomp}.
\end{proof}

\begin{Lem}\label{L:ResDimBound}
    Let $G$ be a profinite group with an Iwahori decomposition $(U,T,V)$. Let $\pi$ be a finite-dimensional representation of $G$. Then
    \begin{equation*}
        \dim r_{U,V}\pi\leq \dim \pi^{V}.
    \end{equation*}
\end{Lem}

\begin{proof}
    Since $\pi$ is finite dimensional, there exists a normal subgroup $N\subseteq G$ such that $N$ acts trivially on $\pi$. Note that for any $f\in\mathcal{H}$ we have $e_{N}*f=f*e_{N}$. By definition $r_{U,V}\pi$ is spanned by vectors of the form $e_{U}e_{V}f\otimes v$, where $f\in \mathcal{H}(G)$ and $v\in\pi$. But $e_{U}e_{V}f\otimes v=e_{U}e_{V}e_{N}\otimes e_{V}f\cdot v$. Now the result follows because $e_{V}f\cdot v\in \pi^{V}$.
\end{proof}

\begin{Lem}\label{L:IndRes2x}
    Let $G$ be a profinite group with an Iwahori decomposition $(U,T,V)$. Let $H\subseteq G$ be an open subgroup such that $(U^{\prime}=H\cap U,T^{\prime}=H\cap T,V)$ is an Iwahori decomposition. Let $\tau$ be an irreducible representation of $T$.
    
    The $T^{\prime}$-module homomorphism $\phi:r_{U^{\prime},V}^{H}i_{U,V}^{G}\tau\rightarrow r_{U,V}^{G}i_{U,V}^{G}\tau$ defined by
    \begin{equation*}
        e_{U^{\prime}}e_{V}f\otimes_{\mathcal{H}(H)} v \mapsto e_{U}e_{V}f\otimes_{\mathcal{H}(G)} v
    \end{equation*}
    is a well defined isomorphism of $T^{\prime}$-modules.
\end{Lem}  

\begin{proof}

    To begin we show that the map $\phi$ is well defined by constructing it directly. We do this in several steps.

    Because $H\subseteq G$ is of finite index, extension by $0$ defines an $(H,\mathcal{H}(H))$-bimodule inclusion $\mathcal{H}(H)\hookrightarrow \mathcal{H}(G)$. (Here we are fixing a Haar measure on $G$ and restricting it to a Haar measure on $H$.) Thus the map $A:e_{U^{\prime}}e_{V}\mathcal{H}(H)\rightarrow e_{U}e_{V}\mathcal{H}(G)$ defined by $e_{U^{\prime}}e_{V} f\mapsto e_{U}e_{V}f$ (i.e. left convolution by $e_{U}$) is a $(T^{\prime},\mathcal{H}(H))$-bimodule homomorphism. We tensor by $\otimes_{H}i_{U,V}^{G}\tau$ and get a map $A\otimes id_{i_{U,V}^{G}\tau}:e_{U^{\prime}}e_{V}\mathcal{H}(H)\otimes_{\mathcal{H}(H)}i_{U,V}^{G}\tau\rightarrow e_{U}e_{V}\mathcal{H}(G)\otimes_{\mathcal{H}(H)}i_{U,V}^{G}\tau$ by functoriality. Finally since the right $\mathcal{H}(G)$-module action on $e_{U}e_{V}\mathcal{H}(G)$ and the left $\mathcal{H}(G)$ action on $i_{U,V}^{G}\tau$ extends the respective $\mathcal{H}(H)$ action, we have a surjective map $B:e_{U}e_{V}\mathcal{H}(G)\otimes_{\mathcal{H}(H)}i_{U,V}^{G}\tau\twoheadrightarrow e_{U}e_{V}\mathcal{H}(G)\otimes_{\mathcal{H}(G)}i_{U,V}^{G}\tau$ defined by $e_{U}e_{V}f\otimes_{\mathcal{H}(H)} v\mapsto e_{U}e_{V}f\otimes_{\mathcal{H}(G)} v$. Then $\phi=B \circ(A\otimes_{\mathcal{H}(H)} id_{i_{U,V}^{G}\tau})$.

    Let $\Gamma$ be the kernel of $i_{U,V}^{G}\tau$, so that $\Gamma$ is a normal subgroup which acts trivially on $i_{U,V}^{G}\tau$. Note that in $\mathcal{H}(G)$ we have $e_{\Gamma\cap H}*e_{\Gamma}=e_{\Gamma}$. Thus under the map $\phi$ we have
    \begin{equation*}
        e_{U^{\prime}}e_{V}e_{\Gamma\cap H}\otimes_{\mathcal{H}(H)} v\mapsto e_{U}e_{V}e_{\Gamma}\otimes_{\mathcal{H}(G)} v.
    \end{equation*}

    Since $r_{U,V}^{G}i_{U,V}^{G}\tau$ is spanned by vectors of the form $e_{U}e_{V}e_{\Gamma}\otimes v$, where $v\in i_{U,V}^{G}\tau$, it follows that the map $\phi$ is surjective. By Proposition \ref{P:two_functors_for_Iwahori_decomp} \ref{P:two_functors_for_Iwahori_decomp2} $r_{U,V}^{G}i_{U,V}^{G}\tau\cong \tau$, so the map is nonzero. By Lemma \ref{L:ResDimBound} and Proposition \ref{P:two_functors_for_Iwahori_decomp}, $\dim r_{U^{\prime},V}^{H}i_{U,V}^{G}\tau\leq \dim\tau$. Thus the map is an isomorphism of $T^{\prime}$-modules.

\end{proof}

The next proposition provides a sufficient condition for when parahoric induction is preserved under restriction. 

\begin{Prop}\label{ResGamma1}
	Let $G$ be a profinite group with an Iwahori decomposition $(U,T,V)$. Let $H\subseteq G$ be an open subgroup with an Iwahori decomposition $(U^{\prime}=H\cap U,T^{\prime}=H\cap T,V)$. Suppose that 
	\begin{itemize}
		\item $\tau$ is a representation of $T$ such that $\tau$ restricted to $T^{\prime}$ remains irreducible;
		\item $(i_{U,V}\tau)^{U^{\prime}}=(i_{U,V}\tau)^{U}$.
	\end{itemize}
	
	Then as $H$-modules
    \[
    (i_{U,V}\tau)|_H\cong  i_{U^{\prime},V}\tau.
    \]

\end{Prop}

\begin{proof}
	To begin we show that $\Hom_{H}(i_{U^{\prime},V}\tau,i_{U,V}\tau)\cong \mathbb{C}$. By Proposition \ref{P:two_functions_are_adjoint},
	\begin{equation*}
		\Hom_{H}(i_{U^{\prime},V}\tau,\, i_{U,V}\tau)\cong \Hom_{T^{\prime}}(\tau,\, r_{U^{\prime},V}i_{U,V}\tau).
	\end{equation*}
	
	By Lemma \ref{L:IndRes2x}, as $T^{\prime}$-modules $r_{U^{\prime},V}i_{U,V}\tau\cong r_{U,V}i_{U,V}\tau$. By Proposition \ref{P:two_functors_for_Iwahori_decomp} \ref{P:two_functors_for_Iwahori_decomp2}, we know $r_{U,V}i_{U,V}\tau\cong \tau$ as $T$-modules. This is also true as $T^{\prime}$-modules since $\tau$ restricted to $T^{\prime}$ remains irreducible. Thus $r_{U^{\prime},V}i_{U,V}\tau\cong \tau$ as $T^{\prime}$-modules. Therefore $\Hom_{H}(i_{U^{\prime},V}\tau,\, i_{U,V}\tau)\cong \mathbb{C}$.
	
	Now let $\phi\in \Hom_{H}(i_{U^{\prime},V}\tau,\, i_{U,V}\tau)$ be nonzero. We want to show that $\phi$ is an $H$-module isomorphism. The map $\phi$ is injective because $i_{U^{\prime},V}\tau$ is an irreducible $H$-module and $\phi\neq 0$. Thus it remains to show that $\phi$ is surjective.
	
	By the Iwahori factorization, the space $i_{U,V}\tau$ is generated as a $V$-module by elements of the form $e_{U}e_{V}\otimes w$, where $w\in \tau$. So, in particular $i_{U,V}\tau$ is generated as an $H$-module by the subspace $(i_{U,V}\tau)^{U}$. By assumption, $(i_{U,V}\tau)^{U}=(_{U,V}\tau)^{U^{\prime}}$, so $i_{U,V}\tau$ is $H$-generated by $(i_{U,V}\tau)^{U^{\prime}}$.
	
	We apply the $U^{\prime}$-invariants functor to the injection
	\begin{equation*}
		i_{U^{\prime},V}\tau \xhookrightarrow{\;\phi\;} i_{U,V}\tau
	\end{equation*}
	to get 
	\begin{equation*}
		\left(i_{U^{\prime},V}\tau\right)^{U^{\prime}}\xhookrightarrow{\;\phi^{U^{\prime}}\;} \left(i_{U,V}\tau\right)^{U^{\prime}}=\left(i_{U,V}\tau\right)^{U}.
	\end{equation*}
	By Proposition \ref{P:two_functors_for_Iwahori_decomp} (2) and (4), this injection must be an isomorphism of $T^{\prime}$-modules, as the domain and codomain are isomorphic to $\tau$. Thus the image of $\phi$ contains $\left(i_{U,V}\tau\right)^{U}$. Therefore as an $H$-module the image of $\phi$ generates $i_{U,V}\tau$. But $\phi$ is an $H$-modules homomorphism, so $\phi$ is surjective.
\end{proof}

If $G$ is finite, it is convenient to make the identification
\[
\Hcal(G)\cong\C[G], \quad \delta_g\mapsto g,
\]
of algebras, where $\C[G]$ is the group algebra of $G$ and $\delta_g$ is the delta function at $g\in G$. Under this identification, we have
\begin{equation}\label{E:e_H_idempotent}
e_H=\frac{1}{|H|}\sum_{h\in H}h
\end{equation}
for each subgroup $H\subseteq G$. Note that for each $g\in G$ such that $gHg^{-1}=H$ we have $ge_H=e_Hg$. Also if $h\in H$, then $he_H=e_Hh=e_H$; namely $e_H$ ``absorbs" $h$. Note also that the third condition for virtual Iwahori decomposition is trivial because one can just choose $K$ to be the trivial group.

Let us prove a few facts for finite groups.
\begin{Lem}\label{L:spanning_vectors}
Assume $G$ is a finite group with an Iwahori decomposition $(U, T, V)$ and $\tau\in\mathcal{M}(T)$. Then $i_{U, V}\tau$ (as a $\C$-vector space) is spanned by the vectors of the form
\[
ye_Ue_V\otimes v\qquad (y\in V,\; v\in\tau).
\]
\end{Lem}
\begin{proof}
By definition, the space $i_{U, V}\tau$ is generated by the vectors of the form $ge_Ue_V\otimes v$ for $g\in G$ and $v\in\tau$. Now, by the Iwahori decomposition, each $g^{-1}$ is written as $g^{-1}=x t y$ for some $x\in U, t\in T, y\in V$, so that $g=y^{-1}t^{-1} x^{-1}$. But $x^{-1}e_U=e_U$ by absorbing $x^{-1}$, and 
\[
t^{-1}e_Ue_V\otimes v=e_Ue_Vt^{-1}\otimes v=e_Ue_V\otimes \tau(t^{-1})v.
\]
Hence we can write $ge_Ue_V\otimes v=y^{-1}e_Ue_V\otimes \tau(t^{-1})v$. The lemma follows.
\end{proof}

The final lemma in this subsection relates parahoric induction with induction in stages where one step involves standard induction.

\begin{Lem}\label{LHoms}
	Let $G$ be a finite group with an Iwahori decomposition $(U,T,V)$. Let $H\subseteq G$ be a subgroup with an Iwahori decomposition $(U^{\prime}=H\cap U,T,V)$. Let $\tau$ be a representation of $T$. Then
	\begin{equation*}
		\Hom_{G}(\Ind_{H}^{G}i_{V,U^{\prime}}\tau,\, i_{V,U}\tau)\neq 0.
	\end{equation*}
\end{Lem}

\begin{proof} The surjective $(G,T)$-bimodule map $\Hcal(G)\otimes_{H}\Hcal(H)e_{V}e_{U^{\prime}} \rightarrow \Hcal(G)e_{V}e_{U}$ defined by $g\otimes_{H} he_{V}e_{U^{\prime}}\mapsto gh e_{V}e_{U}$ induces a surjective $G$-module map 
	\begin{equation*}
		\Hcal(G)\otimes_{H}\Hcal(H)e_{V}e_{U^{\prime}}\otimes_{T}\tau \longrightarrow \Hcal(G)e_{V}e_{U}\otimes_{T}\tau=i_{V,U}\tau.
	\end{equation*}
	Since $i_{V,U}\tau\neq 0$ by \cite[Theorem 2.18 (6)]{CMO19}, the result follows.
\end{proof}



\subsection{Construction of Iwahori types}\label{IwahoriType}


Now, we apply the above described theory of parahoric induction to our pseudo-spherical representations $\tau$. The basic idea is that we will construct a representation of the Iwahori subgroup $\It$ by ``parahorically inducing" $\tau$ to $\It$, and prove that this representation behaves well with respect to the Weyl group by using the machinery of parahoric induction. The main results are stated in Proposition \ref{IRep}. (We will later prove in Theorem \ref{IType} that this representation is an Iwahori type of $\Gt$ in the sense of Bushnell-Kutzko \cite[(4.1) Definition]{BK98}.)

Recall from \S \ref{SSecSplittings} that we have a splitting $\s:\Gamma(2e)\to\Kt$ whose image is normal in $\Kt$. The kernel of the representation we will construct will contain $\s(\Gamma(2e))$. Hence we almost always work with ``representations modulo ${\Gamma(2e)}$", so that we view these representations as representations of subgroups of the finite quotient $\Kft=\Kt\slash \s(\Gamma(2e))$. 

For this purpose, let us introduce the following notational convention. For each subgroup $H\subseteq K$, we denote its image in $K\slash\Gamma(2e)$ by the corresponding sans-serif font $\Hf$. To be precise
\[
\Hf:=H\slash (\Gamma(2e)\cap H)\subseteq K\slash \Gamma(2e)=\Kf.
\]
Further, denote the image of $\Ht$ in $\Kft$ by $\widebar{\Hf}$, so that we have
\[
\begin{tikzcd}
    \Kt\arrow[r, phantom, "\supseteq"]\ar[d]&\Ht\arrow[r, two heads]\ar[d]&\widebar{\Hf}\arrow[r, phantom, "\subseteq"]\ar[d]&\makebox[0pt][l]{$\Kft=\Kt\slash\s(\Gamma(2e))$}\ar[d, shift left=.5em]\\
    K\arrow[r, phantom, "\supseteq"]&H\arrow[r, two heads]&\Hf\arrow[r, phantom, "\subseteq"]&\makebox[0pt][l]{$\Kf=K\slash\Gamma(2e)$\rlap{\; .}}
\end{tikzcd}
\]
In particular, we often use the groups
\[
\Uf_{j}^{\pm}\qand  \Uf_{\alpha, j}\qquad (j\in\Z_{\geq 0}),
\]
where the corresponding $U_j^{\pm}$ and $U_{\alpha, j}$ are recalled from \eqref{E:notation_U_alpha_j}. We often set
\[
\Uf_j=\Uf_j^{+}.
\]
 We are particularly interested in the case when $j=0, 1$.
These groups are usually viewed as subgroups of $\Kft$ via the splittings $\s^{\pm}:U_0^{\pm}\to\Kt$, so that we have
\[
\s^{\pm}:\Uf_0^{\pm}\xhookrightarrow{\qquad}\Kft.
\]
The images of the elements $\xt_{\alpha}(u)$ and $\htt_{\alpha}(t)$ in $\Kft$ are to be denoted by the same symbol.

We should mention that the three groups 
\[
\widebar{\Gammaf}(1),\quad \widebar{\If}=\widebar{\Gammaf}_0(1),\quad \widebar{\If}_2=\widebar{\Gammaf}_0(2)
\]
have Iwahori decompositions
\[
(\Uf_1, \widebar{\Tf}_1, \Uf_1^-),\quad (\Uf_0, \widebar{\Tf}_0, \Uf_1^-),\quad (\Uf_0, \widebar{\Tf}_0, \Uf_2^-),
\]
respectively. Also the group $\Kft$ has a virtual Iwahori decomposition $(\Uf_0,\widebar{\Tf}_0, \Uf_0^-)$.

\quad

Now, let $\tau$ be a pseudo-spherical representation of $\Tt_0$. Since $T_{2e}\subseteq T_{R}$ and $T_R$ is in the kernel of $\tau$, we view $\tau$ as a representation of 
\[
\widebar{\Tf}_0=\Tt_0\slash T_{2e}=\Tt_0\slash \left(\Tt_0\cap \s(\Gamma(2e))\right).
\]
We continue to write $\tau$ for this representation of $\widebar{{\Tf}}_0$. 

We then parahorically induce $\tau$ to a representation of $\widebar{\If}$; namely we set
\[
\sigma:=i_{{\Uf},{\Uf}^{-}_{1}}^{\Ift}\tau,
\]
which is an irreducible representation of $\Ift$ by Proposition \ref{P:two_functors_for_Iwahori_decomp}. We also consider 
\[
\sigma_{1}:=i_{{\Uf}_{1},{\Uf}^{-}_{1}}^{\Gammaft(1)}\tau,
\]
which is an irreducible representation of $\widebar{\Gammaf}(1)$ by Proposition \ref{P:two_functors_for_Iwahori_decomp}. Note that $\widebar{\Gammaf}(1)\subseteq\widebar{\If}$. One of the main results in this section is to show that
\[
\sigma|_{\widebar{\Gammaf}(1)}\cong\sigma_1,
\]
which will be proven in Proposition \ref{IRep} among other things.

We begin with a few technical lemmas.

\begin{Lem}\label{WCompat}
	Let $w\in W$. Then $({\Uf}_{0},{\Tf}_{0},{\Uf}^{-}_{0})$ and $(\,^{w}{\Uf}_{0},{\Tf}_{0},\,^{w}{\Uf}^{-}_{0})$ are compatible virtual Iwahori decompositions of ${\Kf}$. Specifically, 
	\begin{align*}
		{\Uf}_0=&({\Uf}_0\cap \,^{w}{\Uf}_0)({\Uf}_0\cap \,^{w}{\Uf}_0^{-}),\\
		{\Uf}_0^{-}=&({\Uf}_0^{-}\cap \,^{w}{\Uf}_0)({\Uf}_0^{-}\cap \,^{w}{\Uf}_0^{-}),\\
		\,^{w}{\Uf}_0=&(\,^{w}{\Uf}_0\cap {\Uf}_0)(\,^{w}{\Uf}_0\cap {\Uf}_0^{-}),\\
		\,^{w}{\Uf}_0^{-}=&(\,^{w}{\Uf}_0^{-}\cap {\Uf}_0)(\,^{w}{\Uf}_0^{-}\cap {\Uf}_0^{-}).\\
	\end{align*}
	The same also holds for ${\Gammaf}(1)$ with the Iwahori factorization $({\Uf}_{1},{\Tf}_{1},{\Uf}^{-}_{1})$.
\end{Lem}

\begin{proof}
	This follows because ${\Uf}_0\cap \,^{w}{\Uf}_0=\prod_{\alpha\in \Phi^{+}\cap w\Phi^{+}}{\Uf}_{\alpha,0}$, along with similar identities for the other intersections.
\end{proof}

\begin{Lem}\label{LemVarIndRootMove1}
	Let $\alpha\in\Delta$ and $u\in\mathcal{O}\smallsetminus\{0\}$. Then in the group algebra $\mathbb{C}[\;\Ift\;]$
	\begin{equation*}
		e_{{\Uf}_{-\alpha,1}}\xt_{\alpha}(u)e_{{\Uf}_{1}^{-}}e_{{\Uf}_{0}}=e_{{\Uf}_{1}^{-}}\frac{1}{\left|\varpi\Ocal\slash 4\Ocal\right|}\sum_{v\in \varpi\mathcal{O}/4\mathcal{O}}(1+uv,-u)\htt_{\alpha}(1+uv) e_{{\Uf}_{0}}.
	\end{equation*}
If $u\in \Ocal^{\times}$, note that $\frac{1}{\left|\varpi\Ocal\slash 4\Ocal\right|}\sum_{v\in \varpi\mathcal{O}/4\mathcal{O}}(1+uv,-u)\htt_{\alpha}(1+uv)=X_{\alpha,-u,1}$. (Recall \S \ref{SSSqRoot}.)
\end{Lem}

\begin{proof}
Let ${\Uf}_{\Phi^{-}_{\alpha},1}\subseteq {\Kf}$ be the subgroup generated by the subgroups ${\Uf}_{\beta,1}$, where $\beta\in \Phi^{-}\smallsetminus\{-\alpha\}$. Note that the multiplication map defines a set bijection ${\Uf}_{\Phi^{-}_{\alpha},1}\times {\Uf}_{-\alpha,1}\cong {\Uf}_{1}^{-}$, and the order of the factors can be reversed. Thus one can readily see from \eqref{E:e_H_idempotent} that the idempotent $e_{{\Uf}_{1}^{-}}$ can be factored as
\[
	e_{{\Uf}_{1}^{-}}=e_{{\Uf}_{\Phi^{-}_{\alpha},1}}e_{{\Uf}_{-\alpha,1}}=e_{{\Uf}_{-\alpha,1}}e_{{\Uf}_{\Phi^{-}_{\alpha},1}}.
\]
Hence we have
	\begin{equation*}
		e_{{\Uf}_{-\alpha,1}}\xt_{\alpha}(u)e_{{\Uf}_{\Phi^{-}_{\alpha},1}}=e_{{\Uf}_{-\alpha,1}}e_{{\Uf}_{\Phi^{-}_{\alpha},1}}\xt_{\alpha}(u)=e_{{\Uf}_{1}^{-}}\xt_{\alpha}(u),
	\end{equation*}
    where we used that $\xt_{\alpha}(u)$ normalizes ${\Uf}_{\Phi^{-}_{\alpha},1}$.
	
	Using Proposition \ref{PMRootSwap} we have
	\begin{equation*}
		\xt_{\alpha}(u)e_{{\Uf}_{-\alpha,1}}=\frac{1}{\left|\varpi\Ocal\slash 4\Ocal\right|}\sum_{v\in \varpi\mathcal{O}/4\mathcal{O}}(1+uv,-u)\xt_{-\alpha}(\frac{v}{1+uv})\htt_{\alpha}(1+uv) \xt_{\alpha}(\frac{u}{1+uv}).
	\end{equation*}
	Now the result follows because 
	\begin{align*}
		&e_{{\Uf}_{-\alpha,1}}\sum_{v\in \varpi\mathcal{O}/4\mathcal{O}}(1+uv,-u)\xt_{-\alpha}(\frac{v}{1+uv})\htt_{\alpha}(1+uv) \xt_{\alpha}(\frac{u}{1+uv})e_{{\Uf}_{0}}\\
		&=e_{{\Uf}_{-\alpha,1}}\sum_{v\in \varpi\mathcal{O}/4\mathcal{O}}(1+uv,-u)\htt_{\alpha}(1+uv)e_{{\Uf}_{0}}
	\end{align*}
    by letting $e_{{\Uf}_{-\alpha,1}}$ absorb $\xt_{-\alpha}(\frac{v}{1+uv})$.
\end{proof}

\begin{Lem}\label{LemVarIndRootMove2}
	Let $\alpha\in\Delta$, $u\in\mathcal{O}$ and $w\in \tau$. Then in the $\Ift$-module $\sigma=i_{{\Uf}_{1}^{-},{\Uf}_{0}}\tau$
	\begin{equation*}
		e_{{\Uf}_{1}}\xt_{\alpha}(u)e_{{\Uf}_{1}^{-}}e_{{\Uf}_{0}}\otimes w = e_{{\Uf}_{1}}e_{{\Uf}_{1}^{-}}e_{{\Uf}_{0}}\otimes w,
	\end{equation*}
    where we recall that the tensor product is over $\C[\Tft_0]$.
\end{Lem}
\begin{proof}
        This is trivial if $u\in\varpi\mathcal{O}$. So assume $u\in\mathcal{O}^{\times}$. Let $X=X_{-\alpha,u,1}=X_{\alpha, -u, 1}\in\mathbb{C}[\Tft_{0}]$ be as in \eqref{E:X_alpha_u_j_def}. By Corollary \ref{XInvert}, the operator $\tau(X)$ is invertible. Thus
	\begin{align*}
		&e_{{\Uf}_{1}}\xt_{\alpha}(u)e_{{\Uf}_{1}^{-}}e_{{\Uf}_{0}}\otimes w\\
        = & e_{{\Uf}_{1}}\xt_{\alpha}(u)e_{{\Uf}_{1}^{-}}e_{{\Uf}_{0}}X\otimes \tau(X)^{-1}w\\
        = & e_{{\Uf}_{1}}\xt_{\alpha}(u)Xe_{{\Uf}_{1}^{-}}e_{{\Uf}_{0}}\otimes \tau(X)^{-1}w\\
		= & e_{{\Uf}_{1}}\frac{1}{\left|\varpi\Ocal\slash 4\Ocal\right|}\sum_{v\in \varpi\Ocal/4\Ocal}(1+uv,u)\xt_{\alpha}(u)\htt_{-\alpha}(1+uv)e_{{\Uf}_{1}^{-}}e_{{\Uf}_{0}}\otimes \tau(X)^{-1}w,
	\end{align*}
    where for the second equality we used $e_{{\Uf}_{1}^{-}}e_{{\Uf}_{0}}X=Xe_{{\Uf}_{1}^{-}}e_{{\Uf}_{0}}$, noting that $\htt_{-\alpha}(1+uv)$ normalizes both $\Uf_1^-$ and $\Uf_0$. 
    
    Note that
    \[
     e_{\Uf_1}\xt_{\alpha}(u)=e_{\Uf_1}\xt_{\alpha}(\frac{u}{1+uv}),
    \]
    because for $v\in\varpi\Ocal$ we have $u\equiv \frac{u}{1+uv}$ modulo $\varpi$. Also
    \[
    e_{{\Uf}_{1}^{-}}=\xt_{-\alpha}(\frac{v}{1+uv})e_{{\Uf}_{1}^{-}}
    \]
    by absorbing $\xt_{-\alpha}(\frac{v}{1+uv})$. Hence we can write
	\begin{align*}
		&e_{{\Uf}_{1}}\sum_{v\in \varpi\Ocal/4\Ocal}(1+uv,u)\xt_{\alpha}(u)\htt_{-\alpha}(1+uv)e_{{\Uf}_{1}^{-}}e_{{\Uf}_{0}}\otimes \tau(X)^{-1}w \\
		= & e_{{\Uf}_{1}}\sum_{v}(1+uv,u)\xt_{\alpha}(\frac{u}{1+uv})\htt_{-\alpha}(1+uv)\xt_{-\alpha}(\frac{v}{1+uv})e_{{\Uf}_{1}^{-}}e_{{\Uf}_{0}}\otimes \tau(X)^{-1}w\\
        = & e_{{\Uf}_{1}}\left(\sum_{v}\xt_{-\alpha}(v)\right)\xt_{\alpha}(u)e_{{\Uf}_{1}^{-}}e_{{\Uf}_{0}}\otimes \tau(X)^{-1}w\quad (\text{Proposition \ref{PMRootSwap}})\\
		= & \left|\varpi\Ocal\slash 4\Ocal\right|e_{{\Uf}_{1}}e_{{\Uf}_{-\alpha,1}}\xt_{\alpha}(u)e_{{\Uf}_{1}^{-}}e_{{\Uf}_{0}}\otimes \tau(X)^{-1}w.
	\end{align*}
	
	Now we apply Lemmas \ref{LemVarIndRootMove1} and \ref{Xpmalpha} to get 
		\begin{equation*}
		e_{{\Uf}_{1}}e_{{\Uf}_{-\alpha,1}}\xt_{\alpha}(u)e_{{\Uf}_{1}^{-}}e_{{\Uf}_{0}}\otimes \tau(X)^{-1}w = e_{{\Uf}_{1}}e_{{\Uf}_{1}^{-}}e_{{\Uf}_{0}}\otimes w.
	\end{equation*}
    The lemma is proven.
\end{proof}

This lemma implies
\begin{Prop}\label{UInvar}
	Let $\alpha\in \Delta$. Then we have 
	\begin{equation*}
		\sigma^{{\Uf_0}}=\sigma^{{\Uf_0}\cap \,^{w_{\alpha}}\Ift}=\sigma^{{\Uf}_{1}}.
	\end{equation*}
\end{Prop}
\begin{proof}
Since $\Uf_0\supseteq (\Uf_0\cap \,^{w_{\alpha}}\Ift) \supseteq \Uf_1$, the inclusions $\sigma^{{\Uf_{0}}}\subseteq\sigma^{{\Uf_{0}}\cap \,^{w_{\alpha}}\Ift}\subseteq \sigma^{{\Uf}_{1}}$ are immediate. It remains to show $\sigma^{{\Uf}_{1}}\subseteq \sigma^{{\Uf_0}}$. Note that $\sigma=i_{\Uf_0, \Uf_1^-}\tau\cong i_{{\Uf}^{-}_{1},{\Uf_0}}\tau$ by Proposition \ref{P:two_functions_are_adjoint} (1). By the Iwahori factorization and Lemma \ref{L:spanning_vectors}, $(i_{{\Uf}^{-}_{1},{\Uf_0}}\tau)^{{\Uf}_{1}}$  is spanned by elements of the form $e_{{\Uf}_{1}}u e_{{\Uf}_{1}^{-}}e_{{\Uf}_{0}}\otimes w$, where $u\in {\Uf}_{0}$ and $w\in \tau$. We also know that ${\Uf}_{0}$ is generated by the elements of the form $x_{\alpha}(u)$, where $\alpha\in \Delta$ and $u\in \mathcal{O}$. The result follows by Lemma \ref{LemVarIndRootMove2}, induction and the identity
	\begin{equation*}
		e_{{\Uf}_{1}}xx_{\alpha}(u)e_{{\Uf}_{1}^{-}}e_{{\Uf}_{0}}\otimes w =xe_{{\Uf}_{1}}x_{\alpha}(u)e_{{\Uf}_{1}^{-}}e_{{\Uf}_{0}}\otimes w
	\end{equation*}
	for $x\in {\Uf}_{0}$.
\end{proof}

Now we can state and prove the main theorem of this subsection. Let 
\[
\sigma_{\alpha}:=i^{\Ift\cap\,^{w_{\alpha}}\Ift}_{{\Uf}_{0}\cap \,^{w_{\alpha}}\Ift,{\Uf}^{-}_{1}}\tau,
\]
and recall that 
\[
\sigma=i_{{\Uf_{0}},{\Uf}^{-}_{1}}^{\Ift}\tau\qand \sigma_{1}=i_{{\Uf}_{1},{\Uf}^{-}_{1}}^{\Gammaft(1)}\tau,
\]
all of which are irreducible representations of $\Ift\cap\,^{w_{\alpha}}\Ift$, $\widebar{\If}$ and $\widebar{\Gammaf}(1)$, respectively, by Proposition \ref{P:two_functors_for_Iwahori_decomp}.

\begin{Prop}\label{IRep} Let $\alpha\in\Delta$.
	\begin{enumerate}[(1)]
		\item The $\Gammaft(1)$-module $\sigma_{1}$ is $W$-invariant, \ie
        $^w\sigma_1\cong\sigma_1$ for all $w\in W$.\label{sigWInvar}
		\item As $\Gammaft(1)$-modules, $\sigma_{1}\cong\sigma|_{\Gammaft(1)}$.\label{ResSig}
		\item As $\Ift\cap\,^{w_{\alpha}}\Ift$-modules, $\sigma_{\alpha}\cong\sigma|_{\Ift\cap\,^{w_{\alpha}}\Ift}$.\label{ResSigAlpha}
		\item As $\Ift\cap\,^{w_{\alpha}}\Ift$-modules , $\,^{w_{\alpha}}\sigma_{\alpha}\cong\sigma_{\alpha}$.\label{SigAWInv}
	\end{enumerate}
\end{Prop}

\begin{proof}	
	\ref{sigWInvar} and \ref{SigAWInv} follow from Proposition \ref{TwistAut} by setting $\phi$ to be the conjugation by the Weyl group element. Indeed, the two conditions of Proposition \ref{TwistAut} are satisfied by the $W$-invariance of $\tau$ and by Lemma \ref{WCompat}.
	
	\ref{ResSig} and \ref{ResSigAlpha} follow from Proposition \ref{ResGamma1}. Indeed, the two conditions of Proposition \ref{ResGamma1} are satisfied by Lemma \ref{RedtoU1} and Proposition \ref{UInvar}.
\end{proof}


\subsection{Extending $\sigma$ to $\Kft$}\label{S:KType}

Next we show that the $W$-invariance of the $\Gammaft(1)$-module $\sigma_{1}$ allows us to prove
\begin{Thm}\label{KRep}
	The $\widebar{\Gammaf}(1)$-module $\sigma_{1}$ extends uniquely to a $\Kft$-module.
\end{Thm}

Since we know that $\sigma_1$ is $W$-invariant, this is actually almost immediate by using the following known results on the simply-connected simply-laced Chevalley groups over finite fields of characteristic $2$, except for a few low rank cases.
\begin{Prop}\label{SchurMult}
	Let $\kappa$ be a finite field of characteristic $2$. Let $G_\kappa$ be the $\kappa$-points of a simply-connected simply-laced Chevalley group defined over $\kappa$ with an irreducible root system. Then $G_\kappa$ is centrally closed except the following five cases:
\begin{equation}\label{E:exceptional_low_rank_cases}
\SL_{2}(\mathbb{F}_{4}),\quad \SL_{3}(\mathbb{F}_{2}),\quad \SL_{3}(\mathbb{F}_{4}),\quad \SL_{4}(\mathbb{F}_{2}),\quad D_{4}(\mathbb{F}_{2}). 
\end{equation}
Hence except these five cases, $G_\kappa$ has no nontrivial double cover. 
\end{Prop}
\begin{proof}
This follows from \cite[Corollary 2 and Remark (b), page 58]{S16}. 
\end{proof}

\begin{Rmk}
The Schur multiplier of the above five groups can be found in \cite{S81}. Note that \cite[Theorem 1.1]{S81} contains a typo. The group $A_{3}(2)\cong \SL_{4}(\mathbb{F}_{2})\cong A_{8}$ should be listed under (a) in that theorem. Also note other isomorphisms $\SL_{2}(\mathbb{F}_{4})\cong A_{5}$ and $\SL_{3}(\mathbb{F}_{2})\cong \mathrm{PSL}_{3}(\mathbb{F}_{7})$.
\end{Rmk}

Let $\kappa=\Ocal/\varpi\Ocal$ and let $G_{\kappa}$ be the group of $\kappa$-points of the simply-connected Chevalley group over $\kappa$ of the same type as $G$. Note that $G_{\kappa}\cong K/\Gamma(1)\cong{\Kf}/{\Gammaf}(1)$, where the isomorphisms are induced by the natural maps on Chevalley generators. We also write $T_{\kappa}\subseteq B_{\kappa}\subseteq G_{\kappa}$ for the $\kappa$-points of a maximal torus and Borel subgroup of the Chevalley group. We choose these groups so that they are compatible with the Chevalley pinning of $G$.

With this said, we can give the proof of the above theorem
\begin{proof}[Proof of Theorem \ref{KRep}]
 Since we are assuming that $\Phi\neq A_{1}$ or $D_2$ and it is simply-laced, we can see that $\Kft$ is perfect from the relation \eqref{MR2}.  Hence the uniqueness follows.
    
	From Proposition \ref{IRep} \ref{sigWInvar} and \ref{ResSig}, we know that the irreducible $\Gammaft(1)$-module $\sigma_{1}$ is invariant under conjugation by $\Kft$, because $\Kft$ is generated by $\Ift$ and $W$. Hence, if $G_\kappa$ is not among the five exceptional cases of \eqref{E:exceptional_low_rank_cases}, one can just invoke \cite[(11.45) Proposition]{CR81}, which immediately proves the theorem.
    
    It remains to show that the representation extends for the five exceptional cases of \eqref{E:exceptional_low_rank_cases}. Our argument for these exceptional cases is inspired by Savin \cite[Proposition 4.1]{S12}. The basic idea is to embed $G$ to the group $G'$ of the same type with one rank larger. Then since we already know that the result holds for $G'$, we will obtain the result for $G$ by restriction.

    To be precise, by Proposition \ref{SchurMult}, we may assume $G$ is the simply-laced simply-connected Chevalley group of type $A_r$ or $D_r$. (Though we have only to consider when $G=\SL_3, \SL_4$ or $D_4$, our argument works in general.) Define $G'$ to be the simply-laced simply-connected Chevalley group of the same type as that of $G$ with rank $r+1$, so that we have the natural embedding $G\subseteq G'$. This embedding is obtained as follows. Let $\Phi'$ be the root system of $G'$, where the set of simple roots $\Delta'$ is so that $\Delta\subseteq\Delta'$ and the unique root in $\Delta'\smallsetminus\Delta$ is the root at the left endpoint of the Dynkin diagram. Then $\Phi\subseteq\Phi'$ is a subroot system, and the group $G$ viewed as a subgroup of $G$ is the group generated by $\xt_{\alpha}(t)$ for $\alpha\in\Phi$ and $t\in F$. (If $G=\SL_r$ then $G$ is embedded in $G'=\SL_{r+1}$ in the lower right corner.) We denote by $T'$, $K'$, $\Gamma(k)'$, etc, the corresponding groups for $G'$.

    Let $\Nf\subseteq\Kft'$ be the group generated by the subgroups $\Uf_{\alpha, 1}$ for $\alpha\in{\Phi'}^+\smallsetminus\Phi^+$. (If $G=\SL_r$ then $\Nf$ is the unipotent group of the ``first rows" in $\SL_{r+1}$.)  Similarly we let $\Nf^-\subseteq\Kft'$ be the group generated by the subgroups $\Uf_{\alpha, 1}$ for $\alpha\in{\Phi'}^-\smallsetminus\Phi^-$. Let $\sigma$ be our pseudo-spherical representation of $\Gammaft(1)$ constructed from a pseudo-spherical representation $\tau$ of $\Tt_1$. We know from Proposition \ref{PseudoSphericalBranch} that there is a pseudo-spherical representation $\tau'$ of $\Tt_1'$ such that $\tau'|_{\Tt_1}$ contains $\tau$ with multiplicity one. Let $\sigma'$ be the pseudo-spherical representation of $\Gammaf(1)'$ constructed from $\tau'$. We claim that
    \begin{equation}\label{ExtThmIso1}
        \Hom_{\Gammaft(1)}(\sigma,\,{\sigma'}^{\Nf})\cong \mathbb{C}.
    \end{equation}




    Note that $\Gammaft(1)'$ has an Iwahori factorization relative to the triple $(\Nf^{-},\Tft_{1}'\Gammaft(1),\Nf)$. Thus as $\Tft_{1}'\Gammaft(1)$-modules we have $\sigma^{\Nf}\cong r_{\Nf^{-},\Nf}^{\Tft_{1}'\Gammaft(1)}(\sigma')$ by Proposition \ref{P:two_functors_for_Iwahori_decomp}. On the other hand, as $\Gammaft(1)'$-modules, we have $\sigma'\cong i_{\Nf^{-},\Nf}^{\Gammaft(1)'}\left(i_{\Uf^{-}_{1},\Uf_{1}}^{\Tft_{1}'\Gammaft(1)}\tau'\right)$ by \cite[Theorem 2.17 (7)]{CMO19}. If we combine the last two isomorphisms and apply Proposition \ref{P:two_functors_for_Iwahori_decomp} \ref{P:two_functors_for_Iwahori_decomp2}, we know ${\sigma'}^{\Nf}\cong i_{\Uf^-_{1},\Uf_{1}}^{\Tft_{1}'\Gammaft(1)}\tau'$ as $\Tft_{1}'\Gammaft(1)$-modules. Thus by Proposition \ref{P:two_functions_are_adjoint},  
    \begin{equation*}
        \Hom_{\Gammaft(1)}(\sigma,{\sigma'}^{\Nf})\cong \Hom_{\Tft_{1}}(\tau,r_{\Uf_{1}^{-},\,\Uf_{1}}^{\Tft_{1}}\left(i_{\Uf^-_{1},\Uf_{1}}^{\Tft_{1}'\Gammaft(1)}\tau'\right)).
    \end{equation*}
    Now as $\Tft_{1}$-modules we have $r_{\Uf_{1},\Uf_{1}}^{\Tft_{1}}\left(i_{\Uf^-_{1},\Uf_{1}}^{\Tft_{1}'\Gammaft(1)}\tau'\right)\cong \tau'$ by Lemma \ref{L:IndRes2x} and Proposition \ref{P:two_functors_for_Iwahori_decomp} \ref{P:two_functors_for_Iwahori_decomp2}. The isomorphism \eqref{ExtThmIso1} follows because $\tau$ appears in $\tau'|_{\Tft_1}$ with multiplicity one.

    By isomorphism \eqref{ExtThmIso1}, we can identify $\sigma$ with the $\sigma$-isotypic subspace of ${\sigma'}^{\Nf}$, up to scaling. Then for any $k\in\Kft$ we can consider the $\Gammaft(1)$-submodule $k\sigma\subseteq {\sigma'}^{\Nf}$. As $\Gammaft(1)$-modules we have $k\sigma\cong {}^{k}\!\sigma$. But we know that ${}^k\!\sigma\cong \sigma$ as $\Gammaft(1)$-modules by Proposition \ref{IRep}. Since $\sigma$ is identified with the $\sigma$-isotypic subspace of ${\sigma'}^{\Nf}$, it follows that $k\sigma=\sigma$ as subspaces of ${\sigma'}^{\Nf}$. Thus $\sigma\subseteq{\sigma'}^{\Nf}$ is a $\Kft$-submodule. This proves the theorem.
\end{proof}

Recall we have the inclusions
\[
\widebar{\Gammaf}(1)\subseteq \widebar{\If}\subseteq \Kft,
\]
and the $\widebar{\If}$-module $\sigma$ is an extension of $\sigma_1$. Now we prove that the extension of $\sigma_{1}$ to $\Kft$ is also an extension of $\sigma$, namely
\begin{Prop}\label{ExtICompat}
   Let $\sigma_{\Kft}$ be an arbitrary extension of $\sigma_{1}$ to $\Kft$. Then 
   \[
   \sigma_{\Kft}|_{\Ift}\cong \sigma.
   \]
\end{Prop}

\begin{proof}
    First note that the set of extensions of $\sigma_{1}$ to $\Ift$ is a torsor under the group $\Hom(\Ift/\Gammaft(1),\mathbb{C}^{\times})$, so
every extension of $\sigma_{1}$ to $\Ift$ is of the form $\chi\otimes \sigma$ for some $\chi\in \Hom(\Ift/\Gammaft(1),\mathbb{C}^{\times})$. Hence, the restriction of $\sigma_{\Kft}$ to $\Ift$ is of the form $\chi\otimes \sigma$ for some $\chi\in \Hom(\Ift/\Gammaft(1),\mathbb{C}^{\times})$. We will show $\chi=1$. Let $\alpha\in\Delta$. We restrict further to $\Ift\cap\,^{w_{\alpha}}\Ift$ and apply Proposition \ref{IRep} \ref{ResSigAlpha} to see that the restriction of $\sigma_{\Kft}$ to $\Ift\cap\,^{w_{\alpha}}\Ift$ is isomorphic to $\chi\otimes \sigma_{\alpha}$. Since $\sigma_{\Kft}$ and $\sigma_{\alpha}$ are $w_{\alpha}$ invariant by Theorem \ref{IRep} \ref{sigWInvar} and \ref{SigAWInv}, it follows that $\chi$ must be $w_{\alpha}$ invariant. Since $\Phi$ is simply-laced with rank $>1$, the character $\chi$ must be trivial on elements of the form $h_{\alpha}(t)$ for all $t\in \Ocal^{\times}/(1+\varpi^{2e}\Ocal)$. Also by the $w_{\alpha}$-invariance, $\chi$ is trivial on the elements of the form $x_{\alpha}(t)$ for $\alpha>0$ and $t\in \Ocal$. Hence $\chi$ is trivial.
\end{proof}




\subsection{Additional properties of $\sigma$}\label{SigProps}


The $\Ift$-module $\sigma=i_{{\Uf}_{0},{\Uf}^{-}_{1}}\tau$ exhibits some additional properties that will be useful in our investigation of a Hecke algebra on $\Gt$. We record these in this subsection. Recall from \eqref{E:I_and_I_2_def} that ${\If}_{2}\subseteq {\If}$ is the subgroup generated by the subgroups ${\Tf}_{0}$, ${\Uf}_{0}$ and ${\Uf}^{-}_{2}$. Let
\[
\sigma_{2}:=i_{{\Uf}_{0},{\Uf}^{-}_{2}}^{\Ift_2}\tau,
\]
which is an irreducible $\widebar{\If}_2$-module by Proposition \ref{P:two_functors_for_Iwahori_decomp}.

\begin{Lem} \label{+-IdemSwap}

	Let $\alpha\in\Phi^{+}$ and $t\in \mathcal{O}$. Then in the group algebra of $\Kft$, we have
	\begin{equation*}
		e_{{\Uf}_{0}}e_{{\Uf}^{-}_{2}}x_{\alpha}(\varpi^{2e-1}t)=e_{{\Uf}_{0}}e_{{\Uf}^{-}_{2}}.
	\end{equation*}
\end{Lem}

\begin{proof}
	 For all $y\in {\Uf}^{-}_{2}$ the commutator $[y,x_{\alpha}(\varpi^{2e-1}t)]$ is trivial mod $\Gamma(2e)$, which gives
     \[
     e_{{\Uf}^{-}_{2}}x_{\alpha}(\varpi^{2e-1}t)=x_{\alpha}(\varpi^{2e-1}t)e_{{\Uf}^{-}_{2}}.
     \]
     The result follows because $x_{\alpha}(\varpi^{2e-1}t)\in {\Uf}_{0}$.
\end{proof}

\begin{Prop}\label{Gamma02IrrInd}
	The $\Ift$-module $\Ind_{\Ift_{2}}^{\Ift}\sigma_{2}$ is irreducible.
\end{Prop}

\begin{proof}
It suffices to show $\End_{\Ift}(\Ind_{\Ift_{2}}^{\Ift}\sigma_{2})\cong \mathbb{C}$. But by the Frobenius reciprocity and Proposition \ref{P:two_functions_are_adjoint} (2), we have
	\begin{align*}
		\End_{\Ift}(\Ind_{\Ift_{2}}^{\Ift}\sigma_{2})\cong& \Hom_{\Ift_{2}}(\sigma_{2},\Ind_{\Ift_{2}}^{\Ift}\sigma_{2})\\
		\cong& \Hom_{\widebar{{\Tf}}_{0}}(\tau, r_{{\Uf}^{+}_{0},{\Uf}^{-}_{2}}^{\Ift_{2}}\Ind_{\Ift_{2}}^{\Ift}\sigma_{2}).
	\end{align*}
	
	By definition
	\begin{equation*}
		r_{{\Uf}^{+}_{0},{\Uf}^{-}_{2}}^{\Ift_{2}}\Ind_{\Ift_{2}}^{\Ift}\sigma_{2}=e_{{\Uf}^{+}_{0}}e_{{\Uf}^{-}_{2}}\mathcal{H}(\Ift_{2})\otimes_{\Ift_{2}} \mathcal{H}(\Ift)\otimes_{\Ift_{2}}\mathcal{H}(\Ift_{2})e_{{\Uf}^{+}_{0}}e_{{\Uf}^{-}_{2}}\otimes_{\widebar{{\Tf}}_{0}}\tau.
	\end{equation*}
Thus by the Iwahori factorization and Lemma \ref{L:spanning_vectors}, $r_{{\Uf}^{+}_{0},{\Uf}^{-}_{2}}^{\Ift_{2}}\Ind_{\Ift_{2}}^{\Ift}\sigma_{2}$ is spanned by elements of the form 
\begin{equation*}
	e_{{\Uf}^{+}_{0}}e_{{\Uf}^{-}_{2}}\otimes_{\Ift_{2}} y\otimes_{\Ift_{2}}e_{{\Uf}^{+}_{0}}e_{{\Uf}^{-}_{2}}\otimes_{\widebar{{\Tf}}_{0}} w,
\end{equation*}	
where $y\in {\Uf}^{-}_{1}$ and $w\in\tau$.

We claim that if $y\in {\Uf}^{-}_{1}\smallsetminus{\Uf}^{-}_{2}$, then $e_{{\Uf}^{+}_{0}}e_{{\Uf}^{-}_{2}}\otimes _{\Ift_{2}}y\otimes_{\Ift_{2}}e_{{\Uf}^{+}_{0}}\otimes_{\widebar{{\Tf}}_{0}} w=0$. To prove the claim, let us write $y=y^{\prime}x_{-\alpha}(u)$, where $u\in\varpi\mathcal{O}^{\times}$ and $y^{\prime}\in \prod_{\beta}{\Uf}_{-\beta,1}$, where the product is over $\beta\in \Phi^{+}\smallsetminus\{\alpha\}$. Since $v\varpi\mathcal{O}\subseteq 4\mathcal{O}$, there exists $v\in \varpi^{2e-1}\mathcal{O}^{\times}$ such that $(1+ uv,\varpi)=-1$ by Corollary \ref{C:Hilbert_symbol_unique_element_integral_radical}. Because $v\in \varpi^{2e-1}\mathcal{O}$ and $y^{\prime}\in \prod_{\beta}{\Uf}_{-\beta,1}$ where $\beta\in \Phi^{+}\smallsetminus\{\alpha\}$, we have $y^{\prime}\xt_{\alpha}(v)=\xt_{\alpha}(v)y^{\prime}$.

We then have
\begin{align*}
	&e_{{\Uf}^{+}_{0}}e_{{\Uf}^{-}_{2}}\otimes_{\Ift_{2}}y\otimes_{\Ift_{2}}e_{{\Uf}^{+}_{0}}e_{{\Uf}^{-}_{2}}\otimes_{\widebar{{\Tf}}_{0}} w\\
    =&e_{{\Uf}^{+}_{0}}e_{{\Uf}^{-}_{2}}\otimes_{\Ift_{2}}y\otimes_{\Ift_{2}}\xt_{\alpha}(v)e_{{\Uf}^{+}_{0}}e_{{\Uf}^{-}_{2}}\otimes_{\widebar{{\Tf}}_{0}} w\quad(\text{because $\xt_{\alpha}(v)e_{{\Uf}^{+}_{0}}=e_{{\Uf}^{+}_{0}}$})\\
	=&e_{{\Uf}^{+}_{0}}e_{{\Uf}^{-}_{2}}\otimes_{\Ift_{2}}y^{\prime}\xt_{-\alpha}( u)\xt_{\alpha}(v)\otimes_{\Ift_{2}}e_{{\Uf}^{+}_{0}}e_{{\Uf}^{-}_{2}}\otimes_{\widebar{{\Tf}}_{0}} w\\
	=&-e_{{\Uf}^{+}_{0}}e_{{\Uf}^{-}_{2}}\otimes_{\Ift_{2}}y^{\prime}\xt_{\alpha}(v)\htt_{\alpha}(1+uv)\xt_{-\alpha}(u)\otimes_{\Ift_{2}}e_{{\Uf}^{+}_{0}}e_{{\Uf}^{-}_{2}}\otimes_{\widebar{{\Tf}}_{0}} w\quad (\text{Lemma \ref{+-IdemSwap}})\\
	=&-e_{{\Uf}^{+}_{0}}e_{{\Uf}^{-}_{2}}\otimes_{\Ift_{2}}\xt_{\alpha}(v)y^{\prime}\xt_{-\alpha}(u)\otimes_{\Ift_{2}}e_{{\Uf}^{+}_{0}}e_{{\Uf}^{-}_{2}}\otimes_{\widebar{{\Tf}}_{0}} w\quad (\text{by $y^{\prime}\xt_{\alpha}(v)=\xt_{\alpha}(v)y^{\prime}$})\\
	=&-e_{{\Uf}^{+}_{0}}e_{{\Uf}^{-}_{2}}\otimes_{\Ift_{2}}y\otimes_{\Ift_{2}}e_{{\Uf}^{+}_{0}}e_{{\Uf}^{-}_{2}}\otimes_{\widebar{{\Tf}}_{0}} w,
\end{align*}
where for the last equality we used
\[
e_{{\Uf}^{+}_{0}}e_{{\Uf}^{-}_{2}}\xt_{\alpha}(v)=e_{{\Uf}^{+}_{0}}e_{{\Uf}^{-}_{2}},
\]
which holds because $\xt_{\alpha}(v)$ normalizes $\Uf_2^-$ and $\xt_{\alpha}(v)\in\Uf_0^+$. Thus $e_{{\Uf}^{+}_{0}}e_{{\Uf}^{-}_{2}}\otimes_{\Ift_{2}}ye_{{\Uf}^{+}_{0}}e_{{\Uf}^{-}_{2}}\otimes_{\widebar{{\Tf}}_{0}} w=0$, and so we have $r_{{\Uf}^{+}_{0},{\Uf}^{-}_{2}}^{\Ift_{2}}\Ind_{\Ift_{2}}^{\Ift}\sigma_{2}\cong \tau$ as $\widebar{{\Tf}}_{0}$-modules.

Therefore $\End_{\Ift}(\Ind_{\Ift_{2}}^{\Ift}\sigma_{2})\cong \mathbb{C}$ and the result follows.
\end{proof}

\begin{Prop}\label{IndSig2}
	As $\widebar{\If}$-modules
	\begin{equation*}
		\sigma \cong \Ind_{\widebar{\If}_2}^{\widebar{\If}}\sigma_{2}.
	\end{equation*}
\end{Prop}
\begin{proof}
	By Lemma \ref{LHoms} we know that $\Hom_{\widebar{\If}}(\sigma,\Ind_{\widebar{\If}_2}^{\widebar{\If}}\sigma_{2})\neq 0$. By Lemma \ref{Gamma02IrrInd} we know that $\Ind_{\widebar{\If}_{2}}^{\widebar{\If}}\sigma_{2}$ is irreducible. Since $\sigma$ is also irreducible the result follows by Schur's Lemma.
\end{proof}

\begin{Lem}\label{pmUniComm}
	Let $\ell\in \mathbb{Z}$ be such that $0\leq \ell\leq e$. Then the groups ${\Uf}_{2e-2\ell}$ and ${\Uf}^{-}_{2\ell}$ commute.
\end{Lem}
\begin{proof}
Let $\alpha,\beta\in\Phi^{+}$ be such that $\beta\neq \alpha$. Then for each $u,v\in\mathcal{O}$, one can readily derive the following two identities.
	\begin{align*}
		[x_{\alpha}(\varpi^{2e-2\ell}u),x_{-\beta}(\varpi^{2\ell}v)]=&1 \quad (\text{by \eqref{E:MR2'} viewed mod $\Gamma(2e)$});\\
		x_{\alpha}(\varpi^{2e-2\ell}u)x_{-\alpha}(\varpi^{2\ell}v)=&x_{-\alpha}(\varpi^{2\ell}v)x_{\alpha}(\varpi^{2e-2\ell}u)\quad (\text{by Proposition \ref{+-IdemSwap}}).
	\end{align*}
	We note that in the second identity there is no Hilbert symbol $(1+\varpi^{2e}uv,\varpi^{2\ell}v)$ because $1+\varpi^{2e}uv\in R$ and if $v\in \varpi\Ocal$ then $1+\varpi^{2e}uv\in 1+\varpi^{2e+1}\Ocal$. 
\end{proof}


\begin{Prop}\label{Sig2UInvar} In the $\Ift_{2}$-module $\sigma_{2}$ we have
	\begin{equation*}
		\sigma_{2}^{{\Uf}_{2e-2}}=\sigma_{2}.
	\end{equation*}
\end{Prop}

\begin{proof}
	By the Iwahori factorization and Lemma \ref{L:spanning_vectors}, $\sigma_{2}$ is spanned by elements of the form $ye_{{\Uf}_{0}}e_{{\Uf}^{-}_{2}}\otimes_{\Tt_{0}} w$, where $y\in {\Uf}^{-}_{2}$ and $w\in \tau$. Let $x\in \Uf_{2e-2}$, so that $xy=yx$ by Lemma \ref{pmUniComm}. Hence 
    \[
    xye_{{\Uf}_{0}}e_{{\Uf}^{-}_{2}}\otimes_{\Tt_{0}} w=yxe_{{\Uf}_{0}}e_{{\Uf}^{-}_{2}}\otimes_{\Tt_{0}} w=ye_{{\Uf}_{0}}e_{{\Uf}^{-}_{2}}\otimes_{\Tt_{0}} w,
    \]
    because $e_{\Uf_0}$ absorbs $x$. The result follows.
\end{proof}

\subsection{Brief wrap-up}


Let us briefly summarize what we have achieved so far. Our initial input is a pseudo-spherical representation $\tau$ of the compact group $\Tt_0$ constructed in our companion paper \cite{KT26}. This is a finite dimensional Weyl invariant irreducible representation which is trivial on $T_R$.

Next using parahoric induction, we obtained an irreducible representation $\sigma:=i_{\Uf, \Uf_1^-}^{\Ift}\tau$ of the finite group $\Ift\cong \It/\s(\Gamma(2e))$, which we can inflate to the Iwahori subgroup $\It$. This representation $\sigma$ has two important properties: first, it remains irreducible when restricted to $\Gammat(1)$, and in fact can be viewed as being parahorically induced; second, $\sigma$ extends to $\Kft=\Kt/\s(\Gamma(2e))$, which we also call a pseudo-spherical representation and denote it by the same symbol $\sigma$. We can inflate $\sigma$ to a representation of the compact group $\Kt$. 

As in Remark \ref{R:Gamma(2e)_is_max}, the group $\Gamma(2e)$ is the maximal congruence subgroup of the hyperspecial maximal compact subgroup $K$ that splits to $\Kt$ with normal image. Hence the representation $\sigma$ of $\Kt$ can be considered as the covering analogue of the trivial representation. Hence it makes sense to call $\sigma$ a ``pseudo-spherical" representation of $\Kt$. Indeed, $\sigma$ does play a role of the trivial representation in the representation theory of covering groups.


\subsection{Finite Shimura correspondence}\label{SS:finite_Shimura}


A first instance where one can see how our pseudo-spherical representation plays a role of the trivial representation is a ``finite Shimura correspondence" originally established by Savin in \cite[Theorem 4.1]{S12}, when $F/\mathbb{Q}_{2}$ is unramified. In this section, we generalize it to an arbitrary $F$. Indeed, this correspondence establishes an equivalence of the categories of representations of the finite group $K\slash\Gamma(1)$ with a certain category of representations of $\Kft$, and this equivalence of categories is precisely given by tensoring the pseudo-spherical representation as $\pi\mapsto \pi\otimes\sigma$ for a representation $\pi$ of $K\slash\Gamma(1)$.

To be precise, let $\mathcal{M}(G_\kappa)$ be the category of representations of $G_\kappa$, where we recall $\kappa$ is the residue field of our $2$-adic field $F$. Let $\mathcal{M}_{\sigma}(\Kft)$ be the full subcategory of smooth $\Kft$-modules $\widetilde{\pi}$ such that $\widetilde{\pi}|_{\widebar{\Gammaf}(1)}$ is a nontrivial $\sigma$-isotypic $\widebar{\Gammaf}(1)$-module, or $\pit=0$. Let us then define two functors
\begin{align*}
\mathcal{F}_{\sigma}&:\mathcal{M}(G_{\kappa})\rightarrow \mathcal{M}_{\sigma}(\Kft),\qquad \pi\mapsto \sigma\otimes \pi;\\
\mathcal{G}_{\sigma}&:\mathcal{M}_{\sigma}(\Kft)\rightarrow \mathcal{M}(G_{\kappa}),\qquad \widetilde{\pi}\mapsto \Hom_{\widebar{\Gammaf}(1)}(\sigma,\widetilde{\pi}).
\end{align*}
Here, $\Kft$ acts diagonally on $\sigma\otimes \pi$. Also $G_\kappa=\Kft\slash\widebar{\Gammaf}(1)$ acts on $\Hom_{\widebar{\Gammaf}(1)}(\sigma,\widetilde{\pi})$ by ``conjugation", namely for each $g\in \Kft\slash\widebar{\Gammaf}(1)$ and $\phi\in \Hom_{\widebar{\Gammaf}(1)}(\sigma,\widetilde{\pi})$
\[
(g\cdot \phi)(v)=g\phi(g^{-1}v)
\] 
for all $v\in \widetilde{\pi}$, which is indeed well defined.

The main theorem of this subsection is the following ``finite Shimura correspondence".

\begin{Thm}\label{FiniteShimCor}
	The functors $\mathcal{F}_{\sigma}$ and $\mathcal{G}_{\sigma}$ are mutual quasi-inverses and hence define an equivalence of categories between $\mathcal{M}(G_{\kappa})$ and $\mathcal{M}_{\sigma}(\Kft)$.
\end{Thm}

We first need a couple of elementary lemmas.

\begin{Lem}\label{FinRkLem}
	Let $V_{1},V_{2},W$ be vector spaces. If $V_{1},V_{2}$ are finite dimensional, then the map 
    \[
    \Hom(V_{1},V_{2})\otimes W\rightarrow \Hom(V_{1},V_{2}\otimes W),\quad f\otimes w\mapsto (v\mapsto f(v)\otimes w),
    \]
    is an isomorphism of vector spaces. 
	
	Furthermore, suppose that $G$ is a group and $V_{1},V_{2},W$ are $G$-modules. Let $G$ act diagonally on tensor products and on Hom spaces via``conjugation''. Then this isomorphism is a $G$-module isomorphism. 
\end{Lem}	
\begin{proof}
Elementary exercise.
\end{proof}

\begin{Lem}\label{HOMotimesSig}
    Let $M \subseteq L$ be finite groups. Let $\pi_1, \pi_2$ be (not necessarily finite dimensional) $L$-modules which are trivial on $M$, and let $\sigma_1, \sigma_2$ be irreducible (hence necessarily finite dimensional) $L$-modules which remain irreducible when restricted to $M$. If $\phi:\sigma_1\to\sigma_2$ is an $L$-module isomorphism, then the map
    \[
    \Hom_L(\pi_1,\pi_2)\longrightarrow\Hom_L(\pi_1\otimes\sigma_1, \pi_2\otimes\sigma_2),\quad f\mapsto f\otimes \phi,
    \]
    is an isomorphism of vector spaces.
\end{Lem}
\begin{proof}
    By the tenor-Hom adjunction, we have
    \[
     \Hom(\pi_1\otimes\sigma_1, \pi_2\otimes\sigma_2)\cong \Hom(\pi_1, \Hom(\sigma_1,\pi_2\otimes\sigma_2)),
    \]
    which gives
    \[
     \Hom_L(\pi_1\otimes\sigma_1, \pi_2\otimes\sigma_2)\cong \Hom_L(\pi_1, \Hom(\sigma_1,\pi_2\otimes\sigma_2)).
    \]
    Since $M$ acts trivially on $\pi_1$, the image of each element in $\Hom_L(\pi_1, \Hom(\sigma_1,\pi_2\otimes\sigma_2))$ lies in the $M$-invariant subspace of $\Hom(\sigma_1,\pi_2\otimes\sigma_2)$, implying
    \[
     \Hom_L(\pi_1\otimes\sigma_1, \pi_2\otimes\sigma_2)\cong \Hom_L(\pi_1, \Hom_M(\sigma_1,\pi_2\otimes\sigma_2)).
    \]
    
    Noting that both $\sigma_1$ and $\sigma_2$ are finite dimensional, Lemma \ref{FinRkLem} gives
    \[
    \Hom_M(\sigma_1,\pi_2\otimes\sigma_2)\cong \Hom_M(\sigma_1, \sigma_2)\otimes \pi_2\cong\C\otimes\pi_2\cong\pi_2,
    \]
    where we used $\Hom_M(\sigma_1, \sigma_2)\cong\C$ by Schur's lemma. Therefore, we have
    \[
    \Hom_L(\pi_1\otimes\sigma_1, \pi_2\otimes\sigma_2)\cong \Hom_L(\pi_1, \pi_2).
    \]

    All of the preceding isomorphisms are explicit once $\phi$ is used to normalize the isomorphism $\Hom_M(\sigma_1, \sigma_2)\cong\C$ coming from Schur's lemma, and it can be checked directly that $f\in \Hom_L(\pi_{1},\pi_{2})$ corresponds to $f\otimes \phi\in \Hom_L(\pi_{1}\otimes \sigma_{1},\pi_{2}\otimes \sigma_{2})$.
\end{proof}

Now, we are ready to prove the theorem.

\begin{proof}[Proof of Theorem \ref{FiniteShimCor}]
Noting that $G_\kappa=\Kft\slash \widebar{\Gammaf}(1)$, we can identify the category $\mathcal{M}(G_\kappa)$ with the category of $\Kft$-modules which are trivial on $\widebar{\Gammaf}(1)$. Then the above lemma (with $L=\Kft$, $M=\widebar{\Gammaf}(1)$ and $\sigma_1=\sigma_2=\sigma$) means that the functor $\mathcal{F}_\sigma$ is fully faithful. (Note that we necessarily have $\Hom_L(\pi_1,\pi_2)=\Hom_{L\slash M}(\pi_1,\pi_2)$ in the above lemma.)

Hence to show $\mathcal{F}_\sigma$ is an equivalence of the categories, it suffices to show that the functor $\mathcal{F}_\sigma$ is essentially surjective on the objects. But for each $\pit\in\mathcal{M}_\sigma(\Kft)$, we have
\[
\mathcal{F}_\sigma\circ\mathcal{G}_\sigma(\pit)=\sigma\otimes \Hom_{\widebar{\Gammaf}(1)}(\sigma, \pit).
\]
Now, we have the natural $\Kft$-isomorphism
\[
\sigma\otimes \Hom_{\widebar{\Gammaf}(1)}(\sigma, \pit)\longrightarrow \pit,\quad  v\otimes f\mapsto f(v),
\]
because $\pit|_{\widebar{\Gammaf}(1)}$ is $\sigma$-isotypic. Hence the functor $\mathcal{F}_\sigma$ is essentially surjective on the objects.

Now that we have shown that $\mathcal{F}_\sigma\circ\mathcal{G}_\sigma\cong \mathrm{id}_{\mathcal{M}_\sigma(\Kft)}$ and $\mathcal{F}_\sigma$ is an equivalence of the categories, we know that $\mathcal{G}_\sigma$ is a quasi-inverse of $\mathcal{F}_\sigma$.    
\end{proof}

The next proposition shows that the equivalence of categories of Theorem \ref{FiniteShimCor} commutes with induction in certain cases.

\begin{Prop}\label{SCInd}
Assume $\Hft$ is such that $\Gammaft(1)\subseteq\Hft\subseteq\Kft$. Let $\pi$ be a representation of $\Hft$ that is trivial on $\Gammaft(1)$, so that $\pi$ is really a representation of $\Hft/\Gammaft(1)\hookrightarrow G_{\kappa}=\Kft\slash\Gammaft(1)$.
	\begin{enumerate}
		\item The $\Kft$-module $\Ind_{\Hft}^{\Kft}(\pi\otimes\sigma)$ is an object in $\mathcal{M}_{\sigma}(\Kft)$.\label{IndInCat}
		\item The map 
        \[
        \phi:\left(\Ind_{\Hft}^{\Kft}\pi\right)\otimes \sigma\longrightarrow \Ind_{\Hft}^{\Kft}(\pi\otimes \sigma),\quad f\otimes w\mapsto (k\mapsto f(k)\otimes kw),
        \]
        is a surjective map of $\Kft$-modules. Furthermore, if $\pi$ is finite dimensional, then both sides have the same dimension, so that $\pi$ is an isomorphism.\label{IndSigIso}
		\item If $\pi$ is finite dimensional, then the map 
        \[\End_{G_{\kappa}}(\Ind_{\Hft}^{\Kft}\pi)\rightarrow \End_{\Kft}(\Ind_{\Hft}^{\Kft}(\pi\otimes \sigma)),\quad F\mapsto \phi\circ(F\otimes id_{\sigma})\circ \phi^{-1},
        \]
        is an isomorphism of $\mathbb{C}$-algebras. \label{EndIso}
	\end{enumerate}
\end{Prop}

\begin{proof}
	\eqref{IndInCat} follows from the Mackey theory and $\Gammaft(1)\cap \,^{g}\Hft=\Gammaft(1)$ for any $g\in \Kft$.
	
	We prove \eqref{IndSigIso}. A direct computation shows that the map is well defined and a $\Kft$-module homomorphism. We will show this is surjective. Let $W$ be the subspace of $\Ind_{\Hft}^{\Kft}\pi$ consisting of functions with support contained in $\Hft$. Let $W^{\prime}$ be the subspace of $\Ind_{\Hft}^{\Kft}(\pi\otimes \sigma)$ consisting of functions with support contained in $\Hft$. Recall that as $\Hft$-modules $W\cong \pi$ and $W^{\prime}\cong \pi\otimes \sigma$. In both cases the isomorphism is given by evaluation at the identity. Under these identifications, one can see that $\phi:W\otimes \sigma\rightarrow W^{\prime}$ becomes the identity map. Thus $W^{\prime}$ is in the image of $\phi$. Since the index $[\Kft:\Hft]$ is finite and $\Ind_{\Hft}^{\Kft}(\pi\otimes \sigma)$ is generated by $W^{\prime}$ as a $\Kft$-module, it follows that $\phi$ is surjective. 
	
	
	\eqref{EndIso} follows from \eqref{IndSigIso} and Theorem \ref{FiniteShimCor}. 
\end{proof}

The isomorphism of Proposition \ref{SCInd} \eqref{EndIso} has a very important property in that it preserves support, when viewed as an isomorphism of Hecke algebras. We return to this point in \S \ref{FinShimHecke} after we introduce the notation for Hecke algebras.  


\subsection{Whittaker Invariants}\label{SS:WhitInvar}


In this subsection we record a few simple results on twisted ${\Uf}_{0}$-invariants for representations in $\mathcal{M}_{\sigma}(\Kft)$. (Though we will not use it for our main theorems, it is certainly of some interest.) Recall that $\sigma^{{\Uf}_{0}}\cong \tau$ as $\Tft_{0}$-modules.

\begin{Cor}
	Let $\psi\in\Hom({\Uf}_{0},\mathbb{C}^{\times})$ be a nontrivial character that is trivial on ${\Uf}_{1}$. Then
	\begin{equation*}
		\sigma^{({\Uf}_{0},\psi)}=0. 
	\end{equation*}
\end{Cor}

\begin{proof}
	This follows directly from Lemma \ref{UInvar} because $\sigma^{({\Uf}_{0},\psi)}\subseteq \sigma^{\Uf_1}$.
\end{proof}

Thus we see that for $\sigma$ to support a $\psi$-Whittaker functional the conductor of $\psi$ must be strictly larger than $1$.

\begin{Cor}\label{WhitMult}
	Let $\pi$ be a $G_{\kappa}$-module and let $\psi\in\Hom({\Uf}_{0},\mathbb{C}^{\times})$. Then 
	\begin{equation*}
		(\pi\otimes \sigma)^{({\Uf}_{0},\psi)}\supseteq \bigoplus_{\psi=\psi_{1}\psi_{2}}\pi^{({\Uf}_{0},\psi_{1})}\otimes \sigma^{({\Uf}_{0},\psi_{2})},
	\end{equation*}
    where the direct sum is over all $\psi_{1},\psi_{2}\in\Hom({\Uf}_{0},\mathbb{C}^{\times})$ such that $\psi=\psi_{1}\psi_{2}$ and $\psi_{1}$ is trivial on ${\Uf}_{1}$. 
    In particular, if $\psi_{1}$ is non-degenerate, $\psi_{2}$ is trivial, and $\pi$ is $\psi_{1}$-generic, then
	\begin{equation*}
		\dim(\pi\otimes \sigma)^{({\Uf}_{0},\psi)}\geq \dim\tau=\left|Y\slash \Yt\right|^{\frac{d}{2}}.
	\end{equation*}
\end{Cor}

Corollary \ref{WhitMult} implies that multiplicity one for Whittaker functionals fails for $\Kft$. It will be likely be important to study the generalized Gelfand-Graev representations in this finite group setting. Such an investigation will likely be needed to generalize the results of \cite{GGK24,GGK25,KOW26} to the wild double covers studied in this paper.


\section{Hecke algebras}\label{S:Hecke}


Once we have constructed the Iwahori type $(\It, \sigma)$, which will be shown to be indeed a type in the sense of Bushnell-Kutzko, we can form the Iwahori Hecke algebra
\[
\Hcal(\Gt, \It; \sigma^{\vee}).
\]
In this section, we will prove that the support of this Hecke algebra is $\Wt_{\ea}$ (see \S \ref{SS:affine_Weyl_covering_group} for $\Wt_{\ea}$) and establish its IM-presentation. 

\subsection{Basic theory of Hecke algebras}\label{SS:HeckeReview}

In this subsection, we catalog basics of representations of compact groups and Hecke algebras to the extent necessary for our purposes, and fix our notation and conventions.

Let $G$ be a locally compact unimodular group and $K\subseteq G$ an open compact subgroup. Let $(\pi, V)$ be a finite dimensional (not necessarily irreducible) representation of $K$. We define the $\pi$-spherical Hecke algebra of $G$ by
\[
\Hcal(G, K; \pi):=\{f\in C_c^\infty(G, \End_{\C}(\pi^\vee))\st f(k_1gk_2)=\pi^\vee(k_1)f(g)\pi^\vee(k_2)\text{ for all $k_1, k_2\in K$}\}.
\]
This is an algebra under convolution. Specifically, for $f_{1},f_{2}\in \Hcal(G, K; \pi)$ we have
\begin{equation*}
	f_{1}*f_{2}(g):=\int_{G}f_{1}(h)f_{2}(h^{-1}g)dh=\int_{G}f_{1}(gh^{-1})f_{2}(h)dh,
\end{equation*}
where the Haar measure on $G$ is normalized so that the measure of $K$ is $1$.

Let $(\rho,W)$ be a $G$-module. The Hecke algebra $\Hcal(G, K; \pi^\vee)$ acts on $\Hom_{K}(V,W)$ on the right in the following way. Let $f\in \Hcal(G, K; \pi^\vee)$ and $A\in\Hom_{K}(V,W)$. Then
\begin{equation*}
	(A*f)(v):=\int_{G}\rho(g^{-1})A(f(g)v)dg.
\end{equation*}

The Hecke algebra $\Hcal(G, K; \pi^\vee)$ also has a canonical left action on $\ind_{K}^{G}\pi$ defined by
    \begin{equation*}
        f*F(g):=\int_{G}f(x)F(x^{-1}g)dx,
    \end{equation*}
    where $f\in \Hcal(G, K; \pi^\vee)$ and $F\in \ind_{K}^{G}\pi$. The following lemma is well-known. (See, for example, \cite[(2.6)]{BK98}.)

\begin{Lem}\label{HeckeAsEnd}
    The canonical left action of $\Hcal(G, K; \pi^\vee)$ on $\ind_{K}^{G}\pi$ induces an isomorphism of $\mathbb{C}$-algebras $\Hcal(G, K; \pi^{\vee})\cong \End_{G}(\ind_{K}^{G}\pi)$. 
\end{Lem}

For each $g\in G$, we let
\[
\Hcal(G, K; \pi)_g
\]
be the vector space of all $f\in \Hcal(G, K; \pi)$ that vanish outside the double coset $KgK$.

A very important fact we repeatedly use in this paper is 
\begin{Lem}\label{L:Bushnell_Kutzko_equivalence}
For each $g\in G$, we have canonical isomorphisms
\[
\Hom_{K\cap{^gK}}(\;{^g\pi},\pi)\cong\Hcal(G, K; \pi^\vee)_g\cong\Hcal(G, K; \pi)_{g^{-1}}
\]
as vector spaces.
\end{Lem}
\begin{proof}
See \cite[(4.1.1) Proposition, p.144]{Bushnell-Kutzko-Book}.
\end{proof}

\begin{Lem}\label{L:Bushnell_Kutzko_induced}
Let $K\subseteq G$ be as above and let $H$ be an open compact group such that $H\subseteq K\subseteq G$. Let $\rho$ be a finite dimensional representation of $H$. We have a natural isomorphism
\[
\Hcal(G, K; \Ind_H^K\rho)\cong \Hcal(G, H; \rho)
\]
such that for each $g\in G$
\[
\Hcal(G, K; \Ind_H^K\rho)_g\cong\bigoplus_{\substack{h\in H\backslash G \slash H\\KhK=KgK}} \Hcal(G, H; \rho)_h,
\]
where the sum is over all the double cosets $HhH$ such that $KhK=KgK$.
\end{Lem}
\begin{proof}
See \cite[(4.1.5) Corollary, p.146]{Bushnell-Kutzko-Book}.
\end{proof}

The next lemma will be useful for relating some of the Hom-spaces introduced above.

\begin{Lem}\label{HomIsoRel}
	Let $G$  and $H$ be groups and let $\phi:G\rightarrow H$ be a group isomorphism. Let $\sigma,\tau$ be $G$-modules. Then the Hom-spaces $\Hom_{G}(\sigma,\tau)$ and $\Hom_{H}(\,^{\phi}\sigma,\,^{\phi}\tau)$ are equal as subspaces of $\Hom_{\mathbb{C}}(\sigma,\tau)$.
\end{Lem}

Next we describe some results on twisting Hecke algebras by automorphisms.

\begin{Lem}\label{L:HeckeAutTwist}
	Let $\phi\in\Aut(G)$. Then
	\begin{enumerate}
		\item[(1)] The map $\Hcal(G,K,\pi)\rightarrow \Hcal(G,\phi(K),\,^{\phi}\pi)$ defined by $f\mapsto f\circ\phi^{-1}$ is an isomorphism of $\mathbb{C}$-algebras.\label{HeckeAutTwist}
    \end{enumerate}
	Suppose further that $\phi(K)=K$ and $\,^{\phi}\pi\cong \pi$, and let $A_{\phi}\in \GL(\pi^\vee)$ be such that $A_{\phi}\pi^\vee(k)=\pi^\vee(\phi^{-1}(k))A_{\phi}$. Then
    \begin{enumerate}
		\item[(2)]  $\Hcal(G,K,\,^{\phi}\pi)=A_{\phi}\Hcal(G,K,\pi)A_{\phi}^{-1}$ (we emphasize that this is an equality of spaces of functions and not just an isomorphism of algebras);\label{HeckeEqual}
		\item[(3)] $\Hcal(G,K,\,\pi)_{\phi(g)}\neq 0$ if and only if $\Hcal(G,K,\pi)_{g}\neq 0$. \label{HeckeSuppRel}
	\end{enumerate}	
\end{Lem}


\subsection{Finite Hecke algebra Shimura Correspondence}\label{FinShimHecke}


In this subsection we prove that the finite Shimura correspondence of Theorem \ref{FiniteShimCor} preserves the support of Hecke algebras. Let us set up the notation first. Let $H\subseteq G$ be such that $\Gamma(1)\subseteq H\subseteq K$, so that $\Gammaft(1)\subseteq \Hft \subseteq \Kft$. Denote by $H_\kappa$ the image of $H$ in $G_{\kappa}=K\slash\Gamma(1)$, so that we have
    \[
    \begin{tikzcd}
        \Kt\slash \s(\Gamma(2e))\ar[r,phantom, "{=}"]&\Kft\ar[r, two heads]&G_\kappa \ar[r, phantom, "{=}"]&K\slash\Gamma(1)\\
        \Ht\slash \s(\Gamma(2e))\ar[r,phantom, "{=}"]&\Hft\ar[r, two heads]\ar[u,phantom, "{\subseteq}" sloped, description] &H_\kappa \ar[r, phantom, "{=}"]\ar[u,phantom, "{\subseteq}" sloped, description]&H\slash\Gamma(1) \rlap{\,.}
    \end{tikzcd}
    \]
Fix a finite dimensional representation $\pi$ of $H_\kappa$, which we also view as a representation of $\Hft$ by inflating via the above surjection $\Hft\twoheadrightarrow H_{\kappa}$.

For each $x\in G_{\kappa}$, we let $\xt\in\Kft$ be a preimage under the surjection $\Kft\twoheadrightarrow G_{\kappa}$. The choice of $\xt$ is not unique, but if $\xt'$ is another choice, then
\begin{equation}\label{E:double_coset_equal_finite_Hecke}
    \Hft\,\xt\,\Hft=\Hft\,\xt'\,\Hft.
\end{equation}
because $\xt'=\xt \gamma$ for some $\gamma\in\Gammaft(1)$.

We then have
\begin{Thm}\label{T:finite_Hecke_algebra_isom}
    With the above notation, we have a $\C$-algebra isomorphism 
    \[
    \Hcal(G_{\kappa}, H_{\kappa}; \pi^{\vee})\cong \Hcal(\Kft, \Hft; (\pi\otimes\sigma)^{\vee})
    \]
    which restricts to an isomorphism
    \[
    \Hcal(G_{\kappa}, H_{\kappa}; \pi^{\vee})_x\cong \Hcal(\Kft, \Hft; (\pi\otimes\sigma)^{\vee})_{\xt}
    \]
    of $\C$-vector spaces for all $x\in G_{\kappa}$; namely the isomorphism is support-preserving. (Note that the space $\Hcal(\Kft, \Hft; (\pi\otimes\sigma)^{\vee})_{\xt}$ is independent of the choice of $\xt$ thanks to \eqref{E:double_coset_equal_finite_Hecke}.)
\end{Thm}
\begin{proof}
We already know the isomorphisms
\begin{equation}\label{E:composition_of_three_isomorphism_fintie_Hecke}
\Hcal(G_{\kappa},H_{\kappa};\pi^{\vee})\xrightarrow{\;\sim\;} \End_{G_{\kappa}}(\Ind_{H_{\kappa}}^{G_{\kappa}}\pi)
\xrightarrow{\;\sim\;}\End_{\Kft}(\Ind_{\Hft}^{\Kft}(\pi\otimes \sigma))\xrightarrow{\;\sim\;}\Hcal(\Kft,\Hft;(\pi\otimes \sigma)^{\vee}),
\end{equation}
where the first and the third isomorphisms are as in Lemma \ref{HeckeAsEnd}, and the second isomorphism is by Theorem \ref{FiniteShimCor} and Proposition \ref{SCInd} \eqref{EndIso}. Hence we have the isomorphism of the two Hecke algebras.

Note that $\,^{\xt}(\pi\otimes\sigma)\cong\,^{\xt}\pi\otimes\sigma$ because $\sigma$ is a $\Kft$-representation. By Lemma \ref{L:Bushnell_Kutzko_equivalence},
\[
\Hcal(G_{\kappa}, H_{\kappa}; \pi^{\vee})_{x}\cong\Hom_{H_\kappa\cap{\,^{x}H_\kappa}}(\,^{x}\pi, \pi)\qand
\Hcal(\Kft, \Hft; (\pi\otimes\sigma)^{\vee})_{\xt}\cong\Hom_{\Hft\cap{\,^{\xt}\Hft}}(\,^{\xt}\pi\otimes\sigma, \pi\otimes\sigma).
\]
Since $\Hom_{H_\kappa\cap{\,^{x}H_\kappa}}(\,^{x}\pi, \pi)\cong \Hom_{\Hft\cap{\,^{\xt}\Hft}}(\,^{\xt}\pi, \pi)$, Lemma \ref{HOMotimesSig} (with $M=\Gammaft(1)$ and $L=\Hft\cap{\,^{\xt}\Hft}$) gives
\[
\Hom_{H_\kappa\cap{\,^{x}H_\kappa}}(\,^{x}\pi, \pi)\cong \Hom_{\Hft\cap{\,^{\xt}\Hft}}(\,^{\xt}\pi\otimes\sigma, \pi\otimes\sigma).
\]
Here, it should be mentioned that $^{\xt}\Gammaft(1)=\Gammaft(1)$, so $\Gammaft(1)\subseteq \Hft\cap{\,^{\xt}\Hft}$. Hence we have
\[
\Hcal(G_{\kappa}, H_{\kappa}; \pi^{\vee})_{x}\cong \Hcal(\Kft, \Hft; (\pi\otimes\sigma)^{\vee})_{\xt}.
\]

It remains to show that this isomorphism is induced by the composite of the three isomorphisms of \eqref{E:composition_of_three_isomorphism_fintie_Hecke}.  But this is a matter of book keeping since these three isomorphisms are explicit, and is left to the reader.
\end{proof}

This support preserving property has significant consequences when we specialize to the case of parabolic induction. Let
\[
P_\kappa=M_\kappa N_\kappa\subseteq G_\kappa
\]
be a parabolic subgroup of $G_\kappa$, where $M_\kappa$ and $N_\kappa$ are the Levi part and the unipotent subgroup, respectively. Then there exists a subgroup $P\subseteq K$ such that $P\supseteq\Gamma (1)$ and the corresponding $\Pft\subseteq\Kft$ is the preimage of $P_\kappa$ under the surjection $\Kft\twoheadrightarrow G_\kappa$.  We then have

\begin{Cor}\label{ParaHeckeIso}
Let $\pi$ be a finite dimensional representation of $M_\kappa$ which we inflate to a representation of $P_\kappa$. We then have a support-preserving $\C$-algebra isomorphism 
\[
\Hcal(G_{\kappa},P_{\kappa};\pi^{\vee})\cong\Hcal(\Kft,\Pft;(\pi\otimes \sigma)^{\vee}).
\]
In particular, if $B_\kappa$ is the Borel subgroup, we have
\[
\Hcal(G_{\kappa},B_{\kappa};\pi^{\vee})\cong\Hcal(\Kft,\Ift;(\pi\otimes \sigma)^{\vee}).
\]
\end{Cor}
\begin{proof}
    Just let $H=P$ in the above theorem.
\end{proof}

Hecke algebras of the form $\Hcal(G_{\kappa},P_{\kappa};\pi^{\vee})$, where $\pi$ is cuspidal, were thoroughly studied in \cite{HL80}. We highlight a couple of their results. (Here we will be translating the statements of \cite{HL80}, which primarily use the language of endomorphism algebras, into the language of Hecke algebras.) First, \cite[Proposition (3.9)]{HL80} (for which they refer to Springer \cite[pp. 635-636]{S71}) states that the Hecke algebra $\Hcal(G_{\kappa},P_{\kappa};\pi^{\vee})$ has a basis indexed by a subset of $P_{\kappa}$-double cosets such that the support of a basis function $T_{w}$ is the corresponding $P_{\kappa}$-double coset $P_{\kappa}\,w\,P_{\kappa}$.  Second, in terms of this basis, \cite[Theorem (4.14)]{HL80} describes a presentation for $\Hcal(G_{\kappa},P_{\kappa}; \pi^{\vee})$ (roughly braid relations and quadratic relations). The isomorphism of Corollary \ref{ParaHeckeIso} then implies that $\Hcal(\Kft,\Pft; (\pi\otimes \sigma)^{\vee})$ inherits the presentation of $\Hcal(G_{\kappa},P_{\kappa}; \pi^{\vee})$, but this isomorphism also preserves support. Thus if $\Tt_{\wt}\in \Hcal(\Kft, \Pft; (\pi\otimes \sigma)^{\vee})$ corresponds to $T_{w}\in \Hcal(G_{\kappa}, P_{\kappa}; \pi^{\vee})$, then we know that $\supp(\Tt_{\wt})=\Pft\,\wt\, \Pft$ and under the surjection $\Kft\twoheadrightarrow G_{\kappa}$ we have $\wt\mapsto w$.

\subsection{Two isomorphic Hecke algebras $\Hcal$ and $\Hcal_2$}\label{HandH2}

Now, let us introduce the Iwahori Hecke algebra we will study. Recall that we have first constructed a pseudo-spherical representation $\tau$ of the compact group $\Tt_0$. The kernel of $\tau$ contains $T_R$, which allows us to view $\tau$ as a representation of the finite group $\Tft=\Tt_0\slash T_R$, which we also denote by $\tau$. By parahoric induction, we have constructed a representation  $\sigma=i_{{\Uf}_{0},{\Uf}_{1}^{-}}^{\Ift}\tau$ of the finite group $\Ift=\It\slash \s(\Gamma(2e))$. By inflating $\sigma$ by the surjection $\It\twoheadrightarrow \Ift$, we view $\sigma$ as a representation of $\It$, which we also call $\sigma$. We then define 
\begin{equation*}
	\mathcal{H}:=\mathcal{H}(\Gt,\widebar{I};\sigma^{\vee}).
\end{equation*}

For our analysis, however, it will be convenient to also consider another Hecke algebra constructed using the congruence subgroup 
\[
I_{2}:=\Gamma_{0}(2).
\]
Note the inclusions $I_2\subseteq I\subseteq K$. By using the Iwahori decomposition $(U_0, T_0, U_2^-)$ of $I_2$, we consider another parahoric induction
\[
\sigma_{2}:=i_{{\Uf}_{0},{\Uf}_{2}^{-}}^{\Ift_2}\tau,
\]
which we often view as a representation of $\It_2$ by inflating via $\It_2\twoheadrightarrow \Ift_2$. We then define 
\begin{equation*}
	\Hcal_{2}:=\mathcal{H}(\Gt,\It_2;\sigma_{2}^{\vee}).
\end{equation*}

These two Hecke algebras are actually isomorphic:
\begin{Prop}\label{HIsoH2}
	As $\mathbb{C}$-algebras there is an isomorphism
	\begin{equation*}
		\Hcal\cong \Hcal_{2}.
	\end{equation*}
\end{Prop}

\begin{proof}
	The claim follows directly from Proposition \ref{IndSig2} and Lemma \ref{L:Bushnell_Kutzko_induced}. 
\end{proof}

Let $\Xi:\Hcal\rightarrow \Hcal_{2}$ be the isomorphism given by Proposition \ref{IndSig2} and Lemma \ref{L:Bushnell_Kutzko_induced}. We will use the same symbols for generators of $\Hcal$ and $\Hcal_{2}$ and their transported images, when the ambient algebra is clear. 

\begin{Rmk}
We work with $\Hcal_{2}$ primarily because the group $\It_2$ behaves better with the affine reflection $w_{\af}$ than $\It$. For example, the group generated by $\It_2$ and $w_{\af}$ is compact, whereas the one generated by $\It$ and $w_{\aff}$ is not. Accordingly, when we work with $w_{\af}$, we often work with $\Hcal_2$.
\end{Rmk}


\subsection{Support of $\Hcal$: multiplicity one}\label{SS:MultOne}


In this subsection we begin our study of the support of $\Hcal$. Let us start with

\begin{Prop}\label{HMult1}
	For each $g\in \Gt$, we have
	\begin{equation*}
		\dim\Hcal_{g}\leq 1.
	\end{equation*}
\end{Prop}

\begin{proof}
	First note that by Lemma \ref{L:Bushnell_Kutzko_equivalence} it suffices to compute the dimension of the Hom-space $\Hom_{\It\cap\,^{g}\It}(\,^{g}\sigma,\sigma)$. Second, note that if $\It g\It=\It g'\It$ then $\dim\Hcal_{g^{-1}}=\dim\Hcal_{{g'}^{-1}}$, and hence by the Iwahori-Bruhat decomposition, we may assume $g\in N_{\Gt}(\Tt)$.
	
	Let $\Gamma:=\Gamma(1)$. Since $ \Gamma\subseteq I$ we have $\Hom_{\It\cap\,^{g}\It}(\,^{g}\sigma,\sigma)\subseteq \Hom_{\widebar{\Gamma}\cap\,^{g}\widebar{\Gamma}}(\,^{g}\sigma,\sigma)$. Thus it suffices to prove 
    \[
    \dim\Hom_{\widebar{\Gamma}\cap\,^{g}\widebar{\Gamma}}(\,^{g}\sigma,\sigma)\leq 1,
    \]
    where we recall from Proposition \ref{IRep} \ref{ResSig} that $\sigma|_{\Gammat}$ is still irreducible. (The advantage of $\Gamma$ over $I$ is that $\Gamma$ is normalized by $\Wcal$.)
	
	As $g\in N_{\Gt}(\Tt)$ there exist $w\in \Wcal$ and $t\in \Tt$ such that $g=tw$. Then by Proposition \ref{IRep} we know $^w\sigma\cong\sigma$ as $\Gammat$-modules, which implies that $\Hom_{\widebar{\Gamma}\cap\,^{g}\widebar{\Gamma}}(\,^{g}\sigma,\sigma)\cong \Hom_{\widebar{\Gamma}\cap\,^{t}\widebar{\Gamma}}(\,^{t}\sigma,\sigma)$. Hence we may assume $g=t\in\Tt$.
	
	Next, note that there exists $w_{0}\in\Wcal$ such that 
    \[
    (w_0tw_0)U_j(w_0tw_0)^{-1}\subseteq U_j
    \]
    for any $j\in\Z$, where we recall the notation $U_j=U_j^+$ from \eqref{E:notation_U_alpha_j}. By Lemma \ref{HomIsoRel} and Proposition \ref{IRep} \ref{sigWInvar} and \ref{ResSig}, we have
	\begin{equation*}
		\Hom_{\widebar{\Gamma}\cap\,^{t}\widebar{\Gamma}}(\,^{t}\sigma,\sigma)\cong \Hom_{\widebar{\Gamma}\cap\,^{w_{0}tw_{0}^{-1}}\widebar{\Gamma}}(\,^{w_{0}tw_{0}^{-1}}\sigma,\sigma).
	\end{equation*}
	Thus we may assume $g=t\in \Tt$ such that $tU_{j}t^{-1}\subseteq U_{j}$ for any $j\in \mathbb{Z}$. 
	
	In this case, we see that $\widebar{\Gamma}\cap\,^{t}\widebar{\Gamma}$ has an Iwahori decomposition of the form
	\begin{equation*}
		\widebar{\Gamma}\cap\,^{t}\widebar{\Gamma}\cong U^{-}_{1}\times \Tt_{1} \times \,^{t}U_{1}.
	\end{equation*}
	Specifically, we have $U^{-}_{1}\subseteq \widebar{\Gamma}\cap\,^{t}\widebar{\Gamma}$, and $\,^{t}U_{1}\subseteq U_{1}$.
	
	Let $A\subseteq \,^{t}\sigma$ be the span of the elements of the form $e_{U_{1}}e_{U^{-}_{1}}\otimes_{\Tt_{1}} w$, where $w\in\tau$. Note that $A$ is a $\Tt_{1}$-module such that $A\subseteq \,^{t}\sigma^{\,^{t}U_{1}}=\sigma^{U_{1}}$ because the idempotent $e_{U_1}$ absorbs the elements in $U_1$. Also note that, as a $\widebar{\Gamma}\cap\,^{t}\widebar{\Gamma}$-module, $\,^{t}\sigma$ is generated by $A$, because of the above Iwahori decomposition of $\widebar{\Gamma}\cap\,^{t}\widebar{\Gamma}$ and the containment $\widebar{\Gamma}\cap\,^{t}\widebar{\Gamma}\supseteq U^{-}_{1}$. Thus by restriction from $\,^{t}\sigma$ to $A$ we get an injective map
	\begin{equation*}
		\Hom_{\widebar{\Gamma}\cap\,^{t}\widebar{\Gamma}}(\,^{t}\sigma,\sigma)\hookrightarrow\Hom_{\Tt_{1}}(A,\sigma^{U_{1}}).
	\end{equation*}
        Now, as $\Tt_{1}$-modules, $\sigma^{U_{1}}\cong \tau$ by Proposition \ref{P:two_functors_for_Iwahori_decomp} (4). Therefore $\dim\Hom_{\Tt_{1}}(A,\sigma^{U_{1}})\leq 1$, since $\dim A\leq \dim\tau$, proving the proposition.
\end{proof}

The following gives a sufficient condition for $\mathcal{H}_{g}$ to be $0$.

\begin{Prop}\label{HVanish}
	Let $f\in \Hcal$, $w\in \Wcal$ and $y\in Y\smallsetminus\Yt$. Let $\yt(\varpi)\in \Tt$ be any preimage of $y(\varpi)\in T$. Then $f(\yt(\varpi) w)=0$. In particular, 
	\begin{equation*}
		\Hcal_{\yt(\varpi) w}=0.
	\end{equation*}
\end{Prop}
\begin{proof}
	Since $y\notin\Yt$, there exists a simple root $\alpha\in\Delta$ such that $\la \alpha, y\ra\notin 2\Z$. Now, let $t\in (1+\varpi^{2e}\Ocal)\smallsetminus \Ocal^{\times 2}$. Then since $1+\varpi^{2e}\Ocal\subseteq R$, we know $(t, \varpi)=-1$ by Corollary \ref{C:Hilbert_symbol_unique_element_integral_radical}. Thus
	\[
	\htt_{\alpha}(t) \yt(\varpi)=(t, \varpi^{\la\alpha, y\ra})\yt(\varpi)\htt_{\alpha}(t)=-\yt(\varpi)\htt_{\alpha}(t).
	\]
	Recalling that $\Kt$ normalizes $\Gamma(2e)$, we have
	\[
	w^{-1}\htt_{\alpha}(t)w\in T_{2e}.
	\]
	Since $T_{2e}$ acts trivially on $\sigma$, we have
	\[
	f(\yt(\varpi) w)=f(\htt_{\alpha}(t)\yt(\varpi) w)=-f(\yt(\varpi)ww^{-1}\htt_{\alpha}(t)w)=-f(\yt(\varpi)w).
	\]
	Hence $f(\yt(\varpi) w)=0$.
\end{proof}

\subsection{Linear Hecke algebra for $I_2$}\label{SS:LinearHecke}


Here we take an interlude to study a linear Hecke algebra which will help us investigate several aspects of $\Hcal_{2}$. (A similar idea is used by Savin \cite{S04} in the tame case.) In particular, we record generalizations of the results of \cite[Section 5]{Karasiewicz} to the present setting. All the results generalize, but we only record those that we directly use later. The proofs carry over without change once the correct generalization is used, so we will be brief and refer the reader to \cite[Section 5]{Karasiewicz} for more details. The main point is that the mod $4$ congruence conditions in the $F=\mathbb{Q}_{2}$ setting of \cite{Karasiewicz} should become mod $\varpi^{2}$ congruence conditions in the present setting of an arbitrary $2$-adic field (as opposed to mod $\varpi^{2e}$). Note also that the congruence subgroup $I_2:=\Gamma_{0}(2)$ in the current paper corresponds to $\Gamma_{0}(4)$ in \cite{Karasiewicz}.

Let $\underline{\Hcal}:=C_{c}^{\infty}(I_2\backslash G/I_2)$, the algebra of $I_2$-biinvariant compactly supported functions on $G$. The space $\underline{\Hcal}$ has multiplication given by convolution where we normalize the Haar measure so that $I_2$ has measure $1$. 

We consider a length function $\ell_{2}:G\rightarrow\mathbb{R}_{\geq 0}$ defined by
\begin{equation*}
	q^{\ell_{2}(g)}:=[I_2gI_2:I_2].
\end{equation*}
By composing with any set theoretic section $W_{\af}\rightarrow N_{G}(T)$, we can consider the function
\[
\ell_{2}:W_{\af}\rightarrow N_G(T)\rightarrow\mathbb{R}_{\geq 0},
\]
which we also denote by $\ell_2$, because one can see that this composite does not depend on the choice of the section because $T_{0}\subseteq  I_2$.

Now we state relevant generalizations as follows, for which we recall the notation $\Omegat, \Wt_{\aff}, \Wt_{\ea}$ and $\tilde{\ell}$ from \S \ref{SS:affine_Weyl_covering_group}.

\begin{Lem}\label{Len0elts}
    Let $w_{1}\in \Omegat$ and $w_{2}\in \widetilde{W}_{\af}$. Then $\ell_{2}(w_{1}w_{2})=\ell_{2}(w_{2})$ and $\s(w_{1})\in N_{G}( I_2)$.
\end{Lem}

\begin{proof}
	This generalizes \cite[Lemma 5.4]{Karasiewicz} and can be proved in the same way.
\end{proof}

\begin{Prop}\label{CosetProp}
	Let $g_{1},g_{2},g,\delta\in G$ be such that $\ell_{2}(g_{1}g_{2})=\ell_{2}(g_{1})+\ell_{2}(g_{2})$. Suppose that the coset $\delta I_2$ satisfies the two conditions
	\begin{itemize}
		\item $\delta I_2\subseteq  I_2g_{1} I_2$;
		\item $(\delta I_2)^{-1}g\subseteq  I_2g_{2} I_2$.
	\end{itemize} 
	Then
	\[
		I_2g I_2=  I_2g_{1}g_{2} I_2.
	\] 
\end{Prop}

\begin{proof}
	This follows by adapting the proof of \cite[Proposition 5.2]{Karasiewicz} and the subsequent remark. We note that this proof uses $\underline{\Hcal}$.
\end{proof}

\begin{Thm}\label{ProportionalLengths}
	Let $w\in \widetilde{W}_{\ea}$. Then
	\begin{equation*}
		\ell_{2}(w)=2\tilde{\ell}(w).
	\end{equation*}
\end{Thm}

\begin{proof}
	This is proved as in \cite[Theorem 5.5]{Karasiewicz} by studying the structure of $\underline{\Hcal}$.
\end{proof}


\subsection{Support: affine simple reflections}\label{SuppAfSimpRel}


Recall that the group $\Wt_{\af}$ is generated by the simple affine reflections $w_{\af}, w_{\alpha_1}, \dots, w_{\alpha_n}$, where $w_{\af}$ is defined as in \eqref{E:w_aff_definition}. In this section, we will show that our Hecke algebra $\Hcal$ is supported on these elements, namely

\begin{Prop} \label{affSupp_H}
	We have
	\begin{equation*}
		\Hcal_{w_{\alpha}}\cong \Hom_{\It\cap{^{w_{\alpha}}\It}}(\,^{w_{\alpha}}\sigma,\;\sigma)\cong \mathbb{C}
	\end{equation*}
    for all $\alpha\in\Delta\cup\{\af\}$.
\end{Prop}

Since we already know $\dim\Hcal_{w_{\alpha}}\leq 1$ (Proposition \ref{HMult1}) and $\Hcal\cong\Hcal_2$ (Proposition \ref{HIsoH2}), and since $\It w_{\alpha}^{-1}\It=\It w_{\alpha}\It$ and so $\Hcal_{w_{\alpha}}=\Hcal_{w_{\alpha}^{-1}}$, we know from Lemma \ref{L:Bushnell_Kutzko_induced} that Proposition \ref{affSupp_H} above is implied by the following proposition.
\begin{Prop} \label{affSupp}
	We have
	\begin{equation*}
		(\Hcal_2)_{w_{\alpha}}\cong \Hom_{\It_2\cap{^{w_{\alpha}}\It_2}}(\,^{w_{\alpha}}\sigma_{2},\;\sigma_{2})\cong \mathbb{C}
	\end{equation*}
    for all $\alpha\in\Delta\cup\{\af\}$.
\end{Prop}

Before we begin the proof we will set up some preliminary results. We focus on the case of $\alpha=\af$. The case of $\alpha\in\Delta$ can be proved similarly. We omit these details because Proposition \ref{affSupp_H} also follows directly from the finite Shimura correspondence. Alternatively, since we know that the $\It$-module $\sigma$ extends to a $\Kt$-module, we may apply the method of \cite{Takeda_Wood} to prove Proposition \ref{affSupp_H} for $\alpha\in\Delta$.

It will be convenient to have a slightly different model for the representation $\sigma_{2}$. Let 
\begin{equation}\label{E:profinite_model_sigma_2}
\sigma_{2}^{p}:=i_{U_{0},U^{-}_{2}}^{\It_2}\tau
\end{equation}
be the $\It_2$-module constructed by parahoric induction, where $\tau$ is the representation $\Tt_{0}$. Here, the superscript $p$ is for ``profinite" as opposed to ``finite". (Recall that $\sigma_{2}$ was first defined as a representation of the {\it finite} group $\Ift_{2}$ via the finite group version of parahoric induction and then inflated to the {\it profinite} group $\It_2$. Here, $\sigma_2^p$ is defined by the profinite version of parahoric induction.) But this is another model of $\sigma_2$. Namely, we have

\begin{Lem}
As $\It_2$-modules 
\[
\sigma_{2}\cong\sigma_{2}^{p}.
\]
\end{Lem}
\begin{proof}
This is a consequence of the following identity in $\sigma_{2}^{p}$
    \begin{equation*}
        f(g)e_{U_{0}}e_{U^{-}_{2}}\otimes_{\Tt_{0}}w=f(g)e_{U_{0}}e_{\textbf{s}(\Gamma_{2e})}e_{U^{-}_{2}}\otimes_{\Tt_{0}}w=f(g)e_{\textbf{s}(\Gamma_{2e})}e_{U_{0}}e_{U^{-}_{2}}\otimes_{\Tt_{0}}w,
    \end{equation*}
where $w\in\tau$. The main input needed for this identity is Proposition \ref{P:SplitNormIm} \eqref{Gam2eSplitNorm}.
\end{proof}

Hence by abuse of notation, we often denote this profinite model simply by $\sigma_2$

\begin{Rmk}
The reason for switching to this profinite model is that conjugation by $w_{\af}$ does not stabilize $\Gamma(2e)$ and thus does not interact well with the quotients by $\s(\Gamma(2e))$ that we used previously.
\end{Rmk}

\begin{proof}[Proof of Proposition \ref{affSupp}] 
As mentioned above, we only prove the proposition for $\alpha=\af$. Let
\[
I_{2}^{\af}:= I_2\cap \,^{w_{\af}} I_2\qand\sigma_{2}^{\af}:=i_{U^{-}_{2}\cap I_2^{\af},U_{0}}^{ I_2^{\af}}\tau.
\]
Since $\tau$ is $W$-invariant and $Z(\Tt)$-invariant, Proposition \ref{TwistAut} implies
\[
\,^{w_{\af}}\sigma_{2}^{\af}\cong \sigma_{2}^{\af}.
\]
We also know by Proposition \ref{P:two_functions_are_adjoint} and Lemma \ref{L:IndRes2x} that
\[
\Hom_{\It_2^{\af}}(\sigma_{2}^{\af},\sigma_{2})\cong \mathbb{C}.
\]

Hence, we can identify $\sigma_{2}^{\af}$ with the $\sigma_{2}^{\af}$-isotypic subspace of $\sigma_{2}$, and this identification is unique up to nonzero scaling. Let $A:\sigma_{2}^{\af}\hookrightarrow \sigma_{2}$ be the canonical $\It_2^{\af}$-inclusion of this isotypic subspace. Let $C:\sigma_{2}\twoheadrightarrow \sigma_{2}^{\af}$ be the canonical $\It_2^{\af}$-projection onto this isotypic subspace. Let $\,^{w_{\af}}C:\,^{w_{\af}}\sigma_{2}\twoheadrightarrow \,^{w_{\af}}\sigma_{2}^{\af}$ be the $\It_2^{\af}$-homomorphism obtained by twisting $C$ by $w_{\af}$. Note that $\,^{w_{\af}}C=C$ as elements of $\End(\sigma_{2})$. Let $B:\,^{w_{\af}}\sigma_{2}^{\af}\rightarrow \sigma_{2}^{\af}$ be an $\It_2^{\af}$-isomorphism, which is also unique up to nonzero scaling.

Now by construction $A\circ B\circ C\in\Hom_{\It_2^{\af}}(\,^{w_{\alpha}}\sigma_{2},\sigma_{2})$ and this map is nonzero. (We note though that this map may not be an isomorphism.)
\end{proof}

\begin{Rmk}
    Suppose that $e\geq 2$. Since $U_{\alpha_{0},2e}$ acts trivially on $\sigma_{2}^{\af}$ and $\,^{w_{\af}}\sigma_{2}^{\af}\cong \sigma_{2}^{\af}$ the canonical image of $\sigma_{2}^{\af}$ in $\sigma_{2}$ actually must be contained in $\sigma_{2}^{U_{-\alpha_{0},2e-4}}$. (When $e=1$ the group $U_{-\alpha_{0},2e-4}$ is not contained in $I_{2}$, but $U_{-\alpha_{0},0}$ already acts trivially on $\sigma_{2}$.)
    
\end{Rmk}

\subsection{Support: length $0$ elements}\label{SuppLen0}


In this subsection we show that $\Hcal_{2}$ is supported on the length $0$ elements $\Omegat\subseteq \Wt_{\ea}$. Let us first mention
\begin{Lem}\label{L:lenght_0_element_representative}
    The elements of $\Omegat$ are represented by elements in $N_{\Gt}(\Tt)\cap N_{\Gt}(\It_2)$.
\end{Lem}
\begin{proof}
This readily follows from  Lemma \ref{Len0elts} and Theorem \ref{ProportionalLengths}.
\end{proof}

We then have

\begin{Prop}\label{Len0Supp}
	For each $n\in N_{\Gt}(\Tt)\cap N_{\Gt}(\It_2)$, we have
	\begin{equation*}
		(\Hcal_{2})_{n}\cong\Hom_{\It_2\cap{^{n}\It_{2}}}(\,^{n}\sigma_{2},\sigma_{2})\cong \mathbb{C}\qand \Hcal_n\cong\Hom_{\It\cap{^n\It}}(\,^n\sigma, \sigma)\cong\C.
	\end{equation*}
    Hence both $\Hcal_2$ and $\Hcal$ are supported on the length 0 elements.
\end{Prop}

\begin{proof}
For the proof, we need to use the profinite model \eqref{E:profinite_model_sigma_2} for $\sigma_2$. Then the argument for $\Hcal_2$ is similar to Proposition \ref{IRep} \ref{SigAWInv} along with the observation that modulo $\Tt_{0}$ the element $n$ represents an element of $\Omegat\subseteq \widetilde{W}_{\ea}\cong W\ltimes \Yt$. Again we apply Proposition \ref{TwistAut}. The first bullet is satisfied because $\tau$ is $W$-invariant and $Y_{Q,n}$-invariant. The second bullet is satisfied by a minor variant of Lemma \ref{WCompat}. The assertion for $\Hcal$ follows from Proposition \ref{HMult1}, Proposition \ref{HIsoH2} and Lemma \ref{L:Bushnell_Kutzko_induced}.
\end{proof}


\subsection{Quadratic relations}\label{SS:QuadRel}


In this subsection we prove that the elements of $\Hcal$ constructed in \S \ref{SuppAfSimpRel} and \ref{SuppLen0} satisfy quadratic relations. Given our current tool set, we must consider several cases. We begin with the quadratic relations for the elements corresponding to simple linear reflections. For these elements we work with the Iwahori model $\Hcal$.

Let $\alpha\in \Delta$. By Proposition \ref{affSupp_H} we know that $\Hcal_{w_{\alpha}}\cong \mathbb{C}$. We need to identify a suitably normalized generator of this one dimensional vector space. To do this we will use the finite Shimura correspondence.

Let $\Hcal_{\Kt}\subseteq\Hcal$ be the subalgebra of functions supported on $\Kt$. Pulling back along the surjective group homomorphism $\Kt\twoheadrightarrow \Kft$ defines a $\mathbb{C}$-algebra isomorphism $\Hcal(\Kft,\Ift,\sigma^{\vee})\stackrel{\sim}{\rightarrow} \Hcal_{\Kt}$ such that it is support-preserving with respect to the surjection $\Kt\twoheadrightarrow \Kft$. Now we can apply the finite Shimura correspondence of Hecke algebras to describe the structure of $\Hcal(\Kft,\Ift,\sigma^{\vee})$. We make this precise in the next proposition.

\begin{Prop}\label{QuadBraidSimpLinReflect}
	For every $\alpha\in\Delta$ there exists $T_{\alpha}\in \Hcal_{w_{\alpha}}$ such that the follow identities hold. For all $\alpha,\beta\in \Delta$ we have 
	\begin{enumerate}
		\item $T_{\alpha}^{2}=(q-1)T_{\alpha}+q$;\label{QuadSimpLin}
		\item if $\langle\alpha,\beta\rangle=0$, then $T_{\alpha}T_{\beta}=T_{\beta}T_{\alpha}$;\label{CommSimpLin}
		\item if $\langle\alpha,\beta\rangle=-1$, then $T_{\alpha}T_{\beta}T_{\alpha}=T_{\beta}T_{\alpha}T_{\beta}$.\label{BraidSimpLin}
	\end{enumerate}
In particular, for each $\alpha\in\Delta$, the element $T_{\alpha}$ is invertible.
\end{Prop}

\begin{proof}
	Corollary \ref{ParaHeckeIso} describes a support preserving isomorphism 
	\begin{equation*}
		\Hcal(\Kft,\Ift,\sigma^{\vee})\cong \Hcal(G_{\kappa},B_{\kappa},1).
	\end{equation*}
	Iwahori \cite{I64} proved that for each $\alpha\in\Delta$ there exists a unique element $t_{\alpha}\in \Hcal(G_{\kappa},B_{\kappa},1)_{w_{\alpha}}$ satisfying the identities \eqref{QuadSimpLin}, \eqref{CommSimpLin}, \eqref{BraidSimpLin}. For each $\alpha\in\Delta$, define $T_{\alpha}\in \Hcal_{\Kt}$ to be the image of $t_{\alpha}$ under the chain of algebra isomorphisms $\Hcal_{\Kt}\cong\Hcal(\Kft,\Ift,\sigma^{\vee})\cong \Hcal(G_{\kappa},B_{\kappa},1)$ prescribed above. As the two isomorphisms preserve support, it follows that $T_{\alpha}\in \Hcal_{w_{\alpha}}$.
\end{proof}

It remains to consider the simple affine reflection $w_{\af}$. Here it will be convenient to work with the $I_2$-model $\Hcal_{2}$. Let $T_{\af}\in (\Hcal_{2})_{w_{\af}}$ be any nonzero element. First we will prove that $T_{\af}$ satisfies a weak quadratic relation, which will be valid for all Cartan types. Then we will refine this in types $A_r,D_r,E_{6}, E_{7}$, by normalizing $T_{\af}$ so that it satisfies the same quadratic relation that appears in Proposition \ref{QuadBraidSimpLinReflect}. The same quadratic relation is expected to hold in type $E_{8}$, but we are not able to prove it in this paper. We will make some comments about this at the end of this subsection.

\begin{Prop}\label{WeakQuadSimpAff}
	Let $T_{\af}\in (\Hcal_{2})_{w_{\af}}$ be nonzero. Then there exists $a,b\in\mathbb{C}$ with $b\neq 0$ such that
	\begin{equation*}
		T_{\af}^{2}=aT_{\af}+b.
	\end{equation*}
        In particular, $T_{\af}$ is invertible.
\end{Prop}
 
\begin{proof}
In this proof we will occasionally abuse notation and write $w_{\af}$ for the element $\wt_{\alpha_0}(\varpi^2)\in N_{\Gt}(\Tt)$.
 
By the definition of convolution and the Steinberg relations, 
	\begin{align*}
		\supp \,T_{\af}^{2}\subseteq&  \It_2\wt_{\alpha_{0}}(\varpi^{2}) \It_2\wt_{\alpha_{0}}(\varpi^{2}) \It_2\\
		=& \It_2\wt_{\alpha_{0}}(\varpi^{2})U_{\alpha_{0},2}\wt_{\alpha_{0}}(-\varpi^{2}) \It_2\\
		=& \It_2U_{-\alpha_{0},-2} \It_2\\
		=& \It_2\cup \It_2\wt_{\alpha_{0}}(\varpi) \It_2\cup \It_2\wt_{\alpha_{0}}(\varpi^{2}) \It_2.
	\end{align*}
    By Proposition \ref{HVanish} and the support preserving isomorphism $\Hcal\cong \Hcal_{2}$, we know that $T_{\af}^{2}(g)=0$ for any $g\in \It_2\wt_{\alpha_{0}}(\varpi) \It_2$. Therefore there exists $a,b\in\mathbb{C}$ such that
	\begin{equation*}
		T_{\af}^{2}=aT_{\af}+b.
	\end{equation*}
	
	We claim that $T_{\af}^{2}(1)=b\neq 0$. First, for arbitrary $g\in \Gt$, we have
	
	\begin{align*}
		T_{\af}^{2}(g)=&\int_{\It_2w_{\af}\It_2}T_{\af}(h)T_{\af}(h^{-1}g)dh\\
		=&\sum_{h\in\It_2w_{\af}\It_2/\It_2}T_{\af}(h)T_{\af}(h^{-1}g)\\
		=&\sum_{h\in U_{\alpha_{0},2}/U_{\alpha_{0},4}}T_{\af}(hw_{\af})T_{\af}(w_{\af}^{-1}h^{-1}g).
	\end{align*}
	Now we take $g=1$. Then
	\begin{align*}
		T_{\af}^{2}(1)=&\sum_{h\in U_{\alpha_{0},2}/U_{\alpha_{0},4}}T_{\af}(hw_{\af})T_{\af}(w_{\af}^{-1}h^{-1})\\
		=&\sum_{h\in U_{\alpha_{0},2}/U_{\alpha_{0},4}}\sigma_{2}(h)T_{\af}(w_{\af})T_{\af}(w_{\af}^{-1})\sigma_{2}(h^{-1})\\
        =& \sum_{h\in U_{\alpha_{0},2}/U_{\alpha_{0},4}}\sigma_{2}(h)T_{\af}(w_{\af})T_{\af}(w_{\af}\htt_{\alpha_{0}}(-1))\sigma_{2}(h^{-1})\\
        =&\sum_{h\in U_{\alpha_{0},2}/U_{\alpha_{0},4}}\sigma_{2}(h)T_{\af}(w_{\af})^{2}\sigma_{2}(\htt_{\alpha_{0}}(-1)h^{-1}),
	\end{align*}
    where we used $w_{\af}^{-2}=\htt_{\alpha_{0}}(-1)$ for the second equality.
    
        Next we apply Proposition \ref{affSupp} to compute $T_{\af}(w_{\af})^{2}\in\End(\sigma_{2})$. From the proof of Proposition \ref{affSupp} we know that the map $T_{\af}(w_{\af})$ factors into three pieces, $T_{\af}(w_{\af})=A\circ B\circ C$, where $A,B,C$ are as in the proof of Proposition \ref{affSupp}. Thus we have $T_{\af}(w_{\af})^{2}=A\circ B\circ C\circ A\circ B\circ C$. Now by definition $C\circ A = id_{\sigma_{2}^{\af}}$. Thus $A\circ B\circ C\circ A\circ B\circ C=A\circ B^{2}\circ C$. Now $B\in \Hom_{ \It_2^{\af}}(\,^{w_{\af}}\sigma_{2}^{\af},\sigma_{2}^{\af})$, where we have put $I_{2}^{\af}= I_2\cap \,^{w_{\af}} I_2$. Thus 
        \begin{align*}
            B^{2}\in\Hom_{ \It_2^{\af}}(\,^{w_{\af}}\sigma_{2}^{\af},\,^{w_{\af}^{-1}}\sigma_{2}^{\af})=&\Hom_{ \It_2^{\af}}(\sigma_{2}^{\af},\,^{w_{\af}^{-2}}\sigma_{2}^{\af})\\
            =&\Hom_{ \It_2^{\af}}(\sigma_{2}^{\af},\,^{\htt_{\alpha_{0}}(-1)}\sigma_{2}^{\af})\\
            =&\mathbb{C}\sigma_{2}^{\af}(\htt_{\alpha_{0}}(-1)).
        \end{align*}
        Therefore $T_{\af}(w_{\af})^{2} = c A\circ \sigma_{2}^{\af}(\htt_{\alpha_{0}}(-1))\circ C$ for some $c\in\mathbb{C}^{\times}$. Since $\sigma_{2}^{\af}(h_{\alpha_{0}}(-1))\circ C=C\circ \sigma_{2}(\htt_{\alpha_{0}}(-1))$, we have
        \begin{align*}
            T_{\af}^{2}(1)=&\sum_{h\in U_{\alpha_{0},2}/U_{\alpha_{0},4}}\sigma_{2}(h)T_{\af}(w_{\af})^{2}\sigma_{2}(\htt_{\alpha_{0}}(-1)h^{-1})\\
            =& c\sum_{h\in U_{\alpha_{0},2}/U_{\alpha_{0},4}}\sigma_{2}(h)\circ A\circ \sigma_{2}^{\af}(\htt_{\alpha_{0}}(-1))\circ C\circ\sigma_{2}(\htt_{\alpha_{0}}(-1)h^{-1})\\
            =&(-1,-1)c\sum_{h\in U_{\alpha_{0},2}/U_{\alpha_{0},4}} \sigma_{2}(h)\circ A\circ C\circ\sigma_{2}(h^{-1}).
        \end{align*}

        By taking the trace and noting that $A\circ C$ is just a projection operator we see that 
        \begin{equation*}
            \Tr(T_{\af}^{2}(1))=(-1,-1)c\left|\Ocal\slash\varpi\Ocal\right|^{2}\Tr(A\circ C)\neq 0.
        \end{equation*}
        Thus $b=T_{\af}^{2}(1)\neq 0$. This implies $T_{\af}$ is invertible.
    \end{proof}

Now we specialize to the Cartan types $A_r,D_r,E_{6},E_{7}$ and prove an exact quadratic relation. The feature of these types that we exploit is the existence of an automorphism of $\Gt$ with good properties. 

We will construct this automorphism using $\Gbb^{ad}$ the adjoint form of $\Gbb$. Let $\Tbb^{ad}$ be the image of $\Tbb$ under the canonical surjection $\Gbb\rightarrow \Gbb^{ad}$. Let $G^{ad}:=\Gbb^{ad}(F)$ and $T^{ad}:=\Tbb^{ad}(F)$ be the groups of $F$-points. The group $G^{ad}$ acts on $G$ by automorphisms and because $G$ is simply-connected this action uniquely lifts to an action of $G^{ad}$ on $\Gt$ by automorphisms such that the action is trivial on $\mu_{2}$ \cite[pages 5,6]{BD01}. For any $g\in G^{ad}$, let $\varphi_{g}$ be the automorphism of $\Gt$ induced by the element of $g\in G^{ad}$.

We will build up to this automorphism in three steps. First we will work at the level of affine Weyl groups and the apartment of $T$. Second, we use this to identify an element of $G^{ad}$ with the desired action on $G$. Third, we lift the automorphism from $G$ to $\Gt$.

Recall that we have the root datum $(X,\Phi,Y,\Phi^{\vee})$ and the modified root datum $(\Xt,\frac{1}{2}\Phi,\Yt,2\Phi^{\vee})$. Let $Y\subseteq P^{\vee}\subseteq Y\otimes \Q$ be the group of coweights with respect to the roots $\Phi$. Then $\Pt^{\vee}:=2P^{\vee}$ is the group of coweights with respect to the roots $\frac{1}{2}\Phi$.

Let $\mathscr{\Wt}_{ea}^{ad}$ be the group of affine transformations of $Y\otimes \R$ generated by the groups $\Wt_{ea}$ and $\Pt^{\vee}$. Let $\Omegat^{ad}\subseteq \mathscr{\Wt}^{ad}_{ea}$ be the subgroup stablizing the alcove $2\mathfrak{A}\subseteq Y\otimes \R$. It is well-known that this group is isomorphic to $\Pt^{\vee}/2Y$.

When $\Phi$ is simply-laced and not of type $E_{8}$, then there exists a simple root $\alpha_{j_{0}}$ such that in the expression $-\alpha_{0}=\sum c_{j}\alpha_{j}$, we have $c_{j_{0}}=1$. The set of possible $\alpha_{j_{0}}$ is exactly the orbit of $\alpha_{0}$ under the automorphisms of the affine Dynkin diagram induced by $\Omegat^{ad}$, excluding $\alpha_{0}$ itself. Alternatively, this set is in bijection with the nontrivial cosets in $\Pt^{\vee}/2Y$. In particular, the simple root $\alpha_{j_{0}}$ corresponds to the coset $2\omega_{j_{0}}^{\vee}+2Y\in\Pt^{\vee}/2Y$, where $\omega_{j_0}^\vee$ is the fundamental coweight associated with $\alpha_{j_0}$. (See \cite[Ch.\ 6, \S 2.3, Prop.\ 6]{BourLGLA46} for details.)

Now the Tits homomophism is a surjective group homomorphism $\vartheta:N_{G^{ad}}(T^{ad})\rightarrow \mathscr{\Wt}^{ad}_{ea}$ \cite{T79}. Let $N_{G^{ad}}(I_{2})$ be the elements of $G^{ad}$ that induce automorphisms of $G$ that stabilize $I_{2}$. Then it can be checked directly that $\vartheta^{-1}(\Omegat^{ad})=N_{G^{ad}}(T^{ad})\cap N_{G^{ad}}(I_{2})$.

Since we are excluding $E_{8}$, the group $\Omegat^{ad}$ is nontrivial and so $N_{G^{ad}}(T^{ad})\cap N_{G^{ad}}(I_{2})$ is nontrivial.

\begin{Lem}\label{StabGam2}
	Suppose that $\Phi$ is of type $A_r,D_r$, $E_{6}$, or $E_{7}$. Let $n\in N_{G^{ad}}(T^{ad})\cap N_{G^{ad}}( I_2)$ such that $\vartheta(n) = \varepsilon\in \Omegat^{ad}$ is nontrivial. Let $\alpha_{j_{0}}:=\varepsilon(\alpha_{0})\in \Delta$. Then
	\begin{enumerate}
		\item $\varphi_{n}(\Tt)=\Tt$,\label{TInv}
		\item $\varphi_{n}(\It_2)=\It_2$,\label{Gamma2Inv}
		\item $\varphi_{n}(\It_2\wt_{\alpha_{j_{0}}}\It_2)=\It_2w_{\af}\It_2$.\label{Gam2AfDoubCoset}
	\end{enumerate}
\end{Lem}

\begin{proof}
	All three properties follow directly by construction.
\end{proof}

With this lemma, we can prove

\begin{Prop}\label{QuadSimpAff} 
Let $\Phi$ be of type $A_r$,$D_r$, $E_{6}$ or $E_{7}$. Then there exists $T_{\af}\in(\Hcal_{2})_{w_{\af}}$ such that 
\begin{equation*}
		T_{\af}^{2}=(q-1)T_{\af}+q.
	\end{equation*}
\end{Prop}

\begin{proof}
    Let $n\in N_{G^{ad}}(T^{ad})\cap N_{G^{ad}}( I_2)$ as in Lemma \ref{StabGam2}. We claim that $\,^{\varphi_{n}}\sigma_{2}\cong \sigma_{2}$. Assume this for now. Then there exists $A_{n}\in \GL(\sigma_2)$ such that $A_{n}\,^{\varphi_{n}}\sigma_{2}(g)=\sigma_{2}(g)A_{n}$ for all $g\in  \It_2$.

    Define
    \[
        T_{\af}:=A_{n}(T_{\alpha_{j_{0}}}\circ\varphi_{n}^{-1})A_{n}^{-1}\in\Hcal_{2},
    \]
    by which we mean $T_{\af}$ is such that $T_{\af}(g)=A_{n}(T_{\alpha_{j_{0}}}\circ\varphi_{n}^{-1})(g))A_{n}^{-1}$. One can then see from Lemmas \ref{L:HeckeAutTwist} and \ref{StabGam2} that
    \[
        T_{\af}\in(\Hcal_2)_{w_{\af}}.
    \]
    The quadratic relation follows from the quadratic relation for $T_{\alpha_{j_{0}}}$ (Proposition \ref{QuadBraidSimpLinReflect}).

    It remains to prove that $\,^{\varphi_{n}}\sigma_{2}\cong \sigma_{2}$. By Lemma \ref{IsoTwist} we know that 
    \[
        \,^{\varphi_{n}}\sigma_{2}\cong i_{\varphi_{n}(U_{0}),\varphi_{n}(U_{2}^{-})}\,^{\varphi_{n}}\tau.
    \]
     Because the two Iwahori factorizations are compatible, Proposition \ref{P:compatible_Iwahori_decomp} implies
    \[
        i_{\varphi_{n}(U_{0}),\varphi_{n}(U_{2}^{-})}\,^{\varphi_{n}}\tau\cong i_{U_{0},U_{2}^{-}}\,^{\varphi_{n}}\tau.
    \]
    Thus it suffices to prove that $\,^{\varphi_{n}}\tau\cong \tau$ as $\Tt_{0}$-representations. The automorphism $\varphi_{n}$ can be factored as
    \[
        \varphi_{n} = \varphi_{n_{1}}\circ \mathrm{Int}_{n_{2}},
    \]
    where $n_{1}\in T^{ad}$ and $n_{2}\in \Wcal$. In fact we know that $n_{1} = \omega_{j_{0}}^{\vee}(\varpi^{2})$, and one can check using the Steinberg relations that for any $\alpha\in \Delta$ and $u\in F^{\times}$ we have $\varphi_{\omega_{j_{0}}^{\vee}(v)}(\htt_{\alpha}(u)) = (u,v)^{\la\alpha,\omega_{j_{0}}^{\vee}\ra}\htt_{\alpha}(u)$. So in particular $\varphi_{n_{1}}|_{\Tt} = id_{\Tt}$, and thus $\,^{\varphi_{n_{1}}}\tau = \tau$. Finally, we know from Lemma \ref{T:PseudoT0} that $\,^{\mathrm{Int}_{n_{2}}}\tau=\,^{n_{2}}\tau\cong \tau$. This completes the proof of the proposition.
\end{proof}

Finally we prove that the length $0$ elements satisfy a quadratic relation.

\begin{Prop}\label{QuadLen0}
	Let $\varepsilon\in \Omegat$. There exists $T_{\varepsilon}\in (\Hcal_{2})_{\varepsilon}$ (unique up to sign) such that 
	\begin{equation*}
		T_{\varepsilon}^{2}=1.
	\end{equation*}
	In particular, $T_{\varepsilon}$ is invertible.
\end{Prop}

\begin{proof}
	Let $t_{\varepsilon}\in (\Hcal_{2})_{\varepsilon}$ be nonzero. Note that $\Omegat\cong Y_{Q,n}/2Y$ is an elementary abelian $2$-group. With this and the definition of convolution we have $\supp(t_{\varepsilon}^{2})\subseteq  \It_2 \varepsilon^{2}= \It_2$. Thus we see that $t_{\varepsilon}^{2}=C$ for some $C\in \mathbb{C}$.
	
	Note that the subgroup $\langle \It_2, \varepsilon\rangle\subseteq\Gt$ generated by $\It_2$ and $\varepsilon$ has $\It_2$ as a subgroup of index $2$. Thus by Proposition \ref{Len0Supp} and Clifford theory we know that $\Ind_{\It_2}^{\langle \It_2, \varepsilon\rangle}(\sigma_{2})$ decomposes into two distinct irreducible representations $\pi_{1},\pi_{2}$ extending $\sigma_{2}$. Then because $t_{\varepsilon}$ can be viewed as an element of the semisimple algebra $\End_{\langle \It_2, \varepsilon\rangle}(\Ind_{\It_2}^{\langle \It_2, \varepsilon\rangle}(\sigma_{2}))$ it follows that $C\neq 0$. Thus if we  define $T_{\varepsilon}:=(C)^{-\frac{1}{2}}t_{\varepsilon}$ (unique up to sign), then $T_{\varepsilon}^{2}=1$.
\end{proof}


\subsection{Braid relations}\label{SS:BraidRel}


In this subsection we investigate the braid relations. We begin by proving the braid relations up to scaling.

For $w\in \Wt_{\ea}$, we let $t_{w}\in (\Hcal_{2})_{w}$. By Proposition \ref{HMult1} the element $t_{w}$ is unique up to scaling. Later we will be concerned with normalizing the elements $t_{w}$, but for now this is not important.

\begin{Prop}\label{Prop:GenBraid}
	Let $w_{1},w_{2}\in \Wt_{\ea}$. If $\ell(w_{1}w_{2})=\ell(w_{1})+\ell(w_{2})$, then in $\Hcal_{2}$ we have 
    \[
    t_{w_{1}}*t_{w_{2}}=C(w_{1},w_{2})t_{w_{1}w_{2}},
    \]
    for some $C(w_{1},w_{2})\in\mathbb{C}$. Furthermore, if $t_{w_{1}}$ is invertible and $t_{w_{2}}\neq 0$, then $C(w_{1},w_{2})\in \mathbb{C}^{\times}$.
	
\end{Prop}

\begin{proof}
	Our proof follows \cite[Proposition 6.8]{Karasiewicz}, which adapted the argument of \cite[Proposition 6.2]{S04}.
	
	For any $g\in \Gt$ we need to compute $t_{w_{1}}*t_{w_{2}}(g)$. By the definition of convolution we need to identify the cosets $\delta\It_2\subseteq \Gt$ such that
	\[
		\delta\It_2\subseteq \It_2w_{1}\It_2\qand
		(\delta\It_2)^{-1}g\subseteq \It_2w_{2}\It_2.
	\]
	These containments only depend on $G$. Thus by Theorem \ref{ProportionalLengths} and Proposition \ref{CosetProp}, which uses the length hypothesis, it follows that $\supp\,t_{w_{1}}*t_{w_{2}}\subseteq \It_2w_{1}w_{2}\It_2$.
	
	By Proposition \ref{HMult1} there exists $C\in\mathbb{C}$ such that $t_{w_{1}}*t_{w_{2}}=Ct_{w_{1}w_{2}}$. If we further assume that $t_{w_{1}}$ is invertible and $t_{w_{2}}\neq 0$, then $C$ cannot be $0$.
	\end{proof}
	
	Using Proposition \ref{Prop:GenBraid} we can define a basis of $\Hcal_{2}$ such that each element is invertible. 
	
	\begin{Lem}\label{SuppComplete}
		Let $w\in \widetilde{W}_{\ea}$. Let $\varepsilon\in\Omegat$ and let $\beta_{1},\ldots,\beta_{k}\in \Delta\cup\{\alpha_{0}\}$ be a finite sequence such that
		\begin{equation*}
			w=\varepsilon w_{\beta_{1}}\cdots w_{\beta_{k}}
		\end{equation*}
		is a reduced expression for $w$. Then the element $T_{\varepsilon}T_{\beta_{1}}\cdots T_{\beta_{k}}\in (\Hcal_{2})_{w}$ is nonzero and invertible.
	\end{Lem}
	
	\begin{proof}
		This follows by induction on the length of $w$ and by Propositions \ref{QuadBraidSimpLinReflect}, \ref{QuadSimpAff}, \ref{QuadLen0}, and \ref{Prop:GenBraid}. 
	\end{proof}
	
	\begin{Prop}\label{HLineDecompAndInvert}
		The set of subspaces $(\Hcal_{2})_{w}$, where $w\in \widetilde{W}_{\ea}$, is a line decomposition for $\Hcal_{2}$, namely
        \[
        \Hcal_2=\bigoplus_{w\in\Wt_{\ea}}(\Hcal_{2})_{w}.
        \]
        Furthermore, any nonzero element of $(\Hcal_{2})_{w}$ is invertible.
	\end{Prop}
	
	\begin{proof}
		This follows from Proposition \ref{HMult1}, Proposition \ref{HVanish} and Lemma \ref{SuppComplete}, along with the fact that the distinct $(\Hcal_{2})_{w}$'s have disjoint support.
	\end{proof}
	
	For the case of type $A_r,D_r$, $E_{6}$, $E_{7}$, we can refine Proposition \ref{HLineDecompAndInvert} to produce a normalized basis for $\Hcal_{2}$ for which the usual braid relations are satisfied.
	
	From \S \ref{SS:QuadRel} we have the following normalized elements. For $\alpha\in\Delta\cup\{\af\}$ there is a unique $T_{\alpha}\in (\Hcal_{2})_{w_{\alpha}}$ such that $T_{\alpha}^{2}=(q-1)T_{\alpha}+q$. For $\varepsilon\in\Omegat$ there is a unique up to sign $T_{\varepsilon}\in (\Hcal_{2})_{\varepsilon}$ such that $T_{\varepsilon}^{2}=1$. To begin we show that the elements indexed by $\Delta\cup\{\af\}$ satisfy the braid relations. 
    
    Let us first prove the following elementary fact.	
	\begin{Lem}\label{BraidLem}
		Let $a,b,c\in \mathbb{C}^{\times}$. Let $\mathcal{A}$ be an associative $\mathbb{C}$-algebra with unit. Suppose that $X,Y\in \mathcal{A}$ are such that
		\begin{itemize}
			\item $1,X,Y\in \mathcal{A}$ are linearly independent;
			\item $X^{2}=aX+b$;
			\item $Y^{2}=aY+b$;
			\item $XYX=cYXY$, or $XY=cYX$.
		\end{itemize}
		Then $c=1$.
	\end{Lem}
	
	\begin{proof}
		This follows by a direct calculation. We will prove the case where $XYX=cYXY$. The other case is easier.
		
		Since $b\neq 0$, the elements $X$ and $Y$ are invertible. Thus we know that $X=cYXYX^{-1}Y^{-1}$. Squaring and applying the quadratic relations gives
		\begin{equation*}
			aX+b = c^{2}YX(aY+b)X^{-1}Y^{-1}.
		\end{equation*}
		The right hand side simplifies to 
		\begin{equation*}
			c^{2}YX(aY+b)X^{-1}Y^{-1}=c^{2}(aYXYX^{-1}Y^{-1}+b)=acX+c^{2}b.
		\end{equation*}
		Now since $X$ and $1$ are linearly independent we have $a=ac$ and $b=bc^{2}$. Since $a\neq 0$ it follows that $c=1$.
 	\end{proof}
	
	\begin{Prop}\label{CompleteBraidRel}
		Let $\Phi$ be of type $A_r,D_r$, $E_{6}$ or $E_{7}$. Let $\alpha,\beta\in \Delta\cup\{\af\}$ be distinct.
		\begin{enumerate}
			\item If $\langle\alpha,\beta\rangle=0$, then $T_{\alpha}T_{\beta}=T_{\beta}T_{\alpha}$.
			\item If $\langle\alpha,\beta\rangle=-1$, then $T_{\alpha}T_{\beta}T_{\alpha}=T_{\beta}T_{\alpha}T_{\beta}$.
		\end{enumerate}
	\end{Prop}
	
	\begin{proof}
		This follows from Proposition \ref{Prop:GenBraid} and Lemma \ref{BraidLem}.
	\end{proof}
	
	\begin{Prop}\label{P:TwDef}
		Let $\Phi$ be of type $A_r,D_r$, $E_{6}$ or $E_{7}$. Let $w\in\widetilde{W}_{\af}$ and let $w=w_{1}\cdots w_{k}$ be a reduced expression for $w$ in terms of the affine simple reflections $w_{\alpha}$, where $\alpha\in\Delta\cup\{\af\}$. Define $T_{w}:=T_{w_{1}}\cdots T_{w_{k}}\in \Hcal_{2}$. Then the element $T_{w}$ is independent of the reduced expression for $w$.
	\end{Prop}
	
	\begin{proof}
		This follows from \cite[Th\'{e}or\`{e}me 3]{T69} and Proposition \ref{CompleteBraidRel}.
	\end{proof}
	
	\begin{Lem}\label{OmegaAut}
		Let $\Phi$ be of type $A_r,D_r$, $E_{6}$ or $E_{7}$. Let $\varepsilon\in\Omegat$ and $\alpha\in\Delta\cup\{\af\}$. Then $T_{\varepsilon}T_{\alpha}T_{\varepsilon}^{-1}=T_{\varepsilon(\alpha)}$.
	\end{Lem}
	
	\begin{proof}
		By Proposition \ref{Prop:GenBraid} we know that $T_{\varepsilon}T_{\alpha}T_{\varepsilon}^{-1}=cT_{\varepsilon(\alpha)}$ for some $c\in\mathbb{C}^{\times}$. Because $T_{\alpha}$ and $T_{\varepsilon(\alpha)}$ satisfy the same quadratic relation with nonzero linear coefficient it follows that $c=1$ by Lemma \ref{BraidLem}.
	\end{proof}

    \subsection{IM-presentation}\label{SS:IMPres}
    
    Finally we can prove that when $\Phi$ has Cartan type $A_r$, $D_{2r+1}$, $E_{6}$, $E_{7}$ the Hecke algebra $\Hcal$ has an Iwahori-Matsumoto presentation (IM-presentation for short). The same result is expected to hold for $D_{2r}$ and $E_{8}$, though we are not able to prove it in this paper.

        Let us first define an abstract algebra derived from the extended Weyl group $\Wt_{\ea}$ defined via generators and relations as in \cite[Section 3.1]{IM}. Let 
        \[
        H_{\af}=\la 1, t_{\alpha}\st \alpha\in\Delta\cup\{\af\}\ra
        \]
        be the unital $\mathbb{C}$-algebra generated by the symbols $t_{\alpha}$, where $\alpha\in\Delta\cup\{\af\}$, subject to the following relations:
        \begin{align*}
            &t_{\alpha}^{2}=(q-1)t_{\alpha}+q;\\
            &t_{\alpha}t_{\beta}=t_{\beta}t_{\alpha},\quad \text{if } \langle\alpha,\beta\rangle=0;\\
            &t_{\alpha}t_{\beta}t_{\alpha}=t_{\beta}t_{\alpha}t_{\beta},\quad \text{if } \langle\alpha,\beta\rangle=-1.
        \end{align*}
        Note that for $w\in \Wt_{\af}$ with a reduced expression $w=w_{\alpha_{i_{1}}}\cdots w_{\alpha_{i_{k}}}$ we can define $t_{w}=t_{\alpha_{i_{1}}}\cdots t_{\alpha_{i_{k}}}$. This definition is independent of the choice of reduced word by the same argument as in Proposition \ref{P:TwDef}.
        
        The group $\Omegat\subseteq \Wt_{\ea}$ acts on $H_{\af}$ by algebra automorphisms as follows. Recall that the elements of $\Omegat$ act on $\Delta\cup\{\af\}$ by permutation. Then for $\varepsilon\in\Omegat$ define $\varepsilon(t_{\alpha})=t_{\varepsilon(\alpha)}$. This induces an algebra automorphism of $H_{\af}$. Moreover, this is compatible with the action of $\Omegat$ on $\Wt_{\af}$ in the sense that $\varepsilon(t_{w})=t_{\varepsilon w\varepsilon^{-1}}$. Let 
        \[
        H_{\ea}:=\mathbb{C}[\Omegat]\otimes H_{\af}.
        \]
        This $\mathbb{C}$ vector space can be given the structure of a $\mathbb{C}$-algebra as follows. For $\varepsilon_{j}\in \Omegat$ and $w_{j}\in \Wt_{\af}$ define
        \begin{equation*}
            (\varepsilon_{1}\otimes t_{w_{1}})\cdot (\varepsilon_{2}\otimes t_{w_{2}}):=\varepsilon_{1}\varepsilon_{2}\otimes t_{\varepsilon_{2}^{-1}w_{1}\varepsilon_{2}}t_{w_{2}}
        \end{equation*}

        \begin{Thm}\label{IMPres} Assume that $G$ has Cartan type $A_r, D_{2r+1}, E_{6},E_{7}$. Then the linear map $H_{\ea}\rightarrow \Hcal$ defined by $\varepsilon\otimes t_{w}\mapsto T_{\varepsilon}T_{w}$, where $\varepsilon\in\Omegat$ and $w\in \Wt_{\af}$, is an isomorphism of $\mathbb{C}$-algebras.
        \end{Thm}

        \begin{proof}
            The map is a linear isomorphism because the set $\{\varepsilon\otimes t_{w}\st\varepsilon\in\Omegat,\,w\in\Wt_{\af}\}$ is a basis for $H_{\ea}$ and the set $\{T_{\varepsilon}T_{w}\st\varepsilon\in \Omegat,\,w\in\Wt_{\af}\}$ is a basis for $\Hcal$. The map is a $\mathbb{C}$-algebra homomorphism because of Propositions \ref{QuadBraidSimpLinReflect}, \ref{QuadSimpAff}, \ref{QuadLen0}, \ref{CompleteBraidRel}, and Lemma \ref{OmegaAut}.
        \end{proof}

        Theorem \ref{IMPres} and the work of \cite{IM} imply that $\Hcal$ is isomorphic to the Iwahori-Hecke algebra of a linear group. Specifically, let 
        \[
        G^{\prime}:=G/Z_{2},
        \]
        where $Z_{2}$ is the subgroup of $2$-torsion in the center of $G$. (The significance of $Z_{2}$ is that the fundamental group of $G^{\prime}$ is isomorphic to $\Yt/2Y$. Thus the dual root datum of $G^{\prime}$ agrees with the dual root datum of $\Gt$.) Let $I^{\prime}$ be an Iwahori subgroup of $G^{\prime}$ and let $\Hcal(G^{\prime},I^{\prime})$ be the Iwahori-Hecke algebra of locally constant $I^{\prime}$-biinvariant functions on $G^{\prime}$.

        \begin{Cor}\label{HeckeShimCorr}
            As $\mathbb{C}$-algebras, we have
            \[
            \Hcal(G^{\prime},I^{\prime})\cong\Hcal\cong H_{\ea}.
            \]
        \end{Cor}

        \begin{proof}
            This follows directly from Theorem \ref{IMPres} and \cite[Section 3]{IM}. 
        \end{proof}
	
We end this subsection with a few comments on the cases of $E_{8}$ and $D_{2r}$. We expect the same type of Iwahori-Matsumoto presentation to hold in the case of $D_{2r}$ and $E_{8}$ as well. When $F$ is unramified over $\mathbb{Q}_{2}$ the approach of \cite{Karasiewicz} can be generalized by incorporating some ideas from this current paper. We will not go into this because we believe that these cases for any $F$ will follow as we refine and expand the techniques used in this paper. 

For example, for $E_{8}$ the only obstruction is proving that the simple affine reflection satisfies the desired quadratic relation. This should follow from an analysis along the lines of \S \ref{SSecSplittings} for other vertices in the building of $G$ and an analogous finite Shimura correspondence at these vertices. For $D_{2r}$ the difficulty arises because in this case $\Omegat$ is not a cyclic group. The resolution of this issue is related to the question of extending $\sigma_{2}$ from $\It_2$ to its normalizer in $\Gt$. However, at the moment, certain basic properties of $\sigma_{2}$ remain unknown, such as its dimension. As we further investigate the basic properties of $\sigma_{2}$, this will naturally put us in a better position to address the extension question.


\subsection{Satake Isomorphism}\label{SS:Satake}


In this subsection we state the Satake isomorphism for our double cover $\Gt$. Everything in this subsection essentially follows from a Bernstein presentation of $\Hcal$, which can be deduced from an IM presentation. So, it applies when $G$ has Cartan type $A_{r},D_{2r+1},E_{6}$, $E_{7}$; it will apply to the other types when the remaining gaps in the IM presentation are closed.

We begin with a preliminary remark. We have a set inclusion $\Hcal(\Gt,\Kt,\sigma^{\vee})\subseteq \Hcal$, but because of our choice of normalizations of measures this inclusion is not an inclusion of algebras. To identify $\Hcal(\Gt,\Kt,\sigma^{\vee})$ as a subalgebra of $\Hcal(\Gt,\It,\sigma^{\vee})$ we use the algebra homomorphism $\Hcal(\Gt,\Kt,\sigma^{\vee})\hookrightarrow \Hcal$ defined by $f\mapsto \frac{1}{[K:I]}f$.

Let $Z(\Hcal)$ be the center of the Hecke algebra $\Hcal$. Recall that since $\Hcal$ satisfies an IM presentation deforming the extended affine Weyl group $\Wt_{ea}\cong W\ltimes \Yt$ it is known that $Z(\Hcal)\cong \mathbb{C}[\Yt]^{W}$ (e.g. \cite{L89} or \cite[Section 2]{HKP10}).

Let 
\[
e_{\sigma}:=\frac{1}{\sum_{w\in W}q^{\ell(w)}}\sum_{w\in W} T_{w}\in\Hcal.
\]
Note that since the $\It$-representation $\sigma$ extends to $\Kt$ we have $\Hcal(\Gt,\Kt,\sigma^{\vee})\cong e_{\sigma}*\Hcal*e_{\sigma}$. 

Now using the known structure of affine Hecke algebras one can prove the following corollary.

\begin{Cor}\label{C:Satake}
    Suppose that $\Hcal$ admits an IM-presentation. Then under the preceeding identification the map $Z(\Hcal)\rightarrow \Hcal(\Gt,\Kt,\sigma^{\vee})$ defined by $z \mapsto z*e_{\sigma}$ defines an isomorphism of $\mathbb{C}$-algebras. In particular, $\Hcal(\Gt,\Kt,\sigma^{\vee})$ is commutative.
\end{Cor}


\section{Iwahori types}\label{S:IwahoriType}


In this section we show that $(\It,\sigma)$ is a type in the sense of \cite[(4.1) Definition]{BK98}. We emphasize that this result is valid for any Cartan type. As a corollary, we derive a local Shimura correspondence when $G$ is of Cartan type $A_r, D_{2r+1}, E_6$ and $E_7$.

The following arguments essentially follow \cite{BK98}, which build on the Borel-Casselman theorem. However, there is one notable difference. Bushnell-Kutzko introduce the notion of $G$-cover \cite[(8.1) Definition]{BK98} as a way to axiomatize the construction of types for a reductive group $G$ from types of its Levi subgroups. Their main theorem for constructing types using $G$-covers is \cite[(8.3) Theorem]{BK98}. The definition of $G$-cover refers back to \cite[(6.1) Definition]{BK98} which contains two conditions (i) and (ii). Our proposed type $(\It,\sigma)$ satisfies condition (ii) by Proposition \ref{HLineDecompAndInvert}, but it does not satisfy condition (i). Specifically, $\sigma$ is neither trivial on $U_{0}$ nor $U^{-}_{1}$, when $e>1$. Fortunately, we are able to modify the argument of Bushnell-Kutzko using special properties of $\sigma$. For example, Lemma \ref{VeryPosSupp} plays a particularly important role. So, we include the details to make clear the necessary changes.

We begin with a few lemmas.

\begin{Lem}\label{L3to1}
	Let $G$ be a finite group with an Iwahori factorization $(V,T,U)$. Let $\pi\in \Irr(G)$ be such that $r_{U,V}\pi\neq 0$. The element $e_{U}e_{V}e_{U}\in \End({\pi})$ induces a linear isomorphism of the space $\pi^{U}$ and thus is a generator of the one-dimensional space $\Hom_{T}(\pi^{U},\pi^{U})=\mathbb{C}e_{U}$.
\end{Lem}

\begin{proof}
	By Proposition \ref{P:two_functors_for_Iwahori_decomp} the element $e_{U}e_{V}\in \Hom_{T}(\pi^{V},\pi^{U})$ is an isomorphism, and $r_{U,V}\pi\cong \pi^{U}\cong \pi^{V}$ is an irreducible $T$-module. Similarly, $e_{V}e_{U}\in \Hom_{T}(\pi^{U},\pi^{V})$ is an isomorphism. Hence $e_{U}e_{V}e_{U}=(e_Ue_V)(e_Ve_U)\in \Hom_{T}(\pi^{U},\pi^{U})=\mathbb{C}e_{U}$ is nonzero and the result follows.
\end{proof}

In the next lemma we prove that the Hecke algebra $\Hcal$ contains elements supported on certain positive elements of $Z(\Tt)$ with a natural normalization. We note that this lemma does not rely on any of the previous support computations in $\Hcal$. 

\begin{Lem}\label{VeryPosSupp}
	Let $z\in Z(\Tt)$ be such that $zU_{0}z^{-1}\subseteq U_{2e}$. Then the element $e_{\Uf_{0}}e_{\Uf^{-}_{1}}\in\Hcal(\,\Ift\,)$ defines a nonzero element in $\Hom_{\It\cap\,^{z}\It}(\sigma,\,^{z}\sigma)$.
\end{Lem}

\begin{proof}
	Note that since $zU_{0}z^{-1}\subseteq U_{2e}$, it also follows that $z^{-1}U^{-}_{1}z\subseteq U^{-}_{2e+1}$. This implies that $\It\cap\,^{z}\It$ has an Iwahori decomposition
	\begin{equation*}
		\It\cap\,^{z}\It\cong \,^{z}U_{0}\times\Tt_{0}\times U^{-}_{1}.
	\end{equation*} 
	 We have to show $e_{\Uf_{0}}e_{\Uf^{-}_{1}}\in \Hom_{\It\cap\,^{z}\It}(\sigma,\,^{z}\sigma)$. For this it suffices to show that it is $\Tt_{0}$-, $U_0$- and $U_1^-$-intertwining. But one can readily see that it is $\Tt_0$-intertwining because $z\in Z(\Tt)$. Next, since $\,^{z}U_{0}\subseteq U_{2e}\subseteq \ker\sigma$, we see that $\,^{z}U_{0}$ acts trivially on the domain. Now the image of $e_{\Uf_{0}}e_{\Uf^{-}_{1}}$ is contained in $\sigma^{U_{0}}=(\,^{z}\sigma)^{\,^{z}U_{0}}$. Thus $\,^{z}U_{0}$ also acts trivially on the codomain. Thus $e_{\Uf_{0}}e_{\Uf^{-}_{1}}$ intertwines the $\,^{z}U_{0}$ action. The case of $U^{-}_{1}$ is similar.
		
	The map $e_{\Uf_{0}}e_{\Uf^{-}_{1}}$ is nonzero by Proposition \ref{P:two_functors_for_Iwahori_decomp} (4) because $\Hom_{\Tt}(\sigma^{\Uf_0},\sigma^{\Uf_1^-})\neq 0$.
\end{proof}

\begin{Rmk}
    Lemma \ref{VeryPosSupp} and Proposition \ref{HMult1} together imply that $\Hom_{\It\cap\,^{z}\It}(\sigma,\,^{z}\sigma)=\mathbb{C}e_{\Uf_{0}}e_{\Uf^{-}_{1}}$.
\end{Rmk} 

The next two lemmas are based on \cite[(7.9) Theorem]{BK98}. We cannot invoke this result directly because in our case the pair $(\It,\sigma)$ does not satisfy hypothesis (a) of this theorem, because $(\It,\sigma)$ does not satisfy condition (i) in \cite[(6.1) Definition]{BK98} as mentioned at the beginning of this section.

\begin{Lem}\label{LemSurj}
Let $q:\pi\to\pi_U$ be the canonical surjection. The map $\Hom(\sigma,\pi)\rightarrow \Hom(\sigma^{U_{0}},\pi_{U})$ defined by $f\mapsto q\circ f|_{\sigma^{U_{0}}}$ is a homomorphism of $(\It\cap \widebar{B},\widebar{B})$-bimodules. The map induced by restriction $\Psi:\Hom_{\It}(\sigma,\pi)\rightarrow \Hom_{\Tt_{0}}(\sigma^{U_{0}},\pi_{U})$ is a surjective linear map.
\end{Lem}

\begin{proof}
	Let $f\in \Hom_{\Tt_{0}}(\sigma^{U_{0}},\pi_{U})$. Note that by definition $f\circ e_{{\Uf}_{0}}=f$. When convenient we use this to identify $f$ with $f\circ e_{{\Uf}_{0}}\in \Hom_{\Tt_{0}}(\sigma,\pi_{U})$.
	
	Now since the category of $\Tt_{0}$-modules is semisimple we can fix a $\Tt_{0}$-module splitting $s:\pi_{U}\rightarrow \pi$ of the $\Tt_{0}$-module surjection $q:\pi\rightarrow \pi_{U}$. Define 
    \[
    F:=s\circ f\circ e_{{\Uf}_{0}}\in \Hom_{\Tt_{0}}(\sigma,\pi),
    \]
    so that $q\circ F=f\circ e_{{\Uf}_{0}}$.
	
	Since $\sigma$ is finite dimensional, there exists $j\in \mathbb{Z}_{\geq 1}$ such that $\operatorname{Im}(F)\subseteq \pi^{U^{-}_{j}}$. Let $z\in Z(\Tt)$ be such that $z^{-1}U^{-}_{1}z\subseteq U^{-}_{j}\cap U^{-}_{2e+1}$. Define
    \[
    F_{1}:=\int_{\It}\pi(x)\pi(z)F\circ \sigma(x^{-1})dx\in\Hom_{\It}(\sigma,\pi).
    \]
    Let us compare $F_1$ to $f$. Let $w\in \sigma$. Then by the Iwahori factorization, we have
	\begin{align*}
		q\circ F_1(e_{{\Uf}_{0}}w)=&q\circ \int_{\It} \pi(x)\pi(z)F\circ \sigma(x^{-1}) e_{{\Uf}_{0}}w\,dx\\
		=&q\circ \int_{U_{0}\Tt_{0}U^{-}_{1}} \pi(utv)\pi(z)F\circ \sigma((utv)^{-1}) e_{{\Uf}_{0}}w \,du\,dt\,dv\\
		=&q\circ \int_{U_{0}\Tt_{0}U^{-}_{1}} \pi(utvz)F\circ \sigma((utv)^{-1}) e_{{\Uf}_{0}}w \,du\,dt\,dv.
	\end{align*}
	Using the idempotent $e_{{\Uf}_{0}}$ and the fact that $F$ is a $\Tt_{0}$-module homomorphism we have 
	\begin{align*}
		&q\circ \int_{U_{0}\Tt_{0}U^{-}_{1}} \pi(utvz)F\circ \sigma((utv)^{-1}) e_{U_{0}}w \,du\, dt\, dv\\
        \equiv&q\circ \int_{\Tt_{0}U^{-}_{1}} \pi(tvz)F\circ \sigma((tv)^{-1}) e_{{\Uf}_{0}}w \,dt\,dv\\
		\equiv&q\circ \int_{U^{-}_{1}} \pi(vz)F\circ \sigma(v^{-1}) e_{{\Uf}_{0}}w \,dv,
	\end{align*}
    where by $\equiv$ we mean equality up to a positive scalar. (The discrepancy comes from the choice of measures. We will use the same convention in what follow.)
	Now because $\operatorname{Im}(F)\subseteq \pi^{U^{-}_{j}}$ and  $z^{-1}U^{-}_{1}z\subseteq U^{-}_{j}$, we have
	\begin{align*}
		q\circ \int_{U^{-}_{1}} \pi(vz)F\circ \sigma(v^{-1}) e_{{\Uf}_{0}}w \,dv
        \equiv&q\circ \int_{U^{-}_{1}} \pi(z)\pi(z^{-1}vz)F\circ \sigma(v^{-1}) e_{{\Uf}_{0}}w \,dv\\
		\equiv&q\circ \pi(z)F\circ e_{{\Uf}^{-}_{1}}e_{{\Uf}_{0}}w\\
		\equiv&\pi_{U}(z)f \circ e_{{\Uf}_{0}}e_{{\Uf}^{-}_{1}}e_{{\Uf}_{0}}w,
	\end{align*}
    where the last equality is by the definition of $F$. Finally we can apply Lemma \ref{L3to1} to get 
	\begin{equation*}
		\pi_{U}(z)f \circ e_{{\Uf}_{0}}e_{{\Uf}^{-}_{1}}e_{{\Uf}_{0}}w\equiv\pi_{U}(z)f \circ e_{{\Uf}_{0}}w.
	\end{equation*}
	Thus we have shown that
	\begin{equation}\label{SurjEq1}
		q\circ F_1(e_{{\Uf}_{0}}w)\equiv\pi_{U}(z)f \circ e_{{\Uf}_{0}}w.
	\end{equation}
	
	Recall that for $T\in \Hcal$ and $A\in \Hom_{\It}(\sigma,\pi)$, we define $A*T\in\Hom_{\It}(\sigma,\pi)$ by
    \[
    (A*T)(w):=\int_{\Gt}\pi(h^{-1})A(T(h)w)\,dh.
    \]
    Now suppose that $T\in \Hcal_{z^{-1}}$ is such that $T(z^{-1})=e_{{\Uf}_{0}}e_{{\Uf}^{-}_{1}}$. Such an element exists by Lemma \ref{VeryPosSupp} because $z^{-1}U^{-}_{1}z\subseteq U^{-}_{2e+1}$. Then by definition 
	\begin{align*}
		q[A*T(w)] =&q[ \int_{\Gt}\pi(h^{-1})A(T(h)w)\,dh]\\
		=& q[\int_{\It z^{-1}\It}\pi(h^{-1})A(T(h)w)\,dh]\\
		=& q[\sum_{\It\backslash \It z^{-1}\It}\pi(h^{-1})A(T(h)w)].
	\end{align*}
	Now because $z^{-1}U^{-}_{1}z\subseteq U^{-}_{j}\cap U^{-}_{2e+1}$,
	\begin{align*}
		q[\sum_{\It\backslash \It z^{-1}\It}\pi(h^{-1})A(T(h)w)] =& q[\sum_{zU_{0}z^{-1}\backslash U_{0}}\pi(u^{-1}z)A(T(z^{-1})\sigma(u)w)]\\
		=& q[\sum_{zU_{0}z^{-1}\backslash U_{0}}\pi(z)A(e_{{\Uf}_{0}}e_{{\Uf}^{-}_{1}}\sigma(u)w)]\\
		\equiv&\pi_{U}(z)q[A(e_{{\Uf}_{0}}e_{{\Uf}^{-}_{1}}e_{{\Uf}_{0}}w)].
	\end{align*}
	By Lemma \ref{L3to1},
	\begin{equation*}
	\pi_{U}(z)q[A(e_{{\Uf}_{0}}e_{{\Uf}^{-}_{1}}e_{{\Uf}_{0}}w)]\equiv\pi_{U}(z)q[A(e_{{\Uf}_{0}}w)].
	\end{equation*}
	Thus we have shown
	\begin{equation}\label{SurjEq2}
		q[A*T(w)] \equiv\pi_{U}(z)q[A(e_{{\Uf}_{0}}w)].
	\end{equation}
	
	Now we set $A=F_1*T^{-1}$ and combine equations \eqref{SurjEq1} and \eqref{SurjEq2} to get 
    \begin{equation*}
		\pi_{U}(z)f \circ e_{{\Uf}_{0}}w\stackrel{\eqref{SurjEq1}}{\equiv}q\circ F_1(e_{{\Uf}_{0}}w)\equiv q\circ F_1*T^{-1}*T(e_{{\Uf}_{0}}w)\stackrel{\eqref{SurjEq2}}{\equiv}\pi_{U}(z)q[[F_1*T^{-1}](e_{{\Uf}_{0}}w)].
	\end{equation*}
	
	Thus 
	\begin{equation*}
		q\circ [F_1*T^{-1}]\circ e_{{\Uf}_{0}}\equiv f\circ e_{{\Uf}_{0}}
	\end{equation*}
	and the result follows.
\end{proof}

\begin{Lem}\label{LemInj}
	The map $\Hom(\sigma,\pi)\rightarrow \Hom(\sigma^{U_{0}},\pi_{U})$ defined by $f\mapsto q\circ f|_{\sigma^{U_{0}}}$ is a homomorphism of $(\It\cap \widebar{B},\widebar{B})$-bimodules. The map induced by restriction $\Psi:\Hom_{\It}(\sigma,\pi)\rightarrow \Hom_{\Tt_{0}}(\sigma^{U_{0}},\pi_{U})$ is an injective linear map.
\end{Lem}

\begin{proof}
	Suppose that $F\in \Hom_{\It}(\sigma,\pi)$ is such that $q\circ F\circ e_{{\Uf}_{0}}=0$. Since $F(\sigma^{U_{0}})$ is finite dimensional, there exists $j\in\mathbb{Z}$ such that for all $w\in \sigma^{U_{0}}$ we have $\int_{U_{j^{\prime}}}\pi(u) F(w)du=0$ for any $j^{\prime}\leq j$.
	
	Let $z\in Z(\Tt)$ be such that $zU_{j}z^{-1}\subseteq U_{0}\cap U_{j+2e}$. Let $T\in\Hcal_{z^{-1}}$ be such that $T(z^{-1})=e_{{\Uf}_{0}}e_{{\Uf}^{-}_{1}}$, which exists by Lemma \ref{VeryPosSupp}. Then since $zU_{j}z^{-1}\subseteq U_{0}$ for $w\in \sigma^{U_{0}}$, we have 
	\begin{align*}
		F*T(w)=& \int_{\It z^{-1}\It}\pi(h^{-1})F(T(h)w)\,dh\\
		=&\int_{U_{0}}\pi(u^{-1}z)F(T(z^{-1})\sigma(u)w)\,du\\
		=&\int_{U_{0}}\pi(u^{-1}z)F(e_{{\Uf}_{0}}e_{{\Uf}^{-}_{1}}\sigma(u)w)\,du\\
		=&\pi(z)\int_{U_{0}}\pi(z^{-1}u^{-1}z)F(e_{{\Uf}_{0}}e_{{\Uf}^{-}_{1}}\sigma(u)w)\,du\\
		=&0,
	\end{align*}
        where the last equality follows because $zU_{j}z^{-1}\subseteq U_{0}$.
    
	Thus $F*T$ is zero on $\sigma^{U_{0}}$. Since $F*T$ is an $\It$-module homomorphism and $\sigma^{U_{0}}$ generates $\sigma$ as an $\It$-module, it follows that $F*T=0$. Since $T$ is invertible we see that $F=0$.
\end{proof}

Now, let us recall from \S \ref{TtReps} that we defined the representation
\[
\tau_\eta:=\eta\otimes\tau
\]
of $Z(\Tt)\Tt_0$, where $\eta$ is a genuine $W$-invariant character and $\tau$ is a pseudo-spherical representation of $\Tt_0$, and then for each unramified character $\chi:T\to\C^\times$ we defined the representation 
\[
i_{\eta,\tau}(\chi):=\chi\otimes\Ind_{Z(\Tt)\Tt_0}^{\Tt}\tau_{\eta}
\]
of $\Tt$, which is irreducible. We then have

\begin{Prop}\label{PSeriesConstits}
    Let $\pi$ be a smooth irreducible representation of $\Gt$. Then the $\sigma$-isotypic subspace $\pi^{\sigma}$ is nonzero if and only if $\Hom_{\Gt}(\pi,\Ind_{\Bt}^{\Gt}i_{\eta,\tau}(\chi))\neq 0$ for some unramified character $\chi:T\rightarrow \mathbb{C}^{\times}$. 
\end{Prop}

\begin{proof}
    We begin with a few preliminary remarks. The $\Tt$-representation $\pi_{U}$ is finitely generated, because $\pi$ is irreducible, and thus $\pi_{U}$ has an irreducible quotient. By Frobenius reciprocity it is equivalent to proving the following. The $\sigma$-isotypic subspace $\pi^{\sigma}$ is nonzero if and only if $\Hom_{\Tt}(\pi_{U},i_{\eta,\tau}(\chi))\neq 0$ for some unramified character $\chi:T\rightarrow \mathbb{C}^{\times}$. Now we start the proof.
    
    Suppose that $\pi^{\sigma}\neq 0$. This implies that $\Hom_{\It}(\sigma,\pi)\neq 0$. Thus by Lemma \ref{LemInj} it follows that $\Hom_{\Tt_{0}}(\sigma^{U_{0}},\pi_{U})\neq 0$. Since $\sigma^{U_{0}}\cong\tau$ as $\Tt_{0}$-representations we see that $\pi_{U}^{\tau}\neq0$. 

    Since $\tau$ is a $\Tt$-type and $\pi_{U}^{\tau}\neq 0$ it follows that $\pi_{U}$ has an irreducible quotient with a nonzero $\tau$-isotypic subspace. Thus by Proposition \ref{TTypes}, $\Hom_{\Tt}(\pi_{U},i_{\eta,\tau}(\chi))\neq 0$ for some unramified character $\chi:T\rightarrow \mathbb{C}^{\times}$.
    
    Now we prove the converse. Suppose that $\Hom_{\Tt}(\pi_{U},i_{\eta,\tau}(\chi))\neq 0$ for some unramified character $\chi:T\rightarrow \mathbb{C}^{\times}$. Since $i_{\eta,\tau}(\chi)$ is irreducible, there is a surjective $\Tt$-homomorphism $\pi_{U}\twoheadrightarrow i_{\eta,\tau}(\chi)$. Thus $\pi_{U}^{\tau}\neq 0$, because $i_{\eta,\tau}(\chi)^{\tau}\neq 0$. Finally we apply Lemma \ref{LemSurj} to see that $\pi^{\sigma}\neq0$.
\end{proof}

Now we can prove our main theorem. 

\begin{Thm}\label{TypesThm}\label{IType}
Let $\mathcal{M}(\Gt,\sigma)$ be the full subcategory of $\mathcal{M}(\Gt)$ consisting of objects generated by their $\sigma$-isotypic subspace.
	\begin{enumerate}
		\item The category $\mathcal{M}(\Gt,\sigma)$ is closed relative to subquotients in $\mathcal{M}(\Gt)$.\label{SubQuot}
		\item For every irreducible object $(\pi,V)$ in $\mathcal{M}(\Gt,\sigma)$, there is an unramified character $\chi:T\rightarrow \mathbb{C}^{\times}$ such that $(\pi,V)$ is isomorphic to an irreducible subquotient of $\Ind_{\widebar{B}}^{\Gt}(i(\chi))$.\label{BernComp}
	\end{enumerate}
In particular, $(\It, \sigma)$ is a type.
\end{Thm}

\begin{proof} 
\eqref{SubQuot} can be proved as in \cite[Theorem 7.1]{Karasiewicz}, with \cite[Proposition 7.4]{Karasiewicz} being replaced by Proposition \ref{PSeriesConstits}. This argument essentially follows \cite[(8.3) Theorem]{BK98}.

\eqref{BernComp} is already part of Proposition \ref{PSeriesConstits}.
\end{proof}

As a corollary, we have the following local Shimura correspondence.

\begin{Cor}\label{C:local_Shimura_Correspondence}
    Assume the Cartan type of $G$ is $A_r, D_{2r+1}, E_6$ or $E_7$. Then the category $\mathcal{M}(\Gt,\sigma)$ is equivalent to the Bernstein block of the Iwahori spherical representations of $G'=G\slash Z_2$.
\end{Cor}
\begin{proof}
    This is immediate from the above theorem and Corollary \ref{HeckeShimCorr}. Also see the proof of \cite[Theorem 8.1]{Karasiewicz}.
\end{proof}

\appendix


\section{Hilbert symbol over $2$-adic fields}\label{S:Appendix_Hilbert_symbol}


In this appendix, we establish various properties of the quadratic Hilbert symbol for our $2$-adic field $F$. Everything in this appendix is likely known, but we include proofs because we have not been able to locate the proofs in the literature.

Since $F$ is $2$-adic, the following lemma is well-known.
\begin{Lem}
The quadratic Hilbert symbol is not trivial on $\Ocal^\times\times\Ocal^\times$.
\end{Lem}

We need to analyze the Hilbert symbol restricted to $\Ocal^\times\times\Ocal^\times$. For this purpose, define
\[
R=\{a\in\Ocal^\times\st \text{$(a, -)$ is trivial on $\Ocal^\times$}\},
\]
which we call the {\it integral radical} of the Hilbert symbol. By the above lemma, $R\subsetneq\Ocal^\times$. Clearly we have $\Ocal^{\times2}\subseteq R$.

Let us mention the following ``one-dimensionality" of the integral radical.
\begin{Lem}\label{DimRad}
Let us view $F^\times\slash F^{\times 2}$ as an $\F_2$-vector space. Then
\[
\dim_{\mathbb{F}_{2}}(RF^{\times 2}/F^{\times 2})=1,
\]
where we view $RF^{\times 2}\slash F^{\times 2}$ as a subspace of $F^\times\slash F^{\times 2}$.
\end{Lem}
\begin{proof}
    This follows from \cite[Lemma 2.4]{KT26}.
\end{proof}


For each $k\in\Z_{\geq 1}$, we set
\[
U_k=1+\varpi^k\Ocal.
\]
We often use the well-known
\begin{Lem}\label{L:square_root_exits_in_redidue_field}
Every element in $\Ocal\slash\varpi\Ocal$ has a square root, and hence $\Ocal^\times=\Ocal^{\times 2}U_1$.
\end{Lem}
\begin{proof}
The Frobenius map (the squaring map) on $\Ocal\slash\varpi\Ocal$ is an automorphism, and hence every element in $\Ocal\slash\varpi\Ocal$ has a square root. Then every element in $\Ocal^\times$ is written as $a^2+b\varpi$ for some nonzero $a\in\Ocal\slash\varpi\Ocal$ and some $b\in \Ocal\slash\varpi\Ocal$. We can then write $a^2+b\varpi=a^2(1+a^{-2}b\varpi)$.
\end{proof}

We then have
\begin{Lem}\label{L:2-adic_arithmetic}
\quad
\begin{enumerate}[(a)]
\item $U_{2e+1}\subseteq \Ocal^{\times 2}$.
\item\label{Val2eSqrs} For each $u\in\Ocal$, we have $1+4u\in\Ocal^{\times 2}$ if and only if there exists $v\in \Ocal$ such that $v+v^2=u$ in $\Ocal\slash\varpi\Ocal$.
\item $U_{2e}\nsubseteq\Ocal^{\times 2}$.
\item $U_{2e}\slash (U_{2e}\cap\Ocal^{\times 2})\cong\Z\slash 2\Z$.
\end{enumerate}
\end{Lem}
\begin{proof}
\quad
\begin{enumerate}[(a)]
\item This follows from \cite[(5.5) Proposition, page 137]{Neukirch} because the exponential map $\varpi^{e+1}\Ocal\to U_{e+1}$ is an isomorphism.

\item Let $u\in\Ocal$. Assume $1+4u\in\Ocal^{\times 2}$, so $1+4u=w^2$ for some $w\in\Ocal^\times$. One can then write $w=1+\varpi^{\ell}v$ for some $\ell\in\Z_{\geq 1}$ and $v\in \Ocal^\times$. So, we have $1+4u=(1+\varpi^{\ell} v)^2$. A short calculation shows that $\ell\geq e$, so $\varpi^{\ell}v\in 2\Ocal$. Hence by re-choosing $v\in \Ocal$ (if necessary), we have $1+4u=(1+2v)^2$, which implies $v+v^2=u$. Conversely, assume there exists $v\in \Ocal$ such that $v+v^2=u$ in $\Ocal\slash\varpi\Ocal$, namely $u=v+v^2+\varpi v'$ for some $v'\in\Ocal$. We then have $1+4u=(1+2v)^2+4\varpi v'$. But
\[
(1+2v)^2+4\varpi v'=(1+2v)^2(1+\frac{4\varpi v'}{(1+2v)^2})\qand 1+\frac{4\varpi v'}{(1+2v)^2}\in U_{2e+1}.
\]
By part (a), we know $U_{2e+1}\subseteq \Ocal^{\times 2}$, which implies $1+4u\in\Ocal^{\times 2}$.

\item Consider the group homomorphism $\sigma:\F_{q}\to\F_{q}$ given by $x\mapsto x+x^2$. The kernel is exactly $\mathbb{F}_{2}$. Hence there exists $u\in\Ocal$ such that $u=x+x^2$ has no solution. For such $u$, we know $1+4u\notin\Ocal^{\times 2}$ by part (b).

\item Recall that there is a group isomorphism $\Ocal/\varpi\Ocal \rightarrow U_{2e}/U_{2e+1}$ defined by $u+\varpi\Ocal \mapsto (1+4u)U_{2e+1}$. Under this map we see that the squares in $U_{2e}$ correspond to the elements of $\Ocal/\varpi\Ocal\cong \mathbb{F}_{q}$ in the image of the map $\sigma$ of part (c). Since the image of $\sigma$ has index $2$ in $\mathbb{F}_{q}$ the result follows.
\end{enumerate}
\end{proof}

In what follows, we will show $R=U_{2e}\Ocal^{\times 2}$. Let us start with one inclusion.
\begin{Lem}
We have $U_{2e}\subseteq R$.
\end{Lem}
\begin{proof}
Let $1-4u\in U_{2e}$ for $u\in\Ocal$. If $u\notin\Ocal^\times$, then $1-4u\in U_{2e+1}$ and we already know $U_{2e+1}\subseteq\Ocal^{\times 2}\subseteq R$. Hence we assume $u\in\Ocal^\times$. We have to show $(1-4u, a)=1$ for all $a\in\Ocal^\times$. Let $v\in\Ocal$ be such that $1-4v\in U_{2e}\cap\Ocal^{\times 2}$. Then for all $a\in F^\times$ we have
\begin{align*}
(1-4u, a)&=((1-4u)(1-4v), a)=(1-4(u+v)+16uv,a)\\
&=((1-4(u+v))(1+\frac{16uv}{1-4(u+v)}),a)=(1-4(u+v),a),
\end{align*}
where we used $1+\frac{16uv}{1-4(u+v)}\in U_{2e+1}\subseteq \Ocal^{\times 2}$ by Lemma \ref{L:2-adic_arithmetic} (a). Also $(1-4(u+v), u+v)=(1-4(u+v), 4(u+v))=1$, where we used $(1-x, x)=1$ for all $x\in F^\times$. Hence by setting $a=u+v$, we obtain
\begin{equation}\label{E:(1-4u, u+v)=1}
(1-4u, u+v)=(1-4(u+v),u+v)=1.
\end{equation}

Now, for $a\in\Ocal^\times$, one can write $au^{-1}=r^2(1+\varpi s)$ for some $r\in\Ocal^\times$ and $s\in\Ocal$ by Lemma \ref{L:square_root_exits_in_redidue_field}. Then
\[
(1-4u, a)=(1-4u, ur^2(1+\varpi s))=(1-4u, u+\varpi(us)).
\]
But $1-4\varpi(us)\in U_{2e+1}\subseteq\Ocal^{\times 2}$, and hence by taking $v=\varpi(us)$ in \eqref{E:(1-4u, u+v)=1}, we obtain $(1-4u, a)=1$.
\end{proof}

We can now prove
\begin{Prop}
We have $U_{2e}\Ocal^{\times 2}= R$, and hence $R\slash\Ocal^{\times 2}=\Z\slash 2\Z$.
\end{Prop}
\begin{proof}
We already know $\Ocal^{\times 2}\subseteq U_{2e}\Ocal^{\times 2}\subseteq R$ and $[U_{2e}\Ocal^{\times 2}:\Ocal^{\times 2}]=2$. By Lemma \ref{DimRad} we know that $[R:\Ocal^{\times 2}]=2$. Thus $R=U_{2e}\Ocal^{\times 2}$.
\end{proof}

The above discussion can be summarized as follows:
\[
\begin{tikzcd}
&\Ocal^{\times 2}\ar[r, phantom, "\subseteq"]&R=U_{2e}\Ocal^{\times 2}&\\
U_{2e+1}\ar[r, phantom, "\subseteq"]&U_{2e}\cap\Ocal^{\times 2}\ar[r, phantom, "\subseteq"]\ar[u, phantom, "{\rotatebox{90}{$\subseteq$}}"]&U_{2e}\ar[r, phantom, "\subseteq"]\ar[u, phantom, "{\rotatebox{90}{$\subseteq$}}"]&U_{2e-1}
\rlap{\, ,}
\end{tikzcd}
\]
where
\[
U_{2e}\slash (U_{2e}\cap\Ocal^{\times 2})=\Z\slash 2\Z.
\]


Let us also mention
\begin{Cor}\label{C:dimension_integral_radical}
\[
[\Ocal^{\times}:R]=2^{d}.
\]
\end{Cor}
\begin{proof}
It is well-known that
\[
[\Ocal^\times: \Ocal^{\times 2}]=2^{d+1}.
\]
(See, for example, \cite[(5.8) Proposition, p.142]{Neukirch}.) Since $R\slash\Ocal^{\times 2}=\Z\slash 2\Z$, the corollary follows.
\end{proof}

\section{Proof of Proposition \ref{Prop:PS_Multi_one_Intro}}\label{S:Appendix_PS_Mult_One}

\begin{Prop}
    Assume that the $\It$-representation $\sigma$ extends to $\Kt$. Let $\chi:T\rightarrow \C^{\times}$ be an unramified character. The multiplicity space $\mathrm{Hom}_{\Kt}(\sigma,I(\chi))$ has dimension $1$. Thus $I(\chi)$ contains a unique pseudo-spherical constituent.
\end{Prop}
    
\begin{proof}
    Using the Iwasawa decomposition and Mackey theory we get
    \begin{align*}
        I(\chi)|_{\Kt} =&  \mathrm{Ind}_{\Bt}^{\Gt}(\chi\otimes \tau_{\chi_{dist}})|_{\Kt}\\
        \cong& \mathrm{Ind}_{\Tt_{0}U_{0}}^{\Kt}( \oplus_{y\in Y/\Yt}\varphi_{y}\otimes \tau),
    \end{align*}
    where $\Yt\subseteq Y$ is the sublattice containing the elements $y\in Y$ such that $D(y,z)+D(z,y)\in n\Z$ for all $z\in Y$, and $\varphi_{y}$ is a certain character of $T$ defined in \cite[Section 5.1]{KT26}. For the rest of this proof all we need to know about $\varphi_{y}$ is that $\tau\cong \varphi_{y}\otimes \tau$ if and only if $y\in \Yt$. This is proved in \cite[Lemma 5.2]{KT26}.

   Now by Frobenius reciprocity we have
   \begin{equation*}
       \mathrm{Hom}_{\Kt}(\sigma,\mathrm{Ind}_{\Tt_{0}U_{0}}^{\Kt}( \oplus_{y\in Y/\Yt}\varphi_{y}\otimes \tau))
    \cong  \mathrm{Hom}_{\Tt_{0}}(\sigma^{U_{0}}, \oplus_{y\in Y/\Yt}\varphi_{y}\otimes \tau).
   \end{equation*} 
   As $\Tt_{0}$-representations $\sigma^{U_{0}}\cong \tau$. This is a general fact about parahoric induction, which we review in Proposition \ref{P:two_functors_for_Iwahori_decomp}. This, along with the fact about $\varphi_{y}$ mentioned above and Schur's lemma, we have
\begin{align*}
    \mathrm{Hom}_{\Tt_{0}}(\sigma^{U_{0}}, \oplus_{y\in Y/\Yt}\varphi_{y}\otimes \tau)
    \cong & \mathrm{Hom}_{\Tt_{0}}(\tau, \oplus_{y\in Y/\Yt}\varphi_{y}\otimes \tau)\\
    \cong & \mathrm{Hom}_{\Tt_{0}}(\tau, \tau)\\
    \cong & \mathbb{C}.
\end{align*}
\end{proof}

\bibliographystyle{amsalpha}
\bibliography{Hecke_algebra2_bio}

\end{document}